\documentclass{amsart}

\usepackage{amsmath}
\usepackage{amsfonts}
\usepackage{amsopn}
\usepackage{amssymb}
\usepackage{graphicx}
\usepackage{hyphenat}

\makeatletter
\newcommand*\if@single[3]{%
  \setbox0\hbox{${\mathaccent"0362{#1}}^H$}%
  \setbox2\hbox{${\mathaccent"0362{\kern0pt#1}}^H$}%
  \ifdim\ht0=\ht2 #3\else #2\fi
  }
\newcommand*\rel@kern[1]{\kern#1\dimexpr\macc@kerna}
\newcommand*\widebar[1]{\@ifnextchar^{{\wide@bar{#1}{0}}}{\wide@bar{#1}{1}}}
\newcommand*\wide@bar[2]{\if@single{#1}{\wide@bar@{#1}{#2}{1}}{\wide@bar@{#1}{#2}{2}}}
\newcommand*\wide@bar@[3]{%
  \begingroup
  \def\mathaccent##1##2{%
    \if#32 \let\macc@nucleus\first@char \fi
    \setbox\z@\hbox{$\macc@style{\macc@nucleus}_{}$}%
    \setbox\tw@\hbox{$\macc@style{\macc@nucleus}{}_{}$}%
    \dimen@\wd\tw@
    \advance\dimen@-\wd\z@
    \divide\dimen@ 3
    \@tempdima\wd\tw@
    \advance\@tempdima-\scriptspace
    \divide\@tempdima 10
    \advance\dimen@-\@tempdima
    \ifdim\dimen@>\z@ \dimen@0pt\fi
    \rel@kern{0.6}\kern-\dimen@
    \if#31
      \overline{\rel@kern{-0.6}\kern\dimen@\macc@nucleus\rel@kern{0.4}\kern\dimen@}%
      \advance\dimen@0.4\dimexpr\macc@kerna
      \let\final@kern#2%
      \ifdim\dimen@<\z@ \let\final@kern1\fi
      \if\final@kern1 \kern-\dimen@\fi
    \else
      \overline{\rel@kern{-0.6}\kern\dimen@#1}%
    \fi
  }%
  \macc@depth\@ne
  \let\math@bgroup\@empty \let\math@egroup\macc@set@skewchar
  \mathsurround\z@ \frozen@everymath{\mathgroup\macc@group\relax}%
  \macc@set@skewchar\relax
  \let\mathaccentV\macc@nested@a
  \if#31
    \macc@nested@a\relax111{#1}%
  \else
    \def\gobble@till@marker##1\endmarker{}%
    \futurelet\first@char\gobble@till@marker#1\endmarker
    \ifcat\noexpand\first@char A\else
      \def\first@char{}%
    \fi
    \macc@nested@a\relax111{\first@char}%
  \fi
  \endgroup
}
\makeatother

\newcommand{\Bsq}{\ensuremath{B_\square}}
\newcommand{\K}{\ensuremath{\mathbb{K}}}
\newcommand{\N}{\ensuremath{\mathbb{N}}}
\newcommand{\R}{\ensuremath{\mathbb{R}}}
\newcommand{\Z}{\ensuremath{\mathbb{Z}}}
\newcommand{\bI}{\ensuremath{\mathbf{I}}}

\newcommand{\cA}{\ensuremath{\mathcal{A}}}
\newcommand{\cB}{\ensuremath{\mathcal{B}}}
\newcommand{\cF}{\ensuremath{\mathcal{F}}}
\newcommand{\cG}{\ensuremath{\mathcal{G}}}
\newcommand{\cH}{\ensuremath{\mathcal{H}}}
\newcommand{\cR}{\ensuremath{\mathcal{R}}}
\newcommand{\cT}{\ensuremath{\mathcal{T}}}
\newcommand{\CLip}{\ensuremath{C^\mathrm{Lip}}}
\newcommand{\Ccell}{\ensuremath{C}}
\newcommand{\Haus}{\ensuremath{\mathcal{H}}}

\newcommand{\from}{\colon}
\newcommand{\extend}[1]{\widebar{#1}}
\newcommand{\itref}[1]{{property \ref{#1}}}
\newcommand{\fnorm}[1]{\cF(#1)}
\renewcommand{\setminus}{\smallsetminus}

\DeclareMathOperator{\Core}{core}
\DeclareMathOperator{\cov}{cov}

\DeclareMathOperator{\Clos}{closure}
\DeclareMathOperator{\HC}{HC}
\DeclareMathOperator{\NO}{NO}
\DeclareMathOperator{\diam}{diam}
\DeclareMathOperator{\Lip}{Lip}

\DeclareMathOperator{\mass}{mass}
\DeclareMathOperator{\vol}{vol}
\DeclareMathOperator{\area}{area}
\DeclareMathOperator{\size}{size}
\DeclareMathOperator{\FV}{FV}
\DeclareMathOperator{\supp}{supp}
\DeclareMathOperator{\inter}{int}
\DeclareMathOperator{\nbhd}{nbhd}

\newtheorem{thm}{Theorem}[section]

\newtheorem{lemma}[thm]{Lemma}
\newtheorem{prop}[thm]{Proposition}
\newtheorem{cor}[thm]{Corollary}
\theoremstyle{remark}
\newtheorem{defn}[thm]{Definition}

\newtheorem*{ack}{Acknowledgments}

\title{Quantitative nonorientability of embedded cycles}
\date{\today}
\author{Robert Young}
\address{Courant Institute of Mathematical Sciences\\
  New York University\\
  251 Mercer St.\\
  New York, NY  10012\\
  USA}
\email{ryoung@cims.nyu.edu}

\begin{document}

\begin{abstract}
  We introduce an invariant linked to some foundational questions in
  geometric measure theory and provide bounds on this invariant by
  decomposing an arbitrary cycle into uniformly rectifiable pieces.
  Our invariant measures the difficulty of cutting a nonorientable
  closed manifold or mod-2 cycle in $\R^n$ into orientable pieces, and
  we use it to answer some simple but long-open questions on filling
  volumes and mod-$\nu$ currents.
\end{abstract}

\maketitle

\tableofcontents

\section{Introduction}
If $T$ is a Lipschitz curve in $\R^N$, there is a minimal surface $U$
whose boundary is $T$.  If we trace $T$ twice to obtain a curve $2T$,
there is a minimal surface $U'$ whose boundary is $2T$.  At first
glance, one might guess that $U'=2U$.  This is easy to prove when
$N=2$ and is a theorem of Federer \cite{FedererRealVariational} when
$N=3$, but remarkably, it is false when $N\ge 4$!  L. C. Young
constructed an example of a curve $T\from S^1\to \R^4$ that lies on an
embedded Klein bottle and a chain $U$ such that $U$ is a minimal
filling of $T$, but $2U$ is not a minimal filling of $2T$.  In fact,
$\mass U'$ is only about $1.5 \mass U$ \cite{LCYoung}.

A version of Young's example is shown in
Figure~\ref{fig:youngExample}.  Consider a Klein bottle $K$ embedded
in $\R^4$ and draw $2k+1$ equally-spaced rings on $K$.  Since these
rings are drawn on a Klein bottle, we can orient them so that adjacent
rings have ``opposite'' orientations.  Let $T$ be the sum of these
rings.

On one hand, we can fill $2T$ with a chain supported on $K$.  Since
the rings have alternating orientations, we can fill each pair of
adjacent rings with a thin cylindrical band.  The curves in $T$ cut
$K$ into $2k+1$ bands, and if we give these bands alternating
orientations, their boundary is $2T$ (right side of
Fig.~\ref{fig:youngExample}).  When $k$ is large, this is nearly
optimal and has mass roughly $\area{K}$.

\begin{figure}[b]
  \begin{tabular}{ccc}
    \includegraphics[width=.3\textwidth]{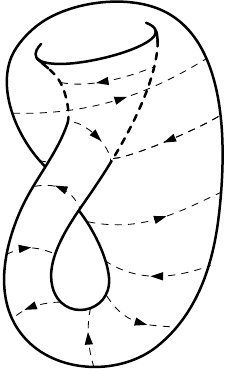} & 
    \includegraphics[width=.3\textwidth]{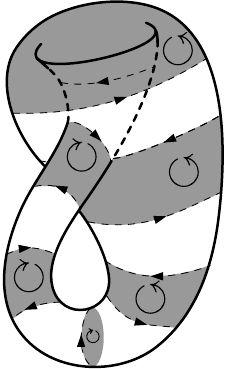} &
    \includegraphics[width=.3\textwidth]{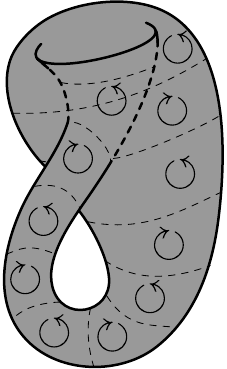}\\
    $T$ & A filling of $T$ & A filling of $2T$
  \end{tabular}
  \caption{\label{fig:youngExample}
    Fillings of a 1-cycle on a Klein bottle.  The 1-cycle $T$
    consists of $2k+1$ loops with alternating orientations.  In the
    middle, we fill $T$ with $k$ cylindrical bands and a disc, and on
    the right, we fill $2T$ with $2k+1$ cylindrical bands with alternating
    orientations.}
\end{figure}

On the other hand, we cannot use the same technique to fill $T$.
Since there's an odd number of rings in $T$, we can fill all but one
of the rings using $k$ bands, but we need to fill the last ring with a
disc (middle of Fig.~\ref{fig:youngExample}).  When $k$ is large, a
filling like this is nearly optimal and has area roughly $(\area{K})/2$
plus the area of the extra disc --- well over half the area of a minimal
filling of $2T$.

Questions in geometric measure theory related to this example and
examples with different multipliers found by Morgan
\cite{MorganMultiples} and White \cite{WhiteMultiples} have been open
almost since Federer and Fleming's first papers developing normal and
integral currents.  Because of these examples, the flat distance
\begin{equation}\label{eq:flatNorm}
  \fnorm{A}=\inf\{\mass R+\mass S\mid A=R+\partial S, R\in \cR_d(\R^N), S\in\cR_{d+1}(\R^N)\}.
\end{equation}
is not a norm; if $T$ is as above, then we may have
$\fnorm{\nu T} < \nu\fnorm{T}$.  Consequently, several basic questions
have remained unanswered, including:
\begin{enumerate}
\item If $\nu>0$ is a positive integer and $\cF_d(\R^N)$ is the space
  of integral flat $d$--chains in $\R^N$, is the multiply-by-$\nu$ map
  $f\from \cF_d(\R^N)\to \cF_d(\R^N)$, $f(T)=\nu T$ an embedding?
\item Is the set of flat chains modulo $\nu$ a quotient of the
  integral flat chains?
\item We can define a real flat norm by replacing the rectifiable
  currents in \eqref{eq:flatNorm} by normal currents.  How is the real
  flat norm related to the flat distance?
\end{enumerate}
In this paper, we will relate the first two of these questions to the
geometry of nonorientable cycles in $\R^N$ and answer both of them
positively (Corollaries~\ref{cor:multiplesEmbedding} and
\ref{cor:integralCurrentsLift}).

Specifically, we will define the following invariant.  If $A$ is a
mod--$\nu$ cycle in $\R^N$ (a Lipschitz cycle or integral current
modulo $\nu$) and $R$ is a $\Z$--cycle (Lipschitz cycle or integral
current) such that $A\equiv R\pmod \nu$, we say that $R$ is a
\emph{pseudo-orientation} of $A$.  Let the \emph{nonorientability} of
$A$ be
\[\NO(A)=\inf \{\mass R\mid \text{$R$ is a pseudo-orientation of
  $A$}\}.\]

Any smooth submanifold of $\R^N$ has a pseudo-orientation.  For
example, suppose that $M$ is a nonorientable genus-$g$ surface
smoothly embedded in $\R^N$ and that $A=[M]$ is the fundamental class
of $M$.  This is a cycle with $\Z_2$ coefficients, but we can lift it
to a cycle with integer coefficients by cutting $M$ into orientable
pieces.  Let $\Gamma$ be a smooth graph embedded in $M$ whose
complement consists of orientable pieces $M_1,\dots, M_n$.  We choose
orientations on the $M_i$ arbitrarily to get fundamental classes
$[M_1],\dots,[M_n]$.  Then $R_0:=\sum [M_i]$ is a 2--chain over $\Z$
and $\supp \partial R_0\subset \Gamma$.  Each edge of $\Gamma$ occurs
in $\partial R_0$ with coefficient 0 or $\pm2$, depending on the
orientations of the neighboring regions in $M$.  Let $R_1$ be a chain
with integer coefficients such that $\partial R_1=\partial R_0/2$ and
define $R=R_0+2R_1$.  Then
$$R\equiv \sum [M_i]\equiv [M] \pmod{2},$$
so $R$ is a pseudo-orientation of $[M]$.  The mass of $R$, however,
could be much larger than the area of $M$, especially if $M$ is very
complicated.  

In this paper, we will show that the nonorientability of a cycle is
bounded by its mass:
\begin{thm}\label{thm:NObound}
  For every $\nu,d,N>0$, there is a $c>0$ such that
  $$\NO(A)\le c\mass A$$
  for every $A\in \bI^\nu_d(\R^N)$ with $\partial A=0$.
\end{thm}
Here, $\bI^\nu_d(\R^N)$ is the set of integral currents modulo $\nu$;
when $\nu>0$, $\nu\in \Z$, the integral currents modulo $\nu$ are a
chain complex that contains the mod-$\nu$ polyhedral chains as a dense
subset.  

Note that it is not clear \emph{a priori} that every integral current
mod $\nu$ with no boundary has a pseudo-orientation.  Federer
\cite[4.2.26]{FedererGMT} asserted that there are integral currents
modulo $2$ that do not lift to integral currents, though his example,
an infinite sum of projective planes with finite total mass, turns out
to have an error (see \cite[2.5]{PaurStokes}).  

The theorem will follow from the following statement about cycles in
the unit grid in $\R^N$
\begin{thm}\label{thm:cellNObound}
  For every $\nu,d,N>0$, there is a $c>0$ such that if
  $A\in \Ccell_d(\tau;\Z_\nu)$ is a mod-$\nu$ cellular cycle in the
  unit grid in $\R^N$, then there is a cellular cycle
  $R\in \Ccell_d(\tau;\Z)$ such that $A\equiv R\pmod{\nu}$ and
  $\mass R\le c\mass A$.  It follows that
  $$\NO(A)\le c\mass A.$$
\end{thm}
(Indeed, though several of our applications will involve integral
currents and flat chains, our use of currents will be restricted to
the proofs of those applications, and the proof of
Theorem~\ref{thm:cellNObound} can be read without previous familiarity
with currents.)

The main difficulty in proving Theorem~\ref{thm:cellNObound} is
dealing with cycles that have complex topology at many scales and many
locations.  For example, consider the following sequence of surfaces:
Let $C_0$ be the 2-dimensional surface of a unit 3-cube embedded in
$\R^4$.  The surface $C_0$ is orientable, but we can make it
nonorientable by gluing crosscaps to its faces.  Let $\Sigma$ be a
crosscap consisting of a union of faces in the unit grid, with
boundary a square of side length 10.  We can partition $C_0$ into $6$
unit squares and construct $C_1$ by replacing each square by a scaled
copy of $\Sigma$.

Then $C_1$ is a surface in the grid of side length
$10^{-1}$, homeomorphic to a connected sum of six projective planes.
Its fundamental class is a mod-2 cycle, and any pseudo-orientation of
$C_1$ must cut through each crosscap, so $\NO([C_1])\sim 1$.

There are no large faces in $C_1$ to replace by crosscaps, but we can
still add nonorientability at smaller scales by replacing smaller
faces in $C_1$ by smaller crosscaps.  Choose $100$ faces of $C_1$ of
side length $10^{-1}$ and replace them by scaled copies of $\Sigma$ to
obtain $C_2$.  Each new crosscap contributes roughly $10^{-2}$ to the
nonorientability, so in total, they contribute roughly $1$.

Proceeding inductively, we replace $100^i$ faces of $C_i$ of side
length $10^{-i}$ to obtain $C_{i+1}$.  A pseudo-orientation of $C_k$
must cut through all of the crosscaps at every scale, so
$$\NO([C_k])\sim \sum_{i=0}^{k-1} 100^i 10^{-2i}=k.$$ 
This is much larger than the area of the surface we started with.  The
only reason that this does not contradict
Theorem~\ref{thm:NObound} is that each added crosscap of diameter
$r$ also increases the area of the surface by roughly $r^2$, so
$$\area(C_k)\sim 1+\sum_{i=0}^{k-1} 100^i 10^{-2i}\sim k+1.$$

One can also imagine more complicated versions of $C_k$ using
different scale factors or replacing squares by more complicated
surfaces.  Theorem~\ref{thm:NObound} implies that in all such
constructions, the extra nonorientability coming from replacing a
square by a surface is bounded by the added area.  Nevertheless, we
conjecture that the ratio $\frac{\NO(A)}{\mass A}$ approaches its
supremum for a sequence of self-similar surfaces like the $C_i$.

A remarkable feature of Theorem~\ref{thm:NObound} is that it gives a
bound that is independent of the topology of $A$; many related bounds
depend on the topology.  If $d=2$ and $A$ is the fundamental class of
a surface $M\subset \R^N$, then bounds on the nonorientability of
$[M]$ are related to bounds on systoles of $M$, which typically
depend on the genus of $M$.  

For example, as noted above, the nonorientability of $M$ is related to
the difficulty of partitioning $M$ into orientable pieces.  By
choosing an orientation on each piece, one can lift such a partition
to a partition of the double cover $\tilde{M}$ into two pieces of
equal area.  Cheeger's inequality implies that there are surfaces
(scalings of arithmetic hyperbolic surfaces) with unit area and genus
$g$ such that any curve or set of curves that cuts the surface into
two equal pieces has length at least $\sqrt{g}$.  Similarly, one way
to obtain a graph $\Gamma$ in $M$ whose complement consists of
orientable pieces is to let $\Gamma$ be a pants decomposition of
$M$.  In a paper with Larry Guth and Hugo Parlier \cite{GPY}, we
showed that every pants decomposition of a ``random'' genus $g$
surface of area 1 has total length at least $g^{2/3-\epsilon}$.

This could suggest that some of the unusual geometric properties
(large systoles, expander-type properties, large pants decompositions,
etc.) that occur in arithmetic hyperbolic surfaces and random surfaces
may not occur in surfaces that embed bilipschitzly (with respect to
the euclidean metric, not in the sense of Nash) in $\R^N$.  It would
be interesting to know if this is the case.

\subsection{Applications}\label{sec:introApplications}

Theorem~\ref{thm:NObound} has several applications in geometric
measure theory and the study of currents.  First,
Theorem~\ref{thm:NObound} provides an answer to a question of
L. C. Young.  Let us define the \emph{filling volume} $\FV(T)$ of a
Lipschitz $d$-cycle $T\in \CLip_d(\R^N)$ with $\partial T=0$ to be the
infimal mass of a Lipschitz $(d+1)$--chain $U\in \CLip_{d+1}(\R^N)$
such that $\partial U=T$.  It follows from Theorem~\ref{thm:NObound}
that:
\begin{cor}\label{cor:YoungsProblem}
  For any $d,N, \nu>0$, there is a $c>0$ such that for any $d$--cycle
  $T\in \CLip_d(\R^N)$,
  $$\FV(T)\le c \FV(\nu T).$$ 
\end{cor}

The behavior of $c$ when $\nu$ is large is an open and interesting
question, because the limit
$\lim_{\nu\to \infty}\frac{\FV(\nu T)}{\nu}$ is equal to the
\emph{real filling volume} $\FV_\R(T)$ of $T$.  The real filling
volume is the infimal mass of a Lipschitz $(d+1)$--chain
$U\in \CLip_{d+1}(\R^N;\R)$ with real coefficients such that
$\partial U=T$.  L. C. Young's example shows that the integral and
real filling volumes of a cycle can be different, but it is unknown
whether the ratio of these filling volumes is bounded.

The theorem also answers some questions about integral currents and
flat chains that have been open since the 1960's.  In particular, the
following corollaries answer a question in 4.2.26 of \cite{FedererGMT}
and part of a cluster of related questions studied by Almgren
\cite{WhiteMathOfAlmgren}.
\begin{cor}\label{cor:multiplesEmbedding}
  Let $d,n,\nu\in \N$.  The multiply-by-$\nu$ map
  $f\from \cF_d(\R^N)\to \cF_d(\R^N)$, $f(T)=\nu T$ is an embedding,
  and the images $\nu \cF_d(\R^N)$ and $\nu \bI_d(\R^N)$ are closed.
\end{cor}
and 
\begin{cor}\label{cor:integralCurrentsLift}
  If $T\in \bI^\nu_d(\R^N)$ is an integral current mod $\nu$, then
  $T\equiv T_\Z \pmod{\nu}$ for some integral current $T_\Z$.
\end{cor}
These corollaries answer a question in 4.2.26 of \cite{FedererGMT} and

Corollary~\ref{cor:integralCurrentsLift} is somewhat subtle because of
the terminology used to describe currents mod $\nu$.  One can define
quotients $\cF_d(\R^N)/\nu\cF_d(\R^N)$ and
$\bI_d(\R^N)/\nu\bI_d(\R^N)$ that have many of the properties of flat
chains and integral currents modulo $\nu$.  But it is not clear
\emph{a priori} that these quotients satisfy completeness and
compactness properties.  For instance, any projective plane, and thus
any finite sum of projective planes, is congruent mod $2$ to an
integral current, but it is unclear whether an infinite sum with
finite total mass is congruent to an integral current.  To avoid this
problem, Federer \cite{FedererGMT} defined the flat chains modulo
$\nu$ as the quotient
$\cF^\nu_d(\R^N)=\cF_d(\R^N)/\overline{\nu\cF_d(\R^N)}$ by the closure
of the multiples of $\nu$ and defined the integral currents modulo
$\nu$ as the set $\bI^\nu_d(\R^N)$ of rectifiable currents mod $\nu$
with rectifiable boundary mod $\nu$.
Corollary~\ref{cor:multiplesEmbedding} and
Corollary~\ref{cor:integralCurrentsLift} imply that these definitions
are the same as the naive definitions.
\begin{cor}\label{cor:modNuQuotient}
  $$\cF^\nu_d(\R^N)=\cF_d(\R^N)/\nu\cF_d(\R^N)$$
  $$\bI^\nu_d(\R^N)=\bI_d(\R^N)/\nu\bI_d(\R^N).$$
\end{cor}

\subsection{Techniques}
In this section, we will sketch the proof of
Theorem~\ref{thm:NObound}.  We will go into more detail in
Section~\ref{sec:expandedSketch}.

As we saw in the example $A=[C_k]$ above, $\NO(A)$ is a sum of
contributions from many different places and scales; the surface $C_k$
consists of many crosscaps, and one large crosscap contributes as much
nonorientability as many small ones.  One can use the Federer-Fleming
deformation theorem to bound the amount of nonorientability that comes
from each scale (see Prop.~\ref{prop:VlogVbound}), but
Theorem~\ref{thm:NObound} requires a bound on the total contribution
from all scales.

We solve this problem by developing new techniques to decompose cycles
in $\R^N$ into topologically and geometrically simple pieces.  In
particular, we devise a way to break down a cycle in $\R^N$ into a sum
of cycles that either lie close to planes or are topologically
bounded.  The decomposition has two stages: first, we decompose cycles
in $\R^N$ into cycles with uniformly rectifiable supports, then we
apply corona decompositions to break those cycles into pieces with
bounded geometry and topology.

Uniformly rectifiable sets were developed by David and Semmes as a
quantitative version of the notion of rectifiable sets.  (See
Section~\ref{sec:URprelims} for definitions and references.)  The
first part of the proof is the following theorem.
\begin{thm}\label{thm:URdecomp}
  If $A\in \Ccell_d(\tau;\Z_\nu)$ is a $d$-cycle in the unit grid in
  $\R^N$, then there are cycles
  $M_1,\dots,M_k\in \Ccell_d(\tau;\Z_\nu)$ and uniformly rectifiable
  sets $E_1,\dots, E_k\subset \R^N$ with bounded uniform
  rectifiability constants such that
  \begin{enumerate}
  \item $A=\sum_{i} M_i,$
  \item $\supp M_i\subset E_i,$
  \item $\mass M_i \sim |E_i|,$ and
  \item $\sum_i |E_i| \lesssim \mass A.$
  \end{enumerate}
\end{thm}
Here, $|\cdot|$ represents $d$-dimensional Hausdorff measure.  The
proof of Theorem~\ref{thm:URdecomp} relies on results of David and
Semmes on quasiminimizing sets; they show that if a set $E$ is not
uniformly rectifiable, then there is a compactly supported deformation
that decreases the volume of the set.  We use a sequence of such
deformations to construct the desired decomposition.

This decomposition breaks complicated surfaces into ``simple''
pieces.  For example, the surface $C_k$ above is built by starting
with a simple surface and repeatedly replacing discs in the surface by
handles and crosscaps.  This decomposition reverses this process.
That is, if we write
$$C_k=C_0+\sum_{i=0}^{k+1} (C_{i+1}-C_{i}),$$
then we can write each term $C_{i+1}-C_{i}$ as a sum of the
fundamental classes of $100^{i}$ disjoint projective planes of
diameter roughly $10^{-i}$.  We can thus write $C_k$ as the sum of the
unit cube and a large number of projective planes of different scales.
The total area of all of these pieces is at most a multiple of the
area of $C_k$, and each piece is uniformly rectifiable.  In fact, each
projective plane is a scaling of a fixed projective plane, so each
piece is uniformly rectifiable with the same constants.

The second stage of the decomposition is more complicated to describe
and we will sketch it more fully in Section~\ref{sec:URcaseSketch}.
The idea of the decomposition is that a uniformly rectifiable set $E$
is close to a Lipschitz graph at ``most'' locations and scales.  This
can be expressed in terms of a \emph{corona decomposition}, which,
very roughly speaking, breaks $E\times (0,\infty)$ into ``good'' and
``bad'' cubes so that the total size of the set of bad cubes is
bounded and such that when $(x,r)$ lies in a good cube, $B_r(x)$ is
close to a Lipschitz graph.  Furthermore, the good cubes can be
collected into \emph{stopping-time regions} so that all the cubes in a
stopping-time region lie close to the same Lipschitz graph, and the
total size of the set of stopping-time regions is also bounded.

If $M$ is supported on a uniformly rectifiable set, we will use a
corona decomposition of $\supp M$ to decompose it into a sum of
cycles, one for each bad cube and each stopping-time region.  The
cycle corresponding to a bad cube will be a sum of boundedly many
cells; these cycles are combinatorially simple, but could have
nontrivial topology.  

The stopping-time regions are more complex.  Each
stopping-time region may contain arbitrarily many good cubes, so the
corresponding cycle may consist of arbitrarily many faces.  A
stopping-time region, however, lies close to a Lipschitz graph, which
restricts the topology of the cycle.  

Thus, each bad cube corresponds to a geometrically simple cycle with
nontrivial topology and each stopping-time region corresponds to a
cycle with complicated geometry but controlled topology.  The cycles
sum to $M$, and each piece $B_i$ satisfies an inequality
$\NO(B_i)\lesssim \mass B_i$.  The bounds on the total size of the
corona decomposition then imply
$$\NO(M)\lesssim \sum_i \mass B_i\lesssim \mass(M).$$
Combining these two stages, we obtain the desired bound on $\NO(A)$.

\subsection{Overview}

We start by introducing some necessary notation and other
preliminaries (Sec.~\ref{sec:prelims}), including cellular and
Lipschitz chains and some versions of the Federer-Fleming deformation
theorem.  Then, in Section~\ref{sec:applications}, we derive the
applications Thm.~\ref{thm:NObound} and
Corollaries~\ref{cor:YoungsProblem}--\ref{cor:integralCurrentsLift}
from Thm.~\ref{thm:cellNObound}.

In the remaining sections of the paper, we prove
Thm.~\ref{thm:cellNObound}.  The proof breaks down into two main
pieces, which we sketch in Section~\ref{sec:expandedSketch}.  First,
in Sec.~\ref{sec:URdecomp}, we introduce uniform rectifiability and
prove Thm.~\ref{thm:URdecomp}, which decomposes an arbitrary cellular
cycle into a sum of cycles with uniformly rectifiable support.  Then,
in Sec.~\ref{sec:URcase}, we prove that Thm.~\ref{thm:cellNObound}
holds for cycles with uniformly rectifiable support and conclude that
Thm.~\ref{thm:cellNObound} holds in general.

\begin{ack} The author was supported by a Discovery Grant from the
  Natural Sciences and Engineering Research Council of Canada, by a
  grant from the Connaught Fund, University of Toronto, and by a Sloan
  Research Fellowship.  Many of the theorems were proved while the
  author was employed at the University of Toronto.

  The author would like to thank Larry Guth for introducing him to the
  problem and for many helpful discussions.  The author would also
  like to thank Jonas Azzam, Frank Morgan, Brian White, and Jean
  Taylor.
\end{ack}

\section{Preliminaries}\label{sec:prelims}

\subsection{Definitions and notation}\label{sec:defNot}

In this section we will give some definitions and notation, including
asymptotic notation, polyhedral complexes, QC complexes, Lipschitz
chains, and flat equivalence.  

We will write $f\lesssim g$ when there is a universal constant $c>0$
such that $f\le c g$.  If, instead, there is a $c=c(\alpha,\beta)$
depending on some parameters $\alpha,\beta$ such that
$f\lesssim c(\alpha,\beta) g$, we write $f\lesssim_{\alpha,\beta} g$,
and if $f \lesssim g$ and $g\lesssim f$, we write $f\sim g$.  In this
paper, all implicit constants will be taken to depend on $d$ and $N$,
so we will omit those subscripts in our notation.

A \emph{polyhedral complex} is a locally finite CW-complex whose cells are isometric
to convex polyhedra, glued by isometries.  Such a complex is
\emph{quasiconformal} if there is a $c$ such that each cell is
$c$-bilipschitz equivalent to a scaling of the unit ball of the same
dimension.  We will refer to
quasiconformal polyhedral complexes as \emph{QC complexes} and we will
refer to $c$ as the QC constant of the complex.

The QC complex we will use most frequently is a complex $\Sigma$ that
subdivides $\R^N\times [1,\infty)$ into dyadic cubes.  To construct
$\Sigma$, we tessellate each slab of the form
$\R^N\times [2^{i},2^{i+1}]$ into by dyadic cubes of side length
$2_i$, then let $\Sigma$ be the QC complex whose top-dimensional cells
are the cubes in these tessellations.  Note that when $i>0$, the plane
$\R^N\times \{2^i\}$ is part of two such tessellations, one with side
$2^i$ and one with side $2^{i-1}$, so the plane is subdivided into
cubes of side $2^{i-1}$.

Suppose that $X$ is a polyhedral complex.  We will also denote
its underlying space by $X$ when it is not ambiguous, and we
denote its $d$-skeleton by $X^{(d)}$.  We will think of cells of
$X$ as closed sets.  We let $\Ccell_*(X;\K)$ be the complex
of cellular chains on $X$ with coefficient group $\K$ and we let
$\CLip_*(X;\K)$ denote the complex of \emph{singular Lipschitz chains}
or simply \emph{Lipschitz chains} on $X$ with coefficients in $\K$.
This is the subcomplex of the complex of singular chains consisting of
formal sums of Lipschitz maps of simplices into $X$.  Given a chain
$A\in \CLip_*(X;\K)$, we define $\supp A$ to be the union of the
images of the simplices that occur in $A$ with non-zero coefficients.
Since the barycentric subdivision of $X$ is a simplicial complex,
we can view $\Ccell_*(X;\K)$ as a subset of $\CLip_*(X;\K)$
by identifying each face of $X$ with the sum of the simplices in
its barycentric subdivision.

Suppose that $A\in \CLip_d(\R^N;\K)$ is a Lipschitz $d$-chain with
coefficients in a normed abelian group $\K$ and that 
\[A=\sum_i a_i \alpha_i,\]
where $a_i\in \K$ and the $\alpha_i$ are Lipschitz maps from the
standard $d$-simplex to $\R^N$.  
By Rademacher's Theorem, the $\alpha_i$'s are differentiable almost
everywhere, so we may define
\[\mass A=\sum_i |a_i|\vol^d \alpha_i,\]
where 
\[\vol^d \alpha=\int_\Delta |J_\alpha(x)|\; dx\]
and $J_\alpha(x)$ is the jacobian determinant of $\alpha$.  In this paper, $\K$ will
either be $\Z$ with the usual norm or it will be $\Z_\nu$ with norm
$$|x|=\min\{|y|\mid x\equiv y \pmod{\nu}\}.$$

If $X$ is a polyhedral complex, then defining the mass of a chain is
slightly more complicated.  Suppose that $\alpha: \Delta\to X$ is a
Lipschitz map defined on a $d$-simplex $\Delta$.  For each cell
$\sigma\subset X$, let $\Delta_\sigma=\alpha^{-1}(\inter \sigma)$.  Then
the $\Delta_\sigma$'s partition $\Delta$ into countably many disjoint measurable
subsets such that the image of each subset lies in a single cell of
$X$.  Consider the restriction
\[\alpha|_{\Delta_\sigma}:\Delta_\sigma\to \sigma.\]
Since $\alpha$ is Lipschitz, we can extend this to a Lipschitz map
$\alpha'_\sigma:\Delta\to \sigma$ by the Whitney extension theorem.
This map is differentiable a.e.\ in $\Delta$, and the derivative
$D\alpha'_\sigma(x)$ is independent of the choice of extension when
$x$ is a Lebesgue density point of $\Delta_\sigma$.  Thus the jacobian
determinant $J_\alpha(x)$ is well-defined almost everywhere on
$\Delta_\sigma$.  Repeating this for the other cells of $X$, we can
define $J_\alpha(x)$ almost everywhere on $\Delta$ and define
\begin{equation}\label{eq:mass1}
  \vol^d \alpha=\int_{\Delta}|J_\alpha(x)|\; dx.
\end{equation}
Let
\[\mass A=\sum_i |a_i|\vol^d \alpha_i,\]
for any chain
\[A=\sum_i a_i \alpha_i.\]

If $B\subset X$ is a Borel set and $A$ is a Lipschitz chain, let
$\mass_B A$ be the mass of the restriction of $A$ to $B$.  That is, if
$\Delta$ is a simplex and $\alpha:\Delta\to X$ is Lipschitz, we let
\[\vol_B \alpha=\int_{\alpha^{-1}(B)}|J_\alpha(x)| \;dx.\]  If $A=\sum_i
a_i \alpha_i$ for some maps $\alpha_i:\Delta\to X$ and some
coefficients $a_i$, we let 
\begin{equation}\label{eq:defMassRestrict}
  \mass_B A = \sum_i |a_i|\vol_B \alpha_i
\end{equation}
and
\[\|A\|_1 = \sum_i |a_i|.\]

A single surface may have many different triangulations, each of which
corresponds to a different Lipschitz chain.  To avoid this, we will
define the notion of \emph{flat equivalence}.  Given a chain
$A\in \CLip_d(X;\K)$, we define its \emph{filling volume} as
\[\FV(A)=\mathop{\inf_{B\in \CLip_{d+1}(X;\K)}}_{\partial B=A} \mass B\]
and define its \emph{flat norm} as
\[\fnorm{A}=\inf\{\mass Q+\mass R\mid Q\in \CLip_d(X;\K),R\in
\CLip_d(X;\K), A=Q+\partial R\}.\]
If we take $Q=A$, $R=0$, this definition implies that
$\fnorm{A}\le \mass A$ and when $A$ is a cycle, then
$\fnorm{A}\le \FV(A)$.  Two chains $A,A'\in \CLip_d(X;\K)$
are called flat-equivalent if $\fnorm{A-A'}=0$.

Lipschitz $d$-chains in a $d$-complex (or in the
$d$-skeleton of a complex) are flat-equivalent to cellular chains.
\begin{lemma} \label{lem:LipToCell}
  If $X$ is a polyhedral complex and $A\in
  \CLip_d(X^{(d)};\K)$ is a Lipschitz $d$-chain such that
  $\supp \partial A\subset X^{(d-1)}$ (in particular, if $A$ is a
  cycle), then there is a cellular chain $A'\in \Ccell_d(X;\K)$
  which is flat-equivalent to $A$.  If we write 
  \[A'=\sum_{K\in X^{(d)}}a_K [K],\]
  where $a_K\in \K$, $K$ ranges over the $d$-cells of $X$, and
  $[K]$ is the chain corresponding to $K$, then
  \[|a_K|\le \frac{\mass_K(A)}{\Haus^d(K)}.\]
\end{lemma}
\begin{proof}
  Consider $A$ as an element of 
  $H^{\text{Lip}}_d(X^{(d)},X^{(d-1)};\K)$, the relative
  Lipschitz homology.  Since $X$ is locally finite and thus
  locally a Lipschitz neighborhood retract, its Lipschitz
  homology and its singular homology are isomorphic.  Since it is a CW
  complex, its singular homology is isomorphic to its cellular
  homology.  Therefore, there is
  a cellular chain $A'\in \Ccell_d(X;\K)$ which is
  homologous to $A$ relative to $X^{(d-1)}$.  That is, there is
  some $(d+1)$-chain $B\in \CLip_d(X^{(d)};\K)$ such that
  \[\partial B-(A-A')\in \CLip_d(X^{(d-1)};\K).\]  
  Then $B$ is a $(d+1)$-chain in $X^{(d)}$, so its mass is 0.
  The difference $\partial B-(A-A')$ is a $d$-chain in
  $X^{(d-1)}$, so its mass is also 0, and
  $\fnorm{A-A'}=0$.

  If $K$ is a $d$-cell, its coefficient $a_K$ is the degree
  with which $A$ covers $K$.  Since $\partial A$ lies in the
  $(d-1)$-skeleton of $X$, this degree is well-defined and
  \[|a_K|\le \frac{\mass_K(A)}{\Haus^d(K)},\]
  as desired.
\end{proof}

More generally, Lipschitz chains in a QC complex can be approximated
by cellular chains.  This is a consequence of the deformation theorem,
which we will discuss in Section~\ref{sec:FF}.

\subsection{Currents over $\Z$ and $\Z_\nu$}

Here we will recall some notation and theorems for currents with
coefficients in $\Z$ and in $\Z_\nu$.  This will primarily be used in
proving the applications to currents in Section~\ref{sec:currentApps};
it is not necessary for the proof of the main theorem.

For a full development of integral currents and flat chains, see
\cite{FedererGMT} or \cite{SimonGMT}.  Our development of currents
modulo $\nu$ is taken from \cite{FedererGMT}, with the change that our
rectifiable currents will be locally rectifiable currents with finite
mass, rather than rectifiable currents with compact support.  Let
$\mathcal{P}_d(\R^N)$ be the set of polyhedral chains with integer
coefficients, and let $\cR_d(\R^N)$ be the set of \emph{rectifiable
  $d$--currents}.  This is the closure of $\mathcal{P}_d(\R^N)$ under
the mass norm; in the terminology of \cite{FedererGMT}, these are
locally rectifiable currents.  The set $\bI_d(\R^N)$ of \emph{integral
  $d$--currents} consists of rectifiable currents with rectifiable
boundary.  All of these are subsets of the set $\cF_d(\R^N)$ of
\emph{integral flat chains}, which can be defined as
$$\cF_d(\R^N)=\{R+\partial S\mid R\in \cR_d(\R^N), S\in
\cR_{d+1}(\R^N)\}.$$
If $T\in \cF_d(\R^N)$, we define its \emph{flat norm} by
$$\fnorm{T}=\inf \{\mass R+\mass S\mid T=R+\partial S, R\in \cR_d(\R^N), S\in
\cR_{d+1}(\R^N)\}.$$
Since $\cR_d$ is complete with respect to mass, the set of integral
flat chains is complete with respect to $\cF$
\cite[4.1.24]{FedererGMT}.

Federer and Fleming proved that integral currents satisfy a
compactness property \cite{FedFlemNormInt}.
\begin{thm}[{see \cite[8.13, 7.1]{FedFlemNormInt} or \cite[27.3, 31.2]{SimonGMT}}] \label{thm:integralCurrentCompactness}
  If $T_i\in \bI_d(\R^N)$ is a sequence of integral currents such that
  $$\sup_i \mass T_i+\mass \partial T_i<\infty,$$
  then there is a subsequence $T_{k_i}$ and an integral current
  $T \in \bI_d(\R^N)$ such that $\lim_i \fnorm{T-T_{k_i}}=0$.
\end{thm}

Extending the definitions above to currents modulo $\nu$ while keeping
the compactness property is subtle.  Again, for a full
development of currents modulo $\nu$, see \cite{FedererGMT}.  Let
$\nu\ge 2$ be an integer.  When $T\in \cF_d(\R^N)$, we define its
\emph{mod-$\nu$ flat norm} by letting
\begin{equation}
  \begin{split}
    \cF^\nu(T)=\inf \{\mass R+\mass S\mid\,& R\in\cR_d(\R^N),S\in\cR_{d+1}(\R^N), Q\in \cF_d(\R^N),\\
    &T=R+\partial S+\nu Q\}.
  \end{split}
\end{equation}
The mod-$\nu$ flat norm of any multiple of $\nu$ is zero, but it is
\emph{a priori} unclear that the converse holds, namely, that if
$\cF^\nu(T)=0$, then $T\in \nu \cF_d(\R^N)$.  (See
Corollary~\ref{cor:multiplesEmbedding}.)  Let
$$\overline{\nu \cF_d(\R^N)}=\{T\in \cF_d(\R^N)\mid \cF^\nu(T)=0\}.$$
This is a closed subgroup with respect to $\cF$, and we define the
flat chains modulo $\nu$ as:
$$\cF_d^\nu(\R^N)=\cF_d(\R^N)/\overline{\nu \cF_d(\R^N)}.$$
If $T, U\in \cF_d(\R^N)$, we denote the coset of $T$ in
$\cF_d^\nu(\R^N)$ by $(T)^\nu$, and if $\cF^\nu(T-U)=0$, we write
$T\equiv U\pmod \nu$.  The set of flat chains modulo $\nu$ is complete
with respect to $\cF^\nu$ \cite[4.2.26]{FedererGMT}.

We define rectifiable and integral currents modulo $\nu$ as subsets of
$\cF_d^\nu(\R^N)$.  Let
$$\cR_d^\nu(\R^N)=\{(T)^\nu \in \cF_d^\nu(\R^N)\mid T\in\cR_d(\R^N)\}$$
and
$$\bI_d^\nu(\R^N)=\{T \in \cR_d^\nu(\R^N)\mid \partial T\in\cR^\nu_{d-1}(\R^N)\}.$$
Note that $(\bI_d(\R^N))^\nu\subset \bI_d^\nu(\R^N)$, but equality is
not obvious.  Indeed, as noted in the introduction, Federer claimed
that generally $\bI_d^\nu(\R^N)\ne (\bI_d(\R^N))^\nu$, using an
infinite sum of embedded projective planes as an example, but this is
incorrect, as we shall see in Corollary~\ref{cor:modNuQuotient}.

Finally, if $T\in \cF_d(\R^N)$, we define $\mass^\nu T$ to be the
smallest $m\in \R$ such that for every $\epsilon>0$, there exists an
$R\in \cR_d(\R^N)$ such that $\cF^\nu(R-T)\le \epsilon$ and
$\mass R\le m+\epsilon$.  This is constant on cosets of
$\overline{\nu \cF_d(\R^N)}$, so it descends to a function on
mod-$\nu$ currents that extends the usual notion of mass.

Like their counterparts with integer coefficients, integral currents
modulo $\nu$ satisfy a compactness property:
\begin{thm}[{see \cite[4.2.27]{FedererGMT}}] 
  If $T_i\in \bI^\nu_d(\R^N)$ is a sequence of integral currents such that
  $$\sup_i \mass^\nu T_i+\mass^\nu \partial T_i<\infty,$$
  then there is a subsequence $T_{k_i}$ and an integral current
  $T \in \bI^\nu_d(\R^N)$ such that $\lim_i \cF^\nu(T-T_{k_i})=0$.
\end{thm}

\subsection{The deformation theorem}\label{sec:FF}

Federer and Fleming proved a deformation theorem stating that a chain
$T$ in $\R^N$ with finite mass and finite boundary mass can be
approximated by a cellular chain $P$ in a grid of side length $r$ such
that the mass of $P$ and the flat norm of $P-T$ are bounded in terms
of the mass of $T$ and the mass of $\partial T$ \cite{FedFlemNormInt}.
White generalized their result to flat chains by introducing
deformation operators $P$ (an approximation operator) and $H$ (a
homotopy from $P$ to the identity).  We state his theorem in part.
\begin{thm}[{\cite{WhiteDeform}}]\label{thm:WhiteDeform}
  Let $\tau_r$ be the grid of side length $r>0$ in $\R^N$ and let
  $\cF_d$ be the space of flat chains in $\R^N$ with coefficients in a
  normed abelian group $\K$.  If $y\in \R^N$ and $A\in \cF_d$, let
  $t_yA\in \cF_d$ be the translation of $A$ by the vector $y$.  There
  is a $c>0$ such that for every $r>0$, there are operators
  $$P=P^r\from \cF_d\to \Ccell_d(\tau_r)$$
  $$H=H^r\from \cF_d\to \cF_{d+1},$$
  where $\cF_d$ is the space of flat chains in $\R^N$, such that for
  all $A,B\in \cF_d$, for almost all $y\in \R^N$,
  $$t_yA=P(t_yA)+\partial H(t_yA)+H(\partial t_yA).$$
  Furthermore,
  \begin{align*}
    \int_{y\in [0,1]^N}\mass P(t_yA)&\le c \mass A\\
    \int_{y\in [0,1]^N}\mass H(t_yA)&\le c r \mass A
  \end{align*}
\end{thm}

We will need a similar approximation lemma which replaces the complex
$\tau_r$ with a QC complex and provides bounds on the approximations
of a locally finite set of chains.  If $X$ is a QC complex and
$B\subset X$, let $\nbhd(B)$ be the union of all the cells of $X$ that
intersect $B$.  If $S\subset X$, let $\HC^d(S)$ denote the
$d$-dimensional Hausdorff content of $S$ and let $\Haus^d(S)$ denote
its $d$-dimensional Hausdorff measure.

\begin{lemma}\label{lem:simultFF}
  Let $X$ be a QC complex of dimension $N$.  Let $\mathcal{T}\subset
  \CLip_*(X)$ be a set of chains, possibly of different
  dimensions, which is closed under taking boundaries.  Suppose that
  $\mathcal{T}$ is locally finite in the sense that there is a $n>0$
  such that for any cell $D\in X$, no more than $n$ elements of
  $\mathcal{T}$ intersect $\nbhd D$.  

  Then there is a $C>0$ depending on $n$, $N$, and the QC constant of
  $X$ and there is a locally Lipschitz map $p:X\to X$ such that for any
  $T\in \mathcal{T}$ of dimension $d=\dim T$, 
  \begin{enumerate}
  \item $p(\supp T)\subset X^{(d)},$ 
  \item $\mass p_\sharp(T)\le C \mass T,$ and
  \item $\HC^{d}(p(\supp (T)))\le C \HC^{d}(\supp T).$
  \end{enumerate}
  By Lemma~\ref{lem:LipToCell}, each chain $p_\sharp(T)$ is
  flat-equivalent to a cellular chain, which we denote $P(T)$.  If
  $\langle \mathcal{T}\rangle$ is the chain complex generated by
  $\mathcal{T}$, we can view $P$ as a homomorphism 
  \[P:\langle \mathcal{T}\rangle \to \Ccell_*(X).\]

  These maps are local in the sense that for any cell $D$ of $X$,
  we have $p(D)\subset D$.  In fact, if $\inter D$ is the interior of
  $D$, then
  \begin{enumerate}
    \setcounter{enumi}{3}
  \item $\Haus^d(p(S\cap \inter D))\le C \Haus^d(S\cap \inter D)$,
  \item \label{it:localApproxNbhd} if $Y\subset X$, then for any $T\in \mathcal{T}$,
    $$\mass_Y P(T)\le C \mass_{\nbhd Y} T,$$
  \item $\HC^{d}(\supp P(T)\cap Y)\le C \HC^{d}(\supp T\cap \nbhd Y).$
  \end{enumerate}
  
  Therefore, if $T\in \mathcal{T}$, then $P(T)$ is supported on $\nbhd
  (\supp T)$, and if $T\in \mathcal{T}$ is cellular, then $P(T)=T.$
\end{lemma}
If $\mathcal{T}\subset \CLip_*(X)$ is a locally finite set of
chains and $P:\langle \mathcal{T}\rangle\to \Ccell_*(X)$ is as in the
lemma, we call $P$ a \emph{deformation operator} approximating
$\mathcal{T}$.

We defer the proof of Lemma~\ref{lem:simultFF} to Appendix~\ref{sec:simultFFProof}.

David and Semmes used a different deformation lemma to
deform $d$\hyp{}dimensional sets into the $d$-skeleton of a
grid in $\R^N$.  This lemma can be generalized to QC complexes.  If
$U\subset X$, we say that a map $f:X\to X$ is a deformation supported
on $U$ if
\[\{x\in X \mid x\ne f(x)\}\cup \{f(x)\in X \mid x\ne f(x)\}\subset U\]
\begin{lemma}[{see \cite[Prop.~3.1, Lemma~3.31]{DSQuasi}}]\label{lem:FFsets}
  Let $X$ be a QC complex of dimension $N$ and let $d<N$.  Let
  $E\subset X$ be a closed set such that $\Haus^d(E\cap
  B)<\infty$ for any ball $B\subset \R^N$ and let $X_0\subset
  X$ be a subcomplex.  Then there is a $C>0$ depending on $N$ and
  the QC constant of $X$ and a deformation
  $p:X\to X$ supported on $\nbhd X_0$ that is Lipschitz on each cell of $X$ and collapses
  $E\cap X_0$ to the $d$-skeleton of $X$.  That is,
  \begin{enumerate}
  \item $p(E\cap X_0)\subset X_0^{(d)}$,
  \item $p$ restricts to the identity map on $X^{(d)}$,
  \item \label{it:FFsetsBound} $p$ satisfies
    \begin{align*}
      \Haus^d(p(E))&\le C \Haus^d(E)\\
      \Haus^d(p(E))-\Haus^d(E)&\le C \Haus^d(E\setminus X^{(d)})
    \end{align*}
  \end{enumerate}

  As in the previous lemma, for any cell $D$ of $X$, we have
  $p(D)\subset D$.  In fact, if $\inter D$ is the interior of $D$, then
  \begin{enumerate}
    \setcounter{enumi}{3}
  \item \label{it:localApproxSets} $\Haus^d(p(S\cap \inter D))\le C \Haus^d(S\cap \inter D),$
  \item \label{it:localApproxSetsNbhd}  if $Y\subset X$, then for any $T\in \mathcal{T}$,
    $$\mass_Y P(T)\le C \mass_{\nbhd Y} T.$$
  \item If 
    \[c^{-1}r^d\le \cH^d(E\cap B(x,r))\le c r^d\]
    for any $x\in E$ and any $0<r<\max_{\sigma\in X} \diam \sigma$,
    (i.e., $E$ is Ahlfors $d$-regular) then we can take $p$ to be
    Lipschitz with Lipschitz constant depending on $c$ and $N$.
  \end{enumerate}
\end{lemma}
\begin{proof}[Sketch of proof]
  This lemma is essentially Prop.~3.1 and Lemma~3.31 of \cite{DSQuasi}
  with two differences.  First, Prop.~3.1 of \cite{DSQuasi}
  applies to grids in $\R^N$ rather than QC complexes.  This is a
  minor difference; the key lemma used in the proof of Prop.~3.1 is a
  bound on the size of a random projection from the interior of a ball
  to its boundary, and this bound applies equally to grid cells and to
  cells in a QC complex.  This bound implies
  part \ref{it:localApproxSets}.

  Second, we need to show the second bound in part \ref{it:FFsetsBound}.  By
  part \ref{it:localApproxSets}, 
  \[\Haus^d(p(E\setminus X^{(d)}))\le C \Haus^d(E\setminus
  X^{(d)}),\]
so, since $p(E\cap X^{(d)})=E\cap X^{(d)}$, we have
\begin{align*}
  \Haus^d(p(E))-\Haus^d(E)&\le \Haus^d(p(E\setminus X^{(d)})) - \Haus^d(E\setminus X^{(d)})\\
  &\le C \Haus^d(E\setminus X^{(d)})
\end{align*}
as desired.
\end{proof}

If $U$ is a closed subset of $X^{(d)}$, a similar process lets us
``trim'' any $d$-cells of $X$ which are only partially covered by
$U$ by pushing $U$ into their boundaries.  This results in an
approximation of $U$ that is almost a union of $d$-cells of $X$.
For any $S\subset X$, let
\[S^*=\{x\in \R^N\mid \Haus^d(S\cap B(x,r))>0 \text{ for all }r>0\}.\]
This is a closed set.
\begin{lemma}\label{lem:cellularize}
  Let $X$ be a QC complex of dimension $N$ and let $d<N$.  Let
  $U\subset X^{(d)}$ be a closed set.  Then there is a 
  map $q:X\to X$ which is Lipschitz on each
  cell of $X$ such that:
  \begin{enumerate}
  \item for any cell $D$ of $X$, we have $q(D)\subset D$,
  \item $q$ restricts to the identity map on $X^{(d-1)}$ and
    restricts to a degree-1 map on each $d$-cell of $X$, and
  \item \label{it:cellularizeBound} $q(U)^*$ is the union of all of the $d$-cells whose interiors are
    contained in $U$, so $q(U)^*\subset U$ and
    $|q(U)|\le |U|$.
  \end{enumerate}
\end{lemma}
\begin{proof}
  We construct $q$ on each $d$-cell of $X$, then extend it to
  $X$.  Let $D\in X^{(d)}$ be a $d$-cell.  Since $X$ is
  a QC complex, we may identify $D$ with a closed ball
  \[B=B(0,R)\subset \R^d\] by a bilipschitz map.  If $\inter B\subset U$
  or if $B\cap U=\emptyset$, we define $q$ as the identity on $B$.
  Otherwise, there's some $y\in \inter B$ such that $y\not \in U$.  Since $U$
  is closed, we may let $\epsilon>0$ be such that $B(x,\epsilon)\cap
  U=\emptyset$ and $B(x,2\epsilon)\subset B$.  Then there is a
  Lipschitz map $B\to B$ which sends $B(x,\epsilon)$
  homeomorphically to $B$, is the identity on $\partial B$, and sends
  $B\setminus B(x,\epsilon)$ to $\partial B$.  We define $q$ to be
  such a map on $B$.  In either case, $q$ is a degree-1 map of $B$ to
  itself and restricts to the identity map on $\partial B$, so $q$ is
  well-defined on all of $X^{(d)}$ and is the identity on
  $X^{(d-1)}$, just as we claimed.  Once we've defined $q$ on the
  $d$-skeleton, we can extend it to all of $X$ by a sequence of
  radial extensions.

  Finally, for each $d$-cell $D\in X^{(d)}$, we either have
  $\inter D\subset q(U)$ (if $\inter D\subset U$) or $\inter D\cap
  q(U)=\emptyset$ (otherwise), so $q(U)^*$ is the union of all of the
  $d$-cells that are contained in $U$.
\end{proof}

\subsection{Nonorientability}\label{sec:defineNO}

Let $\tau$ be the unit grid in $\R^N$ and let $\nu$ be an integer such
that $\nu\ge 2$.  If $A\in \Ccell_d(\tau;\Z_\nu)$ (resp.\
$A\in \bI^\nu_d(\R^N)$) is a cycle, a \emph{mod-$\nu$
  pseudo-orientation} or simply \emph{pseudo-orientation} of $A$ is a
cycle $R\in \Ccell_d(\tau;\Z)$ (resp.\ $A\in \bI_d(\R^N)$) such that
$A\equiv R\pmod{\nu}$.  For all $A\in \Ccell(\tau;\Z_\nu)$ and
$A\in \bI^\nu_d(\R^N)$, we define
\[\NO(A)=\inf \{\mass R\mid \text{$R$ is a pseudo-orientation of
  $A$}\}.\]

Every cellular cycle $A\in \Ccell_d(\tau;\Z_\nu)$ has a
pseudo-orientation.  We can construct one such pseudo-orientation by
letting $A_\Z\in \Ccell_d(\tau)$ be an integral chain such that
$A_\Z\equiv A\pmod{\nu}$; i.e., a chain
$A_\Z=\sum_i \bar{a}_i \sigma_i$ where each coefficient $\bar{a}_i$ is
congruent mod $\nu$ to the corresponding coefficient of $A$.  Then
$\partial A_\Z\equiv \partial A\equiv 0\pmod{\nu}$, so $\partial A_\Z$
is a multiple of $\nu$.  Let $B\in \Ccell_d(\tau;\Z)$ be a multiple of
$\nu$ such that $\partial B=\partial A_\Z$.  Then $A_\Z-B$ is a cycle,
and $A_\Z-B\equiv A\pmod{\nu}$.

Unfortunately, the procedure above does not work if $A$ is an integral
current modulo $\nu$.  In this case, it is not \emph{a priori} clear
that there is an integral current $A_\Z$ such that
$A_\Z\equiv A\pmod{\nu}$ and $\supp A_\Z=\supp A$.  One of the main
goals of this paper is to prove that in fact, every cycle
$A\in \bI^\nu_d(\R^N)$ has a pseudo-orientation.

\section{Applications}\label{sec:applications}
In this section, we will use Theorem~\ref{thm:cellNObound} to prove 
the applications in Section~\ref{sec:introApplications}.

 
\subsection{Nonorientability and filling volumes}
When $T$ is a cycle with integer coefficients, the difference between
$\FV(T)$ and $\FV(\nu T)$ is closely connected to nonorientability.
On one hand, nonorientable surfaces give rise to cycles $T$ such that
$\FV(2T)<2\FV(T)$.  L. C. Young \cite{LCYoung} gave a recipe for
producing a curve $T$ from a nonorientable surface $M$ embedded in
$\R^N$; he defines $T$ as a ``zigzag'' across $M$ that represents a
torsion class in $H_1(M;\Z_2)$.  Then $2T$ can be filled by a surface
lying entirely on $M$, while any filling of $T$ must cut through $M$,
so $2\FV(T)>\FV(2T)$.  Similar techniques for fillings of different
multiplicities appear in \cite{MorganMultiples} and
\cite{WhiteMultiples}.

On the other hand, the following lemma bounds the difference between
$\FV(T)$ and $\FV(\nu T)$ in terms of the nonorientability of
$U\bmod \nu\in \Ccell_{d+1}(\tau;\Z_\nu)$.

\begin{lemma}
  If $T\in \Ccell_d(\tau)$ is a cycle in the unit grid $\tau$ in
  $\R^N$ and $U\in \Ccell_{d+1}(\tau)$ is a chain such that $\partial
  U=T$, then 
  $$\FV(T)\lesssim \frac{\mass U+\NO(U\bmod \nu)}{\nu}.$$
\end{lemma}
\begin{proof}
  Let $R\in \Ccell_{d+1}(\tau)$ be a pseudo-orientation of $U\bmod
  \nu$ such that $\mass R\lesssim \NO(U\bmod \nu)$ and let
  $B=\frac{U-R}{\nu}$.  Since $R\equiv U \pmod{\nu}$, the
  coefficients of $U-R$ are all multiples of $\nu$, so $B\in
  \Ccell_{d+1}(\tau)$.  Further,
  $$\partial B=\frac{\nu T-0}{\nu}=T,$$
  so $B$ is a filling of $T$ and
  $$\FV(T)\le \mass B\le \frac{\mass U}{\nu} +\frac{\mass R}{\nu}.$$
\end{proof}

This implies a cellular version of Corollary~\ref{cor:YoungsProblem}:
\begin{cor}
  If $T\in \Ccell_d(\tau)$ is a cycle in the unit grid $\tau$ in
  $\R^N$, then $\FV(T)\lesssim_\nu \FV(\nu T)$.
\end{cor}
\begin{proof}
  Let $U\in \Ccell_{d+1}(\tau)$ be a chain such that
  $\partial U=\nu T$ and $\mass U\lesssim \FV(\nu T)$.
  By Theorem~\ref{thm:cellNObound}, $\NO(U\bmod \nu)\lesssim_\nu \FV(\nu
  T)$, so by the previous lemma, 
  $$\FV(T)\lesssim_\nu \mass U\lesssim_\nu \FV(\nu T).$$
\end{proof}

Corollary~\ref{cor:YoungsProblem} follows by approximating $T$ by a
cellular cycle.  Namely, if $T$ is a Lipschitz $d$--cycle and
$\epsilon>0$, Lemma~\ref{lem:simultFF} implies that there is an $r>0$
and an approximating cycle $T_r\in \Ccell_d(\tau_r)$ such that
$\FV(T-T_r)< \epsilon$, where $\tau_r$ is the grid of side length $r$.
By the cellular case, there is a $c>0$ depending on $d$, $N$, and
$\nu$ such that $\FV(T_r)\le c \FV(\nu T_r)$, so
$$\FV(T)\le \FV(T_r)+\epsilon\le c\FV(\nu T_r)+\epsilon \le c\FV(\nu
T)+(\nu+1)\epsilon.$$
Letting $\epsilon$ go to zero, we conclude that $\FV(T)\le c\FV(\nu
T)$.

\subsection{Currents modulo $\nu$}\label{sec:currentApps}
In this section, we will show that Theorem~\ref{thm:NObound} and
Corollaries~\ref{cor:multiplesEmbedding}--\ref{cor:modNuQuotient}
follow from Theorem~\ref{thm:cellNObound}.  Our main tool is the
following lemma:
\begin{lemma}\label{lem:fnormMultiples}
  For all $d,N,\nu>0$, and for all $T\in \cF_d(\R^N)$, 
  \begin{equation}\label{eq:fnormMultiples}
    \fnorm{T}\lesssim_\nu \fnorm{\nu T}+\fnorm{\nu T}^{(d+1)/d}.
  \end{equation}
\end{lemma}
\begin{proof}
  By White's deformation theorem, every flat chain is the limit of
  cellular chains in finer and finer grids, so it suffices to prove
  the lemma when $T\in \Ccell_d(\tau_r)$ for some $r$.
  Let $m_\nu\from \Ccell_{*}(\tau_r)\to \Ccell_{*}(\tau_r;\Z_\nu)$ be
  the change-of-coefficients map.

  Let $r>0$ and let $T\in \Ccell_d(\tau_r)$.  There are rectifiable
  currents $R_0$ and $S_0$ such that $\nu T=R_0+\partial S_0$ and
  $\mass R_0+\mass S_0\le 2\fnorm{\nu T}$; since $T$ is cellular,
  Theorem~\ref{thm:WhiteDeform} implies that there are cellular
  approximations $R$ and $S$ such that $R+\partial S=\nu T$ and
  $\mass R+\mass S\lesssim \fnorm{\nu T}$.

  We have $\partial R=\nu \partial T\equiv 0\pmod \nu$, so $m_\nu(R)$
  is a mod-$\nu$ $d$--cycle.  By Theorem~\ref{thm:cellNObound}, there
  is a pseudo-orientation $R'\in \Ccell_d(\tau_r)$ such that
  $\partial R'=0$, $R'\equiv R\pmod{\nu},$ and
  $\mass R'\lesssim_\nu \mass R \lesssim \fnorm{\nu T}$.

  Let $M\in \Ccell_{d+1}(\tau_r)$ be a chain such that
  $\partial M=R'$.  By the isoperimetric inequality for $\R^N$, we may
  assume $\mass M\lesssim \fnorm{\nu T}^{(d+1)/d}$.  Then
  $$\partial M+\partial S\equiv \nu T\equiv 0 \pmod{\nu},$$
  so, using Theorem~\ref{thm:cellNObound} again, there is an
  $S'\in \Ccell_{d+1}(\tau_r)$ such that $\partial S'=0$,
  $$S'\equiv M+S \pmod{\nu},$$
  and
  $$\mass S'\lesssim_\nu \mass M +\mass S\lesssim \fnorm{\nu
    T}+\fnorm{\nu T}^{(d+1)/d}$$

  Let
  \begin{align*}
    R''&=\frac{R-R'}{\nu}\\
    S''&=\frac{M+S-S'}{\nu}.
  \end{align*}
  Since $R'\equiv R$ and $S'\equiv M+S$, the coefficients of $R''$ and
  $S''$ are all integers.  Furthermore, since $\partial S'=0$ and
  $\partial M=R'$,
  \begin{align*}
    R''+\partial S''
    &=\nu^{-1}(R-R'+\partial M+\partial S-\partial S')\\ 
    &=\nu^{-1}(R+\partial S)=T,
  \end{align*}
  so 
  $$\fnorm{T}\le \mass R''+\mass S''\lesssim_\nu \fnorm{\nu
    T}+\fnorm{\nu T}^{(d+1)/d}$$ as desired.
\end{proof}

Cor.~\ref{cor:multiplesEmbedding} follows from the lemma.  
\begin{proof}[{Proof of Cor.~\ref{cor:multiplesEmbedding}}]
  By the lemma, the two norms $\fnorm{T}$ and $\fnorm{\nu T}$ on
  $\cF_d(\R^N)$ induce equivalent topologies, so the multiply-by-$\nu$ map in
  Corollary~\ref{cor:multiplesEmbedding} is an embedding.  

  If $T$ is in the closure of $\nu \cF_d(\R^N)$, then there is a
  sequence $T_i\in \nu \cF_d(\R^N)$ such that $\fnorm{T-T_i}\to 0$.
  Since $\{T_i\}$ is Cauchy, the lemma implies that $\{T_i/\nu\}$ is
  also Cauchy.  Let $T'=\lim_i T_i/\nu \in \cF_d(\R^N)$.  Then
  $$\fnorm{\nu T'-T}\le \nu \fnorm{T'-\frac{T_i}{\nu}}+\fnorm{T_i-T}.$$
  The right side goes to zero, so $T=\nu T'\in \nu \cF_d(\R^N)$.  We
  conclude that $\nu \cF_d(\R^N)$ is closed and that the
  multiply-by-$\nu$ map is an embedding with closed image.

  The same argument with $\cF_d$ replaced by $\bI_d$ implies that
  $\nu \bI_d(\R^N)$ is closed.
\end{proof}

The lemma is also helpful to show that Theorem~\ref{thm:cellNObound}
implies Theorem~\ref{thm:NObound} and
Corollary~\ref{cor:integralCurrentsLift}:
\begin{proof}[{Proof of Theorem~\ref{thm:NObound}}]
  Suppose that $A\in \bI^\nu_d(\R^N)$ and that $\partial A=0$.  By the
  deformation theorem for integral currents modulo $\nu$
  \cite[4.2.26]{FedererGMT}, there is a sequence
  $P_k\in \Ccell_d(\tau_k;\Z_\nu)$ of cellular approximations of $A$
  in finer and finer grids so that as $k\to \infty$, the cycles $P_k$
  converge to $A$ and $\mass P_k\lesssim \mass^\nu A$ for all $A$.

  For each $k>0$, let $R_k\in \bI_d(\R^N)$ be a pseudo-orientation of $P_k$.  By
  Theorem~\ref{thm:cellNObound}, we can choose the $R_k$ so that
  $$\mass R_k\le c\mass P_k \lesssim \mass^\nu A.$$
  By Theorem~\ref{thm:integralCurrentCompactness}, there is some
  subsequence $k_i$ such that $R_{k_i}$ converges.  Let
  $R=\lim_i R_{k_i}\in \bI_d(\R^N)$.  We claim that $R$ is a
  pseudo-orientation of $A$.  Write
  $$R-A=\lim_i R_{k_i}-P_{k_i}.$$
  The set $\nu \bI_d(\R^N)$ is closed and
  $R_{k_i}-P_{k_i}\in \nu \bI_d(\R^N)$ for all $i$, so
  $R-A\in \nu \bI_d(\R^N)$ and $R\equiv A\pmod{\nu}$.  Finally,
  $\partial R=0$ and $\mass R\lesssim \mass^\nu A$, so $\NO(A)\lesssim
  \mass^\nu A$.
\end{proof}

\begin{proof}[{Proof of Cor.~\ref{cor:integralCurrentsLift}}]
  Let $T\in \bI^\nu_d(\R^N)$.  Then $\partial T$ is a mod-$\nu$ cycle,
  so by Theorem~\ref{thm:NObound}, it lifts to an integral current
  $S\in \bI_{d-1}(\R^N)$ such that $\partial S=0$,
  $S\equiv \partial T\pmod{\nu}$, and
  $\mass S\lesssim \mass^\nu \partial T$.  By the isoperimetric inequality
  for integral currents, there is an $R\in \bI_{d}(\R^N)$ such that
  $\partial R=S$ and $\mass R\lesssim (\mass^\nu \partial T)^{d/(d-1)}$.

  Consider the mod-$\nu$ current $T'=T-(R\bmod \nu)$.  Since
  $\partial T \cong \partial R\pmod \nu$, this is a cycle modulo
  $\nu$, so, applying Theorem~\ref{thm:NObound} again, there is a
  $U\in \bI_{d}(\R^N)$ such that $U\equiv T'\pmod \nu$ and
  $\mass U\lesssim \mass^\nu T'$.  The sum $U+R$ is an integral current
  such that $U+R\equiv T'+R\equiv T\pmod{\nu}$ and
  $$\mass(U+R)\lesssim \mass^\nu T+(\mass^\nu \partial T)^{d/(d-1)}.$$
\end{proof}

Finally, Corollaries~\ref{cor:multiplesEmbedding} and
\ref{cor:integralCurrentsLift} imply that the set of integral currents
modulo $\nu$ is a quotient of the set of integral currents.
\begin{proof}[{Proof of Corollary~\ref{cor:modNuQuotient}}]
  Let $q_\nu\from \bI_d(\R^N) \to \bI^\nu_d(\R^N)$ be the map
  $q_\nu(T)=(T)^\nu$.  Corollary~\ref{cor:integralCurrentsLift}
  implies that $q_\nu$ is a surjection with kernel equal to the
  closure of $\nu\bI_d(\R^N)$.  By
  Corollary~\ref{cor:multiplesEmbedding}, the set $\nu\bI_d(\R^N)$ is
  closed, so in fact,
  $$\bI^\nu_d(\R^N)=\bI_d(\R^N)/\nu\bI_d(\R^N).$$
\end{proof}

\section{Sketch of the proof of
  Theorem~\ref{thm:cellNObound}}\label{sec:expandedSketch}
In this section, we will sketch the proof of
Theorem~\ref{thm:cellNObound}.  The proof will use a multiscale
argument to construct a pseudo-orientation of $A$; this argument was
inspired by unpublished work of Larry Guth.

In unpublished work, Guth proved a superlinear bound on the problem in
Corollary~\ref{cor:YoungsProblem}.  He showed that if
$T\in \Ccell_d(\tau)$ is a cycle in the unit grid $\tau$ in $\R^N$,
then $\FV(T)\lesssim (\log \FV(2 T)) \FV(2 T)$ \cite{GuthPersonal}.
His argument used a multiscale argument to bound fillings of $T$ based
on approximations of $2T$ at many scales.  Guth used a similar
argument to prove the Perpendicular Pair Inequality in
\cite{GuthContraction}; the following proposition, obtained in
collaboration with Guth, applies these arguments to
nonorientability.

\begin{prop}[{see \cite[\S8]{GuthContraction}}]\label{prop:VlogVbound}
  For every $\nu,d,N>0$, there is a $c>0$ such that if
  $A\in \Ccell_d(\tau;\Z_\nu)$ is a mod-$\nu$ cellular cycle in the unit
  grid in $\R^N$, then 
  $$\NO(A)\le c\mass A(\log \mass A).$$
\end{prop}
\begin{proof}
  We bound $\NO(A)$ by breaking it into contributions from different
  scales.  We will construct a sequence of cycles
  $P_k,P_{k-1},\dots, P_1$ approximating $A$ at finer and finer
  scales, and show that passing from $P_k$ to $P_{k-1}$ adds up to
  $\mass A$ to the nonorientability of $A$.  Since there are
  logarithmically many scales, we obtain the desired bound.

  Let $k$ be such that $2^{dk} \gg \mass A$ and
  $k\lesssim \log \mass A$.  Let $X_i=\R^N\times [2^{i},2^{i+1}]$ and
  let $\Sigma$ be the QC complex introduced in
  Section~\ref{sec:defNot}, which subdivides $\R^N\times [1,\infty)$
  into dyadic cubes so that $X_i$ is tiled by cubes of side length
  $2^i$.  Then $A\times [1]$ is a cellular cycle in $\Sigma$.  A
  pseudo-orientation of $A\times [1]$ projects to a pseudo-orientation
  of $A$, so $\NO(A)\le \NO(A\times [1])$.

  Let $P$ be a deformation operator as in Lemma~\ref{lem:simultFF}
  approximating all chains of the form $A\times [2^i,2^{i+1}]$ or
  $A\times [2^i]$ by cellular chains in $\Sigma$.  Then for each $i$,
  the cycle 
  $$P_i=P(A\times [2^i]) \in \Ccell_{d}(\Sigma;\Z_\nu)$$
  approximates $A$ at scale $2^i$ and satisfies
  $\mass P_i\lesssim \mass A$.  Similarly, for each $i$, the chain
  $$V_i=P(A\times [2^i,2^{i+1}])\in \Ccell_{d+1}(\Sigma;\Z_\nu)$$
  forms an chain with boundary $P_{i+1}-P_i$ and
  $\mass V_i\lesssim 2^i\mass A$.  We will use the $V_i$ to bound
  $\NO(A)$.

  Note that since $A\times [1]$ is already cellular,
  $P_0=A\times [1]$.  Furthermore, because $P_k$ is a cellular
  $d$-cycle in $X_{k}$ with $\mass P_k \lesssim \mass A$ and each
  $d$-cell in $X_{k}$ has volume on the order of $2^{dk}$ (much bigger
  than $\mass A$), we have $P_k=0.$ It follows that
  $$\sum_{i=0}^{k-1} \partial V_i =P_k-P_0=-A\times [1]$$
  and 
  $$\NO(A)\le \sum_{i=0}^{k-1} \NO(\partial V_i).$$

  For each $i$, we construct a pseudo-orientation of $\partial V_i$ by
  decomposing $V_i$ as a sum of cells.  Let
  $W_i\in \Ccell_{d+1}(\Sigma)$ be a chain with integer coefficients
  between $-\frac{\nu}{2}$ and $\frac{\nu}{2}$ such that
  $W_i\equiv V_i\pmod{\nu}$.  Then $\mass W_i=\mass V_i$.  Let
  $R_i=\partial W_i\in \Ccell_{d}(\Sigma)$.  This is an integral cycle
  and $R_i\equiv \partial V_i\pmod{\nu}$, so $R_i$ is a pseudo-orientation
  of $\partial V_i$.

  We can estimate the mass of $R_i$ by counting the number of cells in
  $W_i$.  By Lemma~\ref{lem:simultFF}, we have
  $$\mass W_i=\mass V_i \lesssim 2^i\mass A.$$  
  Since $W_i$ is a sum of cubes of side length $\sim 2^i$ and
  $(d+1)$-volume $\sim 2^{i(d+1)}$, we have
  \[\|W_i\|_1\sim \frac{2^i\mass A}{2^{i(d+1)}}\sim 2^{-id}\mass A.\]
  The boundary of each of these simplices has volume $\sim 2^{id}$, so
  $\mass R_i \lesssim \mass A$, and
  $$\NO(A)\le \sum_{i=0}^{k-1} \mass R_i\lesssim \mass A (\log \mass A)$$
  as desired.
\end{proof}

This is very close to the desired linear bound, but improving this
argument to a linear bound is difficult.  The main obstacle is that
$\NO(A)$ may have large contributions from a wide range of scales.
The bound in the proposition comes from showing that each scale can
contribute at most $\mass A$ to the nonorientability and that there
are logarithmically many scales.  To prove
Theorem~\ref{thm:cellNObound}, we must show instead that the total
contribution from all scales is bounded by $\mass A$.

In the introduction, we constructed an example of a surface $C_k$ that
contains crosscaps of many scales.  If we rescale $C_k$ to get a
cellular surface with area of order $k10^{2k}$, the result typifies
some of the difficulties we will encounter. The rescaled surface
contains many crosscaps at scales $1,10,\dots, 10^k$, and each scale
contributes roughly $10^{2k}$ to the nonorientability.  By varying the
number of crosscaps added at each scale, we can construct a wide
variety of examples with varying areas and nonorientabilities.  In
order to prove Theorem~\ref{thm:cellNObound}, we must show that any
such surface can be decomposed into simple pieces.

\subsection{Decomposing cycles into uniformly rectifiable pieces}\label{sec:decompSketch}
The first part of the proof of Theorem~\ref{thm:cellNObound}
decomposes a cycle in $\R^N$ into a sum of cycles with uniformly
rectifiable supports; for the full statement, see
Theorem~\ref{thm:URdecomp}.  This decomposition breaks complicated
surfaces into ``simple'' pieces; in particular, it breaks the example
$C_k$ above into the initial cube $C_0$ and a collection of projective
planes of different scales.

Recall that a set $E\subset \R^n$ is
\emph{Ahlfors $d$-regular} (or simply \emph{$d$-regular}) with
regularity constant $c$ if for any $x\in E$ and any $0<r<\diam E$,
\[c^{-1}r^d\le |E\cap B(x,r)|\le c r^d.\]
(Here and in the rest of the paper, we use $|E|$ to denote the
Hausdorff $d$-measure of a subset of $\R^N$.)  We say that $E$ is
\emph{$d$-rectifiable} if it can be covered by countably many
Lipschitz images of $\R^d$.

Uniform rectifiability is a quantitative version of rectifiability
that measures the size and complexity of the Lipschitz images that
cover $E$.  There are several ways to define uniform rectifiability, and
we will primarily use the following definition:
\begin{defn}
  A set $E\subset \R^N$ is \emph{uniformly $d$-rectifiable} if there
  is a $c$ such that $E$ is
  Ahlfors $d$-regular with regularity constant $c$ and, for all $x\in E$
  and $0<r<\diam E$, there is a $c$-Lipschitz map $f:B_d(r)\to \R^n$
  such that
  \[|f(B_d(r))\cap E\cap B(x,r)|\ge c^{-1}r^d,\]
  where $B_d(r)$ is the ball of radius $r$ in $\R^d$.
\end{defn}
This is also known as having \emph{big pieces of Lipschitz images}
(BPLI).  We call $c$ the \emph{uniform rectifiability (UR) constant}
of $E$.  Note that this definition is scale-invariant, so if $E$ is
uniformly rectifiable, any scaling of $E$ is uniformly rectifiable
with the same constant.  

Then, for example, the surfaces $C_i$ constructed in the introduction
are all uniformly rectifiable, but with a constant depending on $i$;
as $i$ grows, the area of the sets grows and their geometry becomes
more complicated.  Indeed, if $A$ is a cellular cycle, then $\supp A$
is a finite union of unit cubes, so it is automatically uniformly
rectifiable, albeit with a constant depending strongly on $A$.  The
important feature of Theorem~\ref{thm:URdecomp} is that it decomposes
$A$ into a sum of pieces with UR constants that are \emph{independent}
of $A$.

The proof of Theorem~\ref{thm:URdecomp} relies on results of David and
Semmes on quasiminimizing sets.  Roughly, a set $E\subset \R^N$ is
\emph{quasiminimizing} if compactly supported deformations do not
locally decrease the volume of the set too much.  (For a more detailed
definition, see Sec.~\ref{sec:QMprelims}.)  In \cite{DSQuasi}, David
and Semmes prove that quasiminimizing sets are uniformly rectifiable.
Consequently, for every $k>1$, there is an $\epsilon$ such that if a
set $E$ is not $\epsilon$--uniformly rectifiable, there is a compactly
supported deformation that locally decreases its volume by a factor of
at least $k$.

Let $A$ be a cellular cycle, let $E=\supp A$ and suppose that
$f\from \R^N\to \R^N$ is such a deformation of $E$.  That is, there is a set
$S\subset \R^N$ such that $f$ is the identity map outside $S$,
$f(S)\subset S$, and
$$|f(S\cap E)|\le \frac{|S\cap E|}{k}.$$
Generally, there will be many possible deformations to choose from; we
choose one so that $\diam S$ is close to minimal.  This ensures that
$E$ is quasi-minimizing (and thus uniformly rectifiable) on scales
smaller than the diameter of $S$.  Then $M=A-f_\sharp(A)$ is a cycle
such that
$$\supp M \subset (S\cap E)\cup f(S\cap E).$$
Since $E$ is uniformly rectifiable on scales smaller than the diameter
of $S$, the set $S\cap E$ is uniformly rectifiable (possibly with
worse constants).  To show the uniform rectifiability of $\supp M$, we
need to control $f(S\cap E)$.

Unfortunately, although $|f(S\cap E)|$ is small, we have poor control
over the geometry of $f(S\cap E)$, especially near the boundary of
$S$.  We thus prove a slight strengthening of David and Semmes's
theorem (Proposition~\ref{prop:QMtoUR}).  This proposition allows us
to choose $S$, $f$, $\epsilon$ depending on $k$ so that if
$$S'=\{x\in \R^N\mid d(x,S)<\epsilon \diam S\},$$
then
$$|f(S'\cap E)|\le \frac{|S'\cap E|}{k}.$$
This lets us adjust $f$ so that $f(S'\cap E)$ is uniformly
rectifiable.  Specifically, we let $\tau_{S'}$ be a Whitney cube
decomposition of $S'$ and use Lemma~\ref{lem:FFsets} to approximate
$f(S'\cap E)$ by a union of cells of $\tau_{S'}$ and approximate
$f_\sharp(A)$ by a cellular chain $P(f_\sharp(A))$.  In
Lemma~\ref{lem:URapprox}, we show that $\supp P(f_\sharp(A))$ is
contained in a uniformly rectifiable set.

The result of this is a cycle $M=A-P(f_\sharp(A))$ whose support is
contained in a uniformly rectifiable set.  Furthermore, $f_\sharp(A)$
has substantially smaller support than $A$; we have
$$|\supp A|-|\supp f_\sharp(A)|\gtrsim |\supp M|.$$
Letting $A_1=P(f_\sharp(A))$, 
we can repeat this process inductively to construct a sequence of
cycles $A_2,A_3,\dots$ such that $A_i$ is a deformation of $A_{i-1}$
and $\vol^d \supp A_i$ is strictly decreasing.  Since the $A_i$ are
cellular, this sequence terminates and we can write $A=\sum_i M_i$,
where $M_i=A_{i-1}-A_i$ and $|\supp A|\lesssim \sum_i |\supp M_i|$.

It is helpful to consider the result of applying this process to the
surface $C_k$ constructed in the introduction.  Let $A$ be a copy of
$C_k$, scaled to be a cellular cycle.  When $r$ is small, sets like
$\supp A\cap B_r(x)$ contain few or no crosscaps and are
quasiminimizing.  As we increase $r$, the intersections
$\supp A\cap B_r(x)$ will contain more and larger crosscaps, until
finally $r$ is large enough that $\supp A\cap B_r(x)$ is no longer
quasiminimizing.  At this point, there is an $x_1\in \supp A$ and a
deformation supported in $B_r(x_1)$ that replaces
$\supp A\cap B_r(x_1)$ with a substantially smaller minimal surface.
Let $A_1$ be the result of deforming $A$ and let $M_1=A_0-A_1$.  Then
$M_1$ contains most of the crosscaps in $\supp A\cap B_r(x_1)$.

In fact, there will be many $x_i$ such that $\supp A\cap B_r(x_i)$ is
not quasiminimizing.  We can repeat this process in each such ball to
remove more and more small crosscaps from $A$, eventually obtaining a
cycle $A_k$ with most of its small crosscaps removed.  Without those
small crosscaps, $A_k$ is quasiminimizing at scale $r$, so we can
increase $r$ again until $\supp A_k\cap B_r(x)$ is no longer
quasiminimizing.  We repeat this process roughly $k$ times, each time
removing larger and larger crosscaps from $A$, until finally, $A$ is
quasiminimizing at all scales.  

This decomposition is like the construction of $A$ in reverse.  We
originally constructed $A$ by starting with a cube, then adding
crosscaps at all scales, starting with the largest crosscaps and
ending with the smallest.  To decompose $A$, we reverse that process,
removing the crosscaps from smallest to largest.

\subsection{Bounding nonorientability of uniformly rectifiable cycles}\label{sec:URcaseSketch}

The second part of the proof of Theorem~\ref{thm:cellNObound} is to
bound the nonorientability of cycles supported on uniformly
rectifiable sets.  Specifically, we claim that
\begin{prop}\label{prop:URcase}
  If $A\in \Ccell_{d}(\tau;\Z_{\nu})$ and $\supp A$ is contained in a
  $d$-dimensional uniformly rectifiable set $E$, then
  \[\NO(A)\lesssim |E|,\]
  with implicit constant depending only on $\nu$, $N$ and the uniform
  rectifiability constant of $E$.
\end{prop}
The main idea of the proof is to combine the methods of
Prop.~\ref{prop:VlogVbound} with a corona decomposition of the support
of $A$.

Recall that in Prop.~\ref{prop:VlogVbound}, we constructed a
pseudo-orientation of a cycle $A$ by approximating cycles of the form
$A\times [2^i]$ in a complex $\Sigma$ that decomposes
$\R^N\times [1,\infty)$ into dyadic cubes.  We let $P$ be a
deformation operator for $\Sigma$ as in Lemma~\ref{lem:simultFF}.  For
each $i$, we constructed an approximation $P_i=P(A\times [2^i])$
consisting of cubes of side length $2^i$, then connected these
approximations by chains $V_i=P(A\times [2^i,2^{i+1}])$.  The chain
$V_i$ has boundary equal to $P_{i+1}-P_i$, and we constructed a
pseudo-orientation of $\partial V_i$ by choosing random orientations
on the cells that make up $V_i$.  The mass of this pseudo-orientation
is bounded by the total volume of the boundaries of all of the cells
of the $V_i$, and by counting cells, we find that
\begin{equation}\label{eq:changeBound}
  \NO(P_i-P_{i+1})\lesssim \mass A
\end{equation}

When $P_{i+1}$ has simpler topology than $P_i$, this estimate is
reasonably accurate.  For example, if $A$ is covered by crosscaps of
diameter roughly $2^i$, those crosscaps will appear in $P_i$ but not
in $P_{i+1}$.  The difference $P_i-P_{i+1}$ is nonorientable, and
\eqref{eq:changeBound} is sharp.

The fact that makes Proposition~\ref{prop:URcase} possible is that if
$A$ is uniformly rectifiable, then there are many scales and locations
on which $A$ is close to a plane.  When this happens, one
approximation looks very similar to another approximation, so we can
skip over intermediate scales.  

For example, suppose that there is a smooth projective plane
$K\subset \R^N$ such that $[K]$ and $A$ are close.  Specifically,
suppose that there is some $j\ll k$ such that for $i\ge j$,
$P_i = P([K]\times [2^i])$ and suppose that $\area K\sim 2^{2k}$,
$\diam K\sim 2^k$.  One such $K$ and $A$ could be obtained by letting
$K$ be an embedded projective plane, letting $K'$ be the result of
replacing many small discs in $K$ by small crosscaps, and letting
$A=[K']$.

We can construct a pseudo-orientation of $P_{k}-P_{j}$ by lifting
$[K]$ to a chain with integer coefficients.  Let $\gamma$ be a simple
closed curve such that $K\setminus \gamma$ is homeomorphic to a disc
$D$.  By the systolic inequality for the projective plane, we may
suppose $\ell(\gamma)\lesssim \sqrt{\area K}\sim 2^k$.  This disc has
a fundamental class $[D]\in \Ccell_2(\Sigma;\Z)$ such that
$[D]\equiv [K]\pmod{2}$ and $\partial [D]=2[\gamma]$.  Since
$[D]\times [2^i]$ is congruent to $[K]\times [2^i]$, we can
approximate $[D]$ to obtain chains
$$D_i=P([D]\times [2^i])\in\Ccell_2(\Sigma;\Z)$$ 
such that $D_i\equiv P_i\pmod{2}$ and chains
$$H_i=P([D]\times [2^i,2^{i+1}])\in\Ccell_3(\Sigma;\Z)$$ 
such that 
\begin{align*}
  \partial H_i&=D_{i+1}-D_{i}+P(\partial[D]\times [2^i,2^{i+1}])\\
  &=D_{i+1}-D_{i}+P(2[\gamma]\times [2^i,2^{i+1}])\\
  &\equiv P_{i+1}-P_i\pmod{2}.
\end{align*}
Then $R=\sum_{i=j}^{k-1}\partial H_i$ is a pseudo-orientation of
$P_{k}-P_j$ and 
\begin{align}
  \notag  \NO(P_{k}-P_j)&\le \mass R \lesssim \mass D_j+\mass D_{k}+\mass P(2[\gamma]\times [2^j,2^{k}])\\ 
  \label{eq:noChangeBound}
                          &\lesssim 2^{2k}+2^{2k}+2^{2k}\lesssim \mass
                            A.
\end{align}

If $A$ only has topological features at a few scales, then we can
alternate between these two estimates, using \eqref{eq:changeBound} at
scales where $P_{i-1}$ and $P_i$ are different and
\eqref{eq:noChangeBound} for ranges of scales where the $P_i$'s do
not change very much.  If the number of times we need to use
\eqref{eq:changeBound} is bounded, we get a linear bound on the
nonorientability.

The main problem with this approach is that even if $A$ is uniformly
rectifiable, it can have features at infinitely many different scales.
To construct such a set, we can start with a cube of side length $2^k$
and replace half of the faces of the cube by crosscaps of scale $2^k$.
If we cover half of the remainder with crosscaps of scale $2^{k-1}$,
then cover half of the remainder with even smaller crosscaps, and so
on, the result is uniformly rectifiable, but has complicated topology
at all scales.

Nevertheless, a uniformly rectifiable set cannot be complicated
everywhere and at all scales.  This idea can be quantified by using a
\emph{corona decomposition} of $E$.  An $(\eta,\theta)$ corona
decomposition of $E$ partitions $E\times \R$ into a set $\cB$ of bad
cubes and a set $\cF$ of stopping-time regions.  (See
Section~\ref{sec:URprelims} for more details.)  The number and size of
the bad cubes and stopping-time regions is bounded.  On bad cubes, we
have little control over the geometry of $E$, but if $S\in \cF$ is a
stopping-time region, there is a Lipschitz graph $\Gamma(S)$ with
Lipschitz constant at most $\eta$ such that for every $(x,t)\in S$,
$d(x,\Gamma(S))\lesssim \theta t$.

We will use this decomposition of $E\times \R$ to decompose $A$.  Let
$k$ be as in Prop.~\ref{prop:VlogVbound}, so that $2^k\sim \diam E$
and $P(A\times [2^k])=0$.  Let $\extend{A}=A\times [1,2^{k}]$.  Then
$\supp \extend{A}\subset E\times \R$, so we can write
\[\extend{A}=\sum_{Q\in \cB} \extend{A}_Q+\sum_{S\in \cF} \extend{A}_S\]
where $\extend{A}_Q$ and $\extend{A}_S$ are the restrictions of
$\extend{A}$ to the corresponding bad cubes and stopping-time regions
in $E\times \R$.  Then
$$P(\partial \extend{A})=P(A\times[2^{k}])-P(A\times [1])$$
and 
$$\sum_{Q\in \cB} P(\partial \extend{A}_Q)+\sum_{S\in \cF} P(\partial \extend{A}_S)=-A\times [1].$$
So
$$\NO(A)\le  \sum_{Q\in \cB} \NO(P(\partial \extend{A}_Q))+\sum_{S\in \cF} \NO(P(\partial \extend{A}_S)).$$

When $Q\in \cB$ is a bad cube, $P(\partial \extend{A}_Q)$ will consist
of boundedly many cells of $\Sigma$, so its nonorientability will be
bounded.  When $S\in \cF$ is a stopping-time region, $S$ will be
approximated by a Lipschitz graph.  This will induce a pseudo-orientation
on $P(\partial \extend{A}_S)$ and give a bound on its
nonorientability

In total, the stopping-time regions will contribute roughly $\mass A$
to $\NO(A)$, and the bad cubes are sparse enough and small enough that
they will also contribute roughly $\mass A$.  Adding these together
will give the desired linear bound on $\NO(A)$.

\section{Decomposing cycles into uniformly rectifiable pieces}
\label{sec:URdecomp}

In this section, we prove Theorem~\ref{thm:URdecomp}, which states
that any cycle in $\R^N$ can be decomposed into a sum of cycles
supported on uniformly rectifiable sets.  In Sec.~\ref{sec:URprelims},
we state some definitions and theorems about uniform rectifiability to
be used later; one can find a full exposition in \cite{DSAnalysis}.
The main tool we use to construct this decomposition is a result of
David and Semmes \cite{DSQuasi} stating that quasiminimizing sets are
uniformly rectifiable.  We will define quasiminimizing sets and prove
a slight generalization of their theorem in Sec.~\ref{sec:QMprelims},
then use this generalization to construct the desired decomposition in
Sec.~\ref{sec:pfURdecomp}.

\subsection{Uniform rectifiability}\label{sec:URprelims}

In this section, we review some definitions and results concerning
uniformly rectifiable sets that we will use in the rest of this paper
Throughout the rest of this paper, if $E\subset \R^N$, we will use
$|E|$ to denote its Hausdorff $d$-measure.  As noted above, all
implicit constants will be taken to depend on $d$ and $N$.

In Sec.~\ref{sec:decompSketch}, we gave a definition of uniform
rectifiability in terms of big pieces of Lipschitz images.  Another
way of defining uniformly rectifiable sets uses cubical patchworks and
corona decompositions.  We say that a collection of sets $\Gamma$ is a
partition of $E$ if the elements of $\Gamma$ are disjoint and their
union is all of $E$.  A cubical patchwork, also known as a set of
Christ cubes, for $E$ is a collection of partitions of $E$ into
\emph{pseudocubes} which generalizes the usual decomposition of $\R^d$
into dyadic cubes.
\begin{defn}\label{def:patchwork} 
  Let $E$ be an Ahlfors $d$-regular set with
  $2^{k}< \diam E\le 2^{k+1}$ for some $k\in \Z$.  A cubical patchwork for $E$ is
  a collection $\{\Delta_i\}_{i=-\infty}^k$ of partitions of $E$ with
  the following properties:
  \begin{enumerate}
  \item $\Delta_k=\{E\}$.
  \item Each $Q\in \Delta_i$ satisfies 
    $\diam Q\sim 2^i$ and $|Q|\sim 2^{di}$.
  \item If $Q\in \Delta_i$ and $Q'\in \Delta_j$, with $i\le j$  then
    either $Q$ and $Q'$ are disjoint or $Q\subset Q'$.  
  \item \label{item:bdary} For any $Q\in \Delta_i$ and any $r>0$, let
    \[\partial Q(r)=\{x\in Q\mid d(x, E\setminus Q)\le r\}\cup \{x\in  E\setminus Q\mid d(x, Q)\le r\}.\]
    There is a $C>1$ such that for any $t>0$,
    \begin{equation}\label{eq:PatchworkSmallBdry}
      |\partial Q(t 2^i)|\le Ct^{1/C}2^{id}
    \end{equation}
    for each $Q\in \Delta_i$.
  \end{enumerate}
  We call the elements of $\Delta_i$ \emph{pseudocubes}, and we let
  \[\Delta=\bigcup_i \Delta_i.\]
\end{defn}

If $Q\in \Delta_i$, we say that any set $Q\in \Delta_{i-1}$ with
$Q'\subset Q$ is a \emph{child} of $Q$ and that any set
$Q\in \Delta_{j}$ with $Q'\subset Q$ and $j<i$ is a \emph{descendant}
of $Q$.

We call the constants in the definition the \emph{patchwork constants}
of $\Delta$.  David \cite{DavidMorceaux} showed that any Ahlfors
$d$-regular set in $\R^n$ has a cubical patchwork whose patchwork
constants are functions of the regularity constant of the set.  Christ
\cite{ChristTb} generalized this result to metric-measure spaces.

Condition \ref{item:bdary} above is a little subtle.  It implies that
the boundary of a pseudocube is small.  One consequence is the
following lemma (\cite[Lemma~I.3.5]{DSAnalysis}):
\begin{lemma}\label{lem:DScenters}
  There is a $C>1$ depending only on $d$, $n$, and the regularity
  constant for $E$ such that for each cube $Q\in \Delta$ there is a
  center $c(Q)\in Q$ such that
  \[d(c(Q),E\setminus Q)\ge C^{-1}\diam Q.\]
\end{lemma}

It does not, however, imply that the boundary of a pseudocube is
very smooth.  In fact, the condition only guarantees that the
Hausdorff dimension of the boundary is strictly less than $d$:
\begin{lemma} \label{lem:pseudocubeBoundary}
  Let $Q\in \Delta^i$ be a pseudocube in a cubical patchwork for a
  set $E$ with $\dim(E)=d$ and let $\partial Q\subset E$ be the
  boundary of $Q$ relative to $E$ (i.e., the intersection of the
  closures of $Q$ and of $E\setminus Q$).  For any $1>t>0$, we can
  cover $\partial Q$ with $\sim C t^{1/C-d}$ balls of radius $t
  2^{i}$, where $C>1$ is as in Def.~\ref{def:patchwork}.
\end{lemma}
\begin{proof}
  Consider $S=\partial Q(t2^i)$.  This contains the $t
  2^i$-neighborhood of $\partial Q$ and has $|S|\le
  Ct^{1/C}2^{id}$.  Let $M$ be a maximal set of points of $\partial
  Q$ spaced a distance $t 2^{i}$ apart.  Then the balls of radius
  $t 2^{i}$ centered at the points of $M$ cover $\partial Q$, and
  the balls of radius $t 2^{i-1}$ are disjoint and
  contained in $S$.  By Ahlfors regularity,
  \[\#(M)t^d 2^{id}\lesssim |S|\le Ct^{1/C}2^{id},\]
  so $\#M\lesssim C t^{1/C-d}$ as desired.
\end{proof}

This makes it difficult to construct chains supported on pseudocubes,
because the boundary of a pseudocube is generally unrectifiable.  We
will avoid this problem by considering the case that $E$ is a union of
$d$-cells of the unit grid $\tau$.  When this is the case, we can find
a patchwork such that the closure of any sufficiently large pseudocube
is a union of $d$-cells. 
\begin{lemma}\label{lem:gridPatchwork}
  If $E$ is a Ahlfors $d$-regular set that is a union of $d$-cells of
  $\tau$ and $2^{k}< \diam E\le 2^{k+1}$ for some $k\in \Z$, then
  there is a cubical patchwork $\{\Delta_i\}$ of $E$ such that if
  $Q\in \Delta_i$ and $i>0$, then  $\Clos(Q)$  is a
  union of $d$-cells of $\tau$.  Furthermore, the patchwork constants depend only on $d, N$, and the 
  regularity constant of $E$.
\end{lemma}
\begin{proof}
  In this proof, all our implicit constants will depend on $d, N$, and
  the Ahlfors regularity constant of $E$.  

  Enumerate the $d$-cells of $E$ as $D_1,\dots,D_m$ and let
  \[D'_i=D_i\setminus \bigcup_{j=1}^{i-1} D_j,\]
  so that the $D'_i$'s form a partition of $E$ and $\Clos(D'_i)=D_i$
  for all $i$.  For each $i=1,\dots, n$, let $x_i$ be the barycenter
  of $D_i$, and if $S\subset E$, let
  \[\delta_S=\bigcup_{x_i\in S} D'_i.\]
  Let $\Gamma=\{\Gamma_i\}_{i\in \Z}$ be a cubical patchwork for $E$.
  We can choose $\Gamma$ so that its patchwork constants depend only on $d, N$, and the Ahlfors
  regularity constant of $E$.  For each $i=0,\dots,k$, let
  \[\Delta_i=\{\delta_Q \mid Q\in \Gamma_i, \delta_Q\ne \emptyset\},\] and for each $i<0$, let $\Delta_{i}$ be the
  partition of $E$ that divides each $d$-cell of $E$ into $2^{-id}$
  cubes of side length $2^i$.  We claim that the $\Delta_i$'s satisfy
  Def.~\ref{def:patchwork}.  Properties 1 and 3 are easy to check, and
  properties 2 and 4 clearly hold for $\Delta_i$ when $i<0$.  It
  remains to check that the $\Delta_i$ satisfy properties 2 and 4 when
  $i\ge 0$.

  First, we check property 2.  Suppose that $i\ge 0$ and $Q\in \Gamma_i$ is a pseudocube such that
  $\delta_Q\ne \emptyset$.  Let $R=\diam Q$.  Let $x=c(Q)$ be the
  center of $Q$ as in Lemma~\ref{lem:DScenters}, and let $C>1$ be as
  in Lemma~\ref{lem:DScenters}.  Note that $C$ depends only on $d, N$,
  and the Ahlfors regularity constant of $E$.  Suppose that $R\le
  2C\sqrt{N}\sim 1$.  Since $\delta_Q\ne \emptyset$, it contains at least
  one cell of $\tau$, so $|\delta_Q|\ge 1$ and $\diam
  \delta_Q\ge 1$.  On the other hand,
  \begin{equation}\label{eq:diamBound}
    \diam \delta_Q\le R+\sqrt{N}\lesssim 1,
  \end{equation}
  so $|\delta_Q|\sim R^d$ and $\diam \delta_Q\sim R$, verifying
  property 2.

  We thus assume that $R> 2C\sqrt{N}$ and claim that
  \begin{equation}\label{eq:deltaQbounds}
    B(x,C^{-1}R/4)\cap E \subset \delta_Q\subset B(x,2R)\cap E.
  \end{equation}
  If a $d$-cell of $E$ intersects $B(x,C^{-1}R/4)$, then its center
  lies inside $B(x,C^{-1}R/2)$.  By Lemma~\ref{lem:DScenters},
  $B(x,C^{-1}R/2)\cap E\subset Q,$ so 
  \[B(x,C^{-1}R/4)\cap E \subset \delta_Q.\]
  On the other hand, if a $d$-cell of $E$ lies in $\delta_Q$, its
  center lies in $Q$.  Since $Q\subset B(x,R)$, we have
  $\delta_Q\subset B(x,2R)$ as desired.
  Equation~\eqref{eq:deltaQbounds} implies property 2 by the Ahlfors
  regularity of $E$.

  To show property 4, let $C'>1$ be the constant in
  \eqref{eq:PatchworkSmallBdry} for the patchwork $\{\Gamma_i\}$.  Let
  $Q\in \Gamma_i$, $\delta_Q\in \Delta_i$, and $t>0$.  We have
  \[\partial Q(r) =\{x\in E\mid d(x, E\setminus Q)\le r \text{ and } d(x, Q)\le r\};\]
  let 
  \[\partial \delta_Q(r)=\{x\in E\mid d(x, E\setminus \delta_Q)\le r \text{
    and } d(x, \delta_Q)\le r\}.\] If $x\in \partial \delta_Q(r)$, then there
  is some $y\in \delta_Q$ such that $d(x,y)\le r$.  Since $y$ is
  contained in a $d$-cell $D\subset \delta_Q$, 
  \[d(x,\delta_Q)\le d(x,x_D)\le d(x,y)+\sqrt{N}\le r+\sqrt{N}.\]
  Likewise, 
  \[d(x,E\setminus \delta_Q)\le r+\sqrt{N},\]
  so $\partial \delta_Q(r)\subset \partial Q(r+\sqrt{N})$.  In particular,
  \begin{equation}\label{eq:boundary}
    |\partial \delta_Q(t 2^i)|\le C'(t+2^{-i}\sqrt{N})^{1/C'}2^{id},
  \end{equation}
  and if $t \ge 2^{-i}$, then
  \[|\partial \delta_Q(t 2^i)|\lesssim C' t^{1/C'}2^{id}.\]

  On the other hand, when $t \le 2^{-i}$, we can bound $|\partial
  \delta_Q(t 2^i)|$ by counting the number of cells that intersect
  $\partial \delta_Q(1)$.  Any cell of $\tau$ that intersects
  $\partial \delta_Q(1)$ is completely contained in $\partial
  Q(3\sqrt{N})$, so if 
  \[K=\#\{D\in \tau^{(d)}\mid D \cap\partial \delta_Q(1)\ne \emptyset\},\]
  then 
  \[K\le |\partial Q(3\sqrt{N})| \lesssim 2^{id-i/C'}.\]
  Consequently, if $\epsilon \le 1$, then $\partial
  \delta_Q(\epsilon)$ is a subset of the $\epsilon$-neighborhood of
  the boundary of at most $K$ $d$-cells.  This neighborhood has
  Hausdorff measure $\lesssim K \epsilon$, so if $\tau \le 2^{-i}$,
  then
  \[|\partial \delta_Q(\tau 2^i)|\lesssim \tau 2^i K\lesssim
  \tau^{1/C} 2^{id}\]
  as desired.
\end{proof}
If $E$, $k$, and $\{\Delta_i\}_{i=0}^k$ are as in the lemma, we will
refer to $\Delta=\bigsqcup_i \Delta_i$ as a \emph{cellular cubical
  patchwork} for $E$.

David and Semmes used cubical patchworks in an alternative definition
of uniform rectifiability.  To state this definition, we first need to
define coronizations.  Our definition is taken from \cite{DSAnalysis}.

\begin{defn}\label{def:coronization}
  Let $E\subset \R^N$ be a $d$-dimensional Ahlfors regular set,
  equipped with a cubical patchwork $\Delta$.  A \emph{coronization}
  of $E$ is a partition of $\Delta$ into \emph{bad cubes} and
  \emph{stopping-time regions}.  More precisely, it is a triple
  $(\cB,\cG,\cF)$ such that $\cB$ (the set of bad cubes) and $\cG$
  (the set of good cubes) partition $\Delta$ into two disjoint sets
  and $\cF$ is a collection of subsets of $\cG$, called
  stopping-time regions.  These sets have the following properties:
  \begin{enumerate}
  \item $\cB$ satisfies a Carleson packing condition.
  \item The elements of $\cF$ are disjoint and their union is $\cG$.
  \item Each $S\in \cF$ is coherent.  This entails three properties.
    First, every $S$ has a unique maximal element $Q(S)\in S$ which
    contains every element of $S$.  Second, if $Q\in S$, then $S$
    contains every $Q'\in \Delta$ such that $Q\subset Q'\subset Q(S)$.
    Third, if $Q\in S$, then either all the children of $Q$ lie in $S$
    or none of them do.
  \item The set of maximal cubes $Q(S)$, $S\in \cF$, satisfies a
    Carleson packing condition.
  \end{enumerate}
\end{defn}

A Carleson packing condition bounds the density of a set of
pseudocubes.  Specifically, we say that $\mathcal{A}\subset \Delta$
satisfies a \emph{Carleson packing condition} if there is a $c>0$ such
that for every $Q\in\Delta$,
\[\mathop{\sum_{Q'\in \cA}}_{Q'\subset Q} |Q'| \le c|Q|.\]
For example, for any $i$, $\Delta_i\subset \Delta$ satisfies a
Carleson packing condition, and if $x\in E$, then
\[\mathcal{A}_x=\{Q\in \Delta\mid x\in Q\}\]
satisfies a Carleson packing condition.

The term ``stopping-time region'' comes from the way that
coronizations are usually constructed.  Many coronizations are
constructed by finding ``good'' pseudocubes, then repeatedly
subdividing them until they stop being good.  The set of good
descendants of a particular pseudocube then forms a stopping-time
region.  In our case, stopping-time regions
correspond to parts of $E$ which are close to a Lipschitz graph.

\begin{defn}
  If $V$ is a subspace in $\R^N$,  $V^\perp$ is its orthogonal complement,
  and $h:V\to V^\perp$ is a Lipschitz function, we say that 
  \[\{x+h(x)\mid x\in V\}\]
  is the graph of $h$.  We call sets of this form \emph{Lipschitz
    graphs}.
\end{defn}

\begin{defn}\label{def:coronaDecomp}
  Let $E\subset \R^N$ be a $d$-dimensional Ahlfors regular set,
  equipped with a cubical patchwork $\Delta$.  We say that $E$ admits
  a \emph{corona decomposition} if for every $\eta, \theta>0$, there is a
  coronization $(\cB,\cG,\cF)$ of $E$ such that for each $S\in \cF$
  there exists a Lipschitz graph $\Gamma(S)$ with Lipschitz constant
  $\le \eta$ such that 
  \[d(x,\Gamma(S))\le \theta \diam Q\]
  for every $x\in E$ such that $d(x,Q)\le \diam Q$ and every $Q\in S$.
\end{defn}
Note that the constants in Carleson packing condition may depend on
$\eta$ and $\theta$.

David and Semmes proved that this property is equivalent to uniform rectifiability:
\begin{thm}[{\cite{DSAsterisque}}]\label{thm:coronaUR} 
  Suppose $E$ is a $d$-dimensional Ahlfors regular set in $\R^N$ with
  a cubical patchwork $\Delta$.  Then $E$ is uniformly rectifiable if
  and only if it admits a corona decomposition with respect to
  $\Delta$.  Furthermore, if $E$ is uniformly rectifiable, then the
  implicit constants of the corona decomposition depend only on
  $\eta$, $\theta$, $d$, $N$, the patchwork constants of $\Delta$, and
  the UR constant of $E$.
\end{thm}

If the patchwork $\Delta$ in Definition~\ref{def:coronization} or
\ref{def:coronaDecomp} is cellular, we call the resulting corona
decomposition or coronization a \emph{cellular corona decomposition}
or a \emph{cellular coronization}.  

\subsection{Quasiminimizing sets}\label{sec:QMprelims}

A quasiminimizing set, or quasiminimizer, is a set whose volume cannot
be reduced too much by a small deformation.  David and Semmes showed
that the solutions to many minimization problems are uniformly
rectifiable by showing that quasiminimizers are uniformly rectifiable
\cite{DSQuasi}.  We will state an abbreviated version of their
results; their results also apply to sets that are quasiminimizers
with respect to deformations inside some set $U$, but we will take
$U=\R^N$ throughout.

\begin{defn}\label{def:quasimin}
  Let $0<d<N$ be an integer.  If $\phi:\R^N\to \R^N$ is a Lipschitz
  map such that $\phi(x)=x$ for all $x$ outside some compact set, let
  $W=\{x\in \R^n\mid \phi(x)\ne x\}$.  We say that $\phi$ is a
  deformation of $\R^N$ supported on the set $\supp \phi=W\cup \phi(W)$.

  If $k\ge 1$ and $0<r\le \infty$ and $S\subset \R^N$ is a nonempty closed set
  with Hausdorff dimension $d$, we say that $S$ is a
  \emph{$(k,r)$-quasiminimizer} if:
  \begin{itemize}
  \item $|S\cap B|<\infty$ for every ball $B\subset \R^N$, and
  \item if $\phi$ is a deformation supported on a set of diameter $\le
    r$ and $W$ is as above, we have
    \[|S\cap W|\le k |\phi(S\cap W)|.\]
  \end{itemize}
\end{defn}

For example, a $d$-plane in $\R^N$ is a minimal surface and thus a
$(1,\infty)$\hyp{}quasiminimizer.  The unit sphere in $\R^N$ is
not an $(k,3)$-quasiminimizer for any $k$, since the map that
collapses the sphere to the origin can be extended to a deformation
supported on the ball of radius $1+\epsilon$.  It is, however, a
$(k,r)$-quasiminimizer for sufficiently large $k$ and
sufficiently small $r$, since a deformation of a small piece of
the sphere cannot reduce its volume very much.

For any set $S\subset \R^N$ of Hausdorff dimension $d$, we define
\[S^*=\{x\in \R^N\mid |S\cap B(x,r)|>0 \text{ for all }r>0\}.\]
David and Semmes proved:
\begin{thm}[{\cite[Thm.~2.11]{DSQuasi}}]\label{thm:quasisAreUR}
  Let $S$ be a $(k,r)$-quasiminimizer.  For each $x\in
  S^*$ and each $0<R<r$, there is a uniformly rectifiable,
  Ahlfors regular set $E$ of dimension $d$ such that
  \[S^*\cap B(x,R)\subset E \subset S^*\cap B(x,2R).\]
  The uniform rectifiability constants of $E$ can be taken to depend
  only on $N$ and $k$.
\end{thm}

\begin{defn}\label{def:locallyUR}
  If a set $S\subset \R^N$ has $S=S^*$ and satisfies the conclusion of
Theorem~\ref{thm:quasisAreUR}, we say that it is \emph{locally uniformly
rectifiable}.  That is, if for every $x\in S$ and $R<r$, there is a
compact, Ahlfors regular set $E$ of dimension $d$ such that
\[S\cap B(x,R)\subset E \subset S\cap B(x,2R)\] and $E$ is uniformly
rectifiable with regularity and uniform rectifiability constants
bounded by $\epsilon$, we say that $S$ is \emph{$(r,\epsilon)$-locally
  UR}.
\end{defn}

David and Semmes proved that this definition is equivalent to a local version of the BPLI property.
\begin{lemma}[{\cite[Chap.\ 10]{DSQuasi}}]\label{lem:DS-localUR-localBPLI}
  Let $\epsilon>0$.  There is an $\epsilon'>0$ such that if $r>0$ and
  $E$ is $(r,\epsilon)$-locally UR, then $E$ is locally Ahlfors regular and locally satisfies BPLI.  That is, 
  for any $x\in E$ and $0<R<r$, 
  \[\epsilon' R^d \le |E\cap B(x,R)|\le R^d/\epsilon'\]
  and there is a $(\epsilon')^{-1}$-Lipschitz map $f:B_d(R)\to \R^n$
  such that 
  \[|f(B_d(R))\cap E \cap B(x,R)|\ge \epsilon' R^d.\]

  Conversely, for any $\epsilon'>0$, $r>0$, and $E\subset \R^N$ which satisfy the conditions above, there is an $\epsilon>0$ depending
  on $\epsilon'>0$ and $N$ such that $E$ is $(r,\epsilon)$-locally UR.
\end{lemma}
\begin{cor}\label{cor:localURunions}
  For every $\epsilon>0$, there is an $\epsilon'>0$ such that a union
  of two $(r,\epsilon)$-locally UR sets is $(r,\epsilon')$-locally UR.
\end{cor}
\begin{cor}\label{cor:rURto2rUR}
  For every $\epsilon>0$, there is an $\epsilon'>0$ depending on
  $\epsilon$ and $N$ such that if $S$ is $(r,\epsilon)$-locally UR, then 
  it is $(2r,\epsilon')$-locally UR.
\end{cor}

The definition of quasiminimizer in Def.~\ref{def:quasimin} is
slightly too strong for our purposes.  The main problem is that if $S$
is not a quasiminimizer, we know that there is a deformation $\phi$
that decreases the measure of $S$, but we have no control over $\phi$.
We thus define a slightly weaker notion.
\begin{defn}\label{def:weakQuasi}
  Let $k>0$, $r>0$, and $\epsilon>0$.  Let $S\subset \R^N$ be a
  set such that
  \begin{equation}\label{eq:finitenessCond}
    \text{$|B(x,r)\cap S|<\infty$ for every $x\in S$ and $r>0$}
  \end{equation}
  and $S=S^*$.  We say that $h:\R^N\to \R^N$ is an
  \emph{$\epsilon$-padded deformation on $W$} if $W\subset \R^N$ is a
  bounded open set such that $\supp h\subset \Core_\epsilon{W}$, where
  \[\Core_\epsilon{W}=\{x\in W\mid d(x,\partial W) \ge \epsilon \diam W\}.\]

  We say that $S$ is a \emph{$(k,\epsilon,r)$-weak quasiminimizer} if
  for every $W\subset \R^N$ with $\diam W\le r$ and
  \[|S\cap W|\ge \epsilon (\diam W)^d\]
  and every $\epsilon$-padded deformation $h$ on $W$, we have
  \begin{equation}\label{eq:weakQuasi}
    |h(S\cap W)|\ge |S\cap W|/k.
  \end{equation}
\end{defn}
Note that if $0<\epsilon'<\epsilon$, then any $(k,\epsilon',r)$-weak
quasiminimizer is a $(k,\epsilon,r)$-weak quasiminimizer.
The main difference between $(k,\epsilon,r)$-weak quasiminimizers and
$(k,r)$-quasiminimizers is that if $|S\cap (W\setminus \Core_\epsilon{W})|$ is
large, we might have
\[|h(S\cap \Core_\epsilon{W})|\le |S\cap \Core_\epsilon{W}|/k\]
but not
\[|h(S\cap W)|\le |S\cap W|/k.\]

Nevertheless, a version of Theorem~\ref{thm:quasisAreUR} holds for
small $\epsilon$.  We will follow the proof of
Theorem~\ref{thm:quasisAreUR} to show the following result:
\begin{prop}\label{prop:QMtoUR}
  For any $k>1$, there are $\epsilon,\epsilon'>0$ such that for any
  $r>0$, any $(k,\epsilon,r)$-weak quasiminimizer is
  $(r,\epsilon')$-locally UR. 
\end{prop}
David and Semmes use the quasiminimizing condition in three places in
the proof of Theorem~\ref{thm:quasisAreUR}: to ensure Ahlfors
regularity, to construct a Lipschitz map from a subset of $S$ to
$\R^d$ whose image has positive measure, and to show that the map is
in fact bilipschitz on part of $S$.  We claim that if $\epsilon$ is
sufficiently small, then in all three cases, the deformations that
they use can be chosen to be $\epsilon$-padded, perhaps with a slight
loss in the constants.

First, we prove that a weak quasiminimizer is locally Ahlfors regular.
This is proved for quasiminimizers in Lemma~4.1 of \cite{DSQuasi}.
David and Semmes use a sequence of candidate deformations with smaller
and smaller ``buffer zones'' in their proof, so we need a slightly
different argument to prove that weak quasiminimizers are locally
Ahlfors regular.  In the rest of this section, all implicit constants
will be taken to depend on $d$, $N$, and $k$.

\begin{lemma}\label{lem:quasiImpliesRegular}
  For any $k>1$, there is an $\epsilon>0$ such that for any $r>0$, any
  $(k,\epsilon,r)$-weak quasiminimizer $S$ is locally Ahlfors regular.  That is,
  \[|B(x,R)\cap S|\sim R^d\]
  for all $x\in S$ and $0<R<r$.
\end{lemma}

\begin{proof}
  Let $C=C(N)>1$ be a constant such that Lemma~\ref{lem:FFsets} is
  satisfied when $X$ is the cubical grid in $\R^N$.  Let $n>6$ be an
  integer and let $\delta=\frac{n-2}{n}$.  If $n$ is sufficiently
  large, then
  \[\delta^{-2d}<1+\frac{1}{2Ck}\]
  and there is a $j_0$ such that $\delta^{-2j_0}<2$ but
  $(1+\frac{1}{2Ck})^{j_0}>2\cdot 4^N$.  Let
  $$\epsilon=\min \{(1-\delta)(2\sqrt{N})^{-1}, C^{-1}(2n)^{-d}/32\}.$$
  We claim that for any $r$, if $S$ is a $(k,\epsilon,r)$-weak
  quasiminimizer, then $S$ is locally Ahlfors regular.

  Specifically, if $x\in \R^N$ and $R>0$, let $\Bsq(x,R)$ be the
  closed axis-aligned cube of side length $2R$ that is centered at
  $x$.  We claim that if $x\in S$ and $0<4R\sqrt{N}<r$, then
  $$|\Bsq(x,R) \cap S|\sim R^d.$$

  Our main tool is the following deformation.  Let
  $Q=\Bsq(x,R)\subset \R^N$.  If $s>0$, we define $sQ=\Bsq(x,sR)$.  By
  the choice of $\epsilon$, we have $\delta Q\subset \Core_\epsilon(Q)$.  Let
  $\sigma$ be the $n\times\dots\times n$ grid in $\delta Q$, so that
  $\sigma$ divides $\delta^2Q$ into a lattice with $n-2$ cubes on each
  side.

  By Lemma~\ref{lem:FFsets} and Lemma~\ref{lem:cellularize}, there is
  a Lipschitz deformation $\phi=q\circ p$ that is supported on
  $\delta Q$ and deforms $S\cap \delta^2 Q$ into the $d$-skeleton of
  $\sigma$.  That is, $\phi(S\cap \delta^2 Q)^*$ is a union of
  $d$--cells of $\sigma$ and
  \begin{align}
    \label{eq:boundaryDeform}  |\phi(S\cap A)| & \le C |S\cap A|\\
    \label{eq:coreDeform}  |\phi(S\cap \delta^2 Q)| & \le C |S\cap \delta^2 Q|,
  \end{align}
  where $A=Q\setminus \delta^2Q$.

  In fact, since $\phi(S\cap \delta^2 Q)$ lies in $\sigma^{(d)}$, we
  have
  \begin{align*}
    |\phi(S\cap Q)| &\le C|S\cap A|+|\sigma^{(d)}|\\ 
                  &\le C|S\cap A|+2nN (2R)^{d},
  \end{align*}
  so
  $$\frac{|S\cap Q|}{k}\le C|S\cap A|+2nN (2R)^{d}.$$
  Replacing $Q$ by $\delta^{-2}Q$, we obtain
  $$\frac{|S\cap \delta^{-2} Q|}{k} \le C|S\cap (\delta^{-2}Q\setminus 
  Q)|+2nN (2\delta^{-2}R)^{d} $$
  so
  \begin{equation}\label{eq:lowerAnnular}
    \frac{|S\cap Q|}{k} \le C|S\cap (\delta^{-2}Q\setminus Q)|+2nN (4R)^{d}.
  \end{equation}
  We will use this to prove an upper bound on $|S\cap Q|$.  
  
  Suppose that $|S\cap Q| > 4 k nN (4R)^d$.  Then, applying
  \eqref{eq:lowerAnnular} to $\delta^{-2} Q$, we find that
  \begin{align*}
    |S\cap (\delta^{-2} Q\setminus Q)| &> \frac{|S\cap Q|}{2Ck},\\
    |S\cap \delta^{-2} Q|&>(1+\frac{1}{2Ck}) |S\cap Q|.
  \end{align*}
  By our choice of $n$, 
  $$|S\cap \delta^{-2} Q|> 4 k nN (4\delta^{-2}R)^d,$$
  so we can apply this repeatedly to show
  $$|S\cap \delta^{-2j} Q|>(1+\frac{1}{2Ck})^j |S\cap Q|$$
  for all $j\le j_0$.  In particular, 
  $$|S\cap 2Q|\ge |S\cap \delta^{-2j_0} Q|>(1+\frac{1}{2Ck})^{j_0} |S\cap
  Q| \ge 2\cdot 4^N |S\cap Q|.$$
  The cube $2Q$ has side length $4R$, so it can be decomposed into
  $4^N$ cubes of side length $R$, each of which intersects $Q$.  One
  of these subcubes, say, $D_1$, satisfies
  $|S\cap D_1|\ge 2 |S\cap Q|.$ Repeating this process, we can
  construct cubes $D_2,D_3,\dots$ such that
  $|S\cap D_j|\ge 2^j|S\cap Q|,$ $D_i\cap D_{i+1}\ne\emptyset,$ and
  the diameter of the $D_j$'s shrinks geometrically.  All of these
  cubes lie in $4Q$, so $|S\cap 4Q|=\infty$, but since $S$ is
  locally finite, this is a contradiction.

  To prove the lower bound, we consider the upper density of $S$.  Define
  $$\Theta^{*d}(S,x)=\mathop{\lim \sup}_{t\to 0} \frac{|S\cap
    B(x,t)|}{(2t)^d}.$$
  By Thm.~6.2 of \cite{Mattila}, if $|E|>\infty$, then
  $\Theta^{d*}(E,e)\ge 2^{-d}$ for almost every $e\in E$ with respect
  to Hausdorff $d$-measure.  In particular, the set of points $x\in S$
  such that $\Theta^{d*}(S,x)\ge 2^{-d}$ is dense in $S$.

  Let $c= C^{-1}(2n)^{-d}/16$ and let $x\in S$ be such that
  $\Theta^{d*}(S,x)\ge 2^{-d}$.  We claim that
  $|S\cap \Bsq(x,t)|\ge c t^d$ for all such $x$ and all $t<rN^{-1/2}$.

  Since $\Theta^{d*}(S,x)\ge 2^{-d}$, there is an $R_0<t$ such that
  $|S\cap \Bsq(x,R_0)|>R_0^d>16cR_0^d$. Let $i_0>0$ be the minimal
  integer such that
  $|S\cap \Bsq(x,\delta^{-i_0}R_0)|< 2 c (\delta^{-i_0}R_0)^d$.  Let
  $R=\delta^{-i_0}R_0$ and let $Q=\Bsq(x,R)$.  We claim that
  $R\ge t$.  If not, we have 
  $$\epsilon (\diam Q)^d< cR^d <|S\cap Q|< 2 cR^d$$
  and $(\diam Q)< r$, so we can apply \eqref{eq:weakQuasi}.

  Let $\phi$ be an $\epsilon$-padded deformation on $Q$ as above.  The
  image $\phi(S\cap \delta^2 Q)$ is made up of $d$--cells of $\sigma$.
  Each of these has volume at least $R^d(2n)^{-d}$.  Since
  $C|S\cap Q|<R^d(2n)^{-d}$, we have $|\phi(S\cap \delta^2 Q)|=0$.
  Since $\phi$ is $\epsilon$-padded on $Q$, we have
  $$|\phi(S\cap Q)|=|\phi(S\cap A)|\ge \frac{|S\cap Q|}{k},$$
  and by \eqref{eq:boundaryDeform},
  $|S\cap A| \ge \frac{|S\cap Q|}{Ck}$.  Therefore,
  \begin{equation}\label{eq:weakLowerInductive}
    |S\cap \delta^2Q|\le |S\cap Q|\left(1-\frac{1}{Ck}\right).
  \end{equation}

  By our choice of $n$, we have $\delta^{2d}>1-\frac{1}{2Ck}$, so 
  $$|S\cap \delta^2Q|< 2c (\delta^2R)^d.$$
  This contradicts the minimality of $i_0$, so $R>rN^{-1/2}$.
  Therefore, $|S\cap \Bsq(x,t)|\ge c t^d$ for all $t<rN^{-1/2}$ and
  all $x$ in a dense subset of $S$, so $S$ is locally Ahlfors regular.
\end{proof}

The second place that the quasiminimizing condition arises in
\cite{DSQuasi} is in the proof of Proposition~5.1 of \cite{DSQuasi},
which constructs Lipschitz maps from $S$ to $\R^d$ whose images have
positive measure.  We will show the corresponding proposition for weak
quasiminimizers:
\begin{prop}[{\cite[Prop.~5.1]{DSQuasi}}]\label{prop:DSprojection}
  For any $k>1$, there are $\epsilon,C>0$ that depend only on $k$ and
  $N$ such that if $S$ is a $(k,\epsilon,r)$-weak quasiminimizer and
  $Q$ is a cube centered on $S$ with $\diam Q<r$, then there is a
  $C$-Lipschitz map $h:\R^N\to \R^d$ such that
  \[|h(S\cap Q)|\ge C^{-1}(\diam Q)^d.\]
\end{prop}
The proof closely follows the proof of Proposition~5.1 of
\cite{DSQuasi}.
\begin{proof}
  By Lemma~\ref{lem:quasiImpliesRegular}, if $\epsilon$ is
  sufficiently small, we may assume that $S$ is locally Ahlfors
  regular with regularity constant $C_0=C_0(k)$.  Suppose that
  $Q=\Bsq(x,R)$ for some $x\in S$.  For any $0<\epsilon<1/4$, a
  pigeonhole argument implies that there is a radius $R/2<R_0<R$ such
  that if $Q_0=\Bsq(x,R_0)$, $\delta=1-\epsilon$, and
  $A=Q_0\setminus \delta^2Q_0$, then
  \[|S\cap A|\le 10\epsilon |S\cap Q|.\] 
    
  Divide $\delta Q_0$ into a grid of side length
  $\epsilon \delta R_0$.  As in the proof of
  Lemma~\ref{lem:quasiImpliesRegular}, Lemma~\ref{lem:FFsets} and
  Lemma~\ref{lem:cellularize} give us an $\epsilon$-padded deformation
  $\phi$ on $Q_0$ that pushes $S\cap \delta^2Q_0$ into the
  $d$-skeleton of the grid and satisfies
  \begin{align*}
    |\phi(S\cap A)| & \le C |S\cap A|\\
    |\phi(S\cap \delta^2 Q_0)| & \le C |S\cap \delta^2 Q_0|.
  \end{align*}
  Furthermore, since $S$ is locally Ahlfors regular, we can take $g$
  to be Lipschitz with constant depending on $k$.

  The Ahlfors regularity of $S$ gives a lower bound on $|S\cap Q_0|$,
  so \eqref{eq:weakQuasi} implies
  \[|\phi(S\cap Q_0)|\ge |S\cap Q_0|/k\gtrsim_k R^d.\]
  But
  \begin{align*}
    |\phi(S\cap Q_0)|&\le |\phi(S\cap \delta^2 Q_0)|+|\phi(S\cap A)|\\
                     &\le |\phi(S\cap \delta^2 Q_0)|+10 C \epsilon |S\cap Q|\\
                     &\lesssim_k |\phi(S\cap \delta^2 Q_0)|+\epsilon R^d.
  \end{align*}
  so if $\epsilon$ is sufficiently small, then
  $\phi(S\cap \delta^2 Q_0)$ must contain at least one full $d$-cell
  of side length $\epsilon R$.  If we compose $\phi$ with the projection
  to a plane parallel to this cell, we get a Lipschitz map
  $h:\R^N\to \R^d$ such that $|h(S\cap Q)|\ge (\epsilon R)^d$, as
  desired.
\end{proof}

Finally, David and Semmes use the quasiminimizing condition to show
that $S$ has big pieces of bilipschitz images.  They first use the
map constructed in Proposition~\ref{prop:DSprojection} to transform $S$
into a quasiminimizer $S'$ in $\R^d\times \R^N$ such that the
projection to the $\R^d$ factor has a large image, then show that a
quasiminimizer with a large projection must have a big piece of a
Lipschitz image.  

\begin{proof}[Proof of Proposition~\ref{prop:QMtoUR}]
  Let $D\subset \R^N$ be a cube centered on $S$ with $\diam D<r$.  Let
  $h:\R^N\to \R^d$ be the map constructed in
  Proposition~\ref{prop:DSprojection}, so that $\Lip(h)\lesssim 1$ and
  \[|h(S\cap D)|\gtrsim (\diam D)^d.\] Let $\gamma:\R^N\to \R^d\times
  \R^N$ be the map $\gamma(x)=(h(x),x)$ and let $S'=\gamma(S)$.  Since
  $\gamma$ is bilipschitz on $S$, it follows that $S'$ is Ahlfors
  regular, and in fact, $S'$ is a quasiminimizer.  Let $p_1$ and $p_2$
  be the projections to $\R^d$ and $\R^N$, respectively.

  David and Semmes show that $S\cap D$ has a big piece of a Lipschitz
  image.  Specifically, there is a $T\subset S\cap D$ such that
  $p_1\circ \gamma|_{T}$ is bilipschitz and $|T|\sim |S\cap D|$.  To
  prove this, they suppose that $S'$ is Ahlfors regular and has a big
  projection but no such $T$ exists, then construct a deformation
  $\phi$ that reduces the volume of $S'$; this contradicts the fact
  that $S'$ is a quasiminimizer.  

  In fact, $\phi$ can be chosen to be an $\epsilon$-padded
  deformation.  Suppose that $S$ and $S'$ are as above and that
  no such $T$ exists.  For any $U\subset \R^N$ and $r>0$, let
  \[B(U,r)=\{x\in \R^n\mid d(x,U)\le r\}.\] In Prop.~9.6 and Sec.~9.2
  of \cite{DSQuasi}, David and Semmes show that there is a $C\sim 1$
  such that for any sufficiently large integer $K$ (in \cite{DSQuasi},
  this is denoted $N$), there are a cube $P_0\subset \R^d$, a ball
  $V_0\subset \R^{N}$, and a set $Q=P_0\times V_0$ with $\diam Q\le
  r/2$ such that $S'\cap Q$ is (very roughly) close to a strict subset
  of the graph of a function $P_0\to V_0$.  
  Consequently, there is a deformation $\phi$ that shrinks $S'\cap Q$
  substantially.  

  To be specific, $P_0$, $V_0$, and $\phi$ satisfy the properties
  below.  As in
  \cite{DSQuasi}, we will rescale distances so that $P_0$ is a cube of
  side length $2K$.  (All references are to \cite{DSQuasi}, and all
  implicit constants depend only on $d$, $N$, and $k$.)
  \begin{enumerate}
  \item \label{it:diam} $\diam V_0\sim \diam Q\sim K$ (Lemma 9.74 of \cite{DSQuasi}).
  \item \label{it:endsClear}
    $d(P_0\times \partial V_0, S')\gtrsim K$ (Lemma 9.74).
  \item \label{it:containsBall}
    $Q$ contains a ball of radius
    $\sim K$ centered on a point of $S'$, so $|S'\cap Q|\gtrsim K^d$ (9.63).
  \item \label{it:support}
    There is a $r_0\sim 1$ such that $\supp \phi\subset B(S'\cap
    Q,r_0)$ (9.10).
  \item\label{it:unitCubes}
    If $D$ is a unit cube in $\R^d$ such that $D\subset
    B(P_0,K/C)$, then $|S'\cap (D \times V_0)|\lesssim 1$
    (Lemma~9.84).  
    
    If $1<c\ll \frac{K}{Cr_0}$, then the region
    $B(Q,cr_0) \setminus Q$ can be broken up into two regions, one
    in a neighborhood of $P_0\times \partial V_0$ and one in a
    neighborhood of $\partial P_0\times V_0$.  The second region can
    be covered by products of the form $D \times V_0$, so by
    property \ref{it:endsClear}, we have
    \[|S'\cap (B(Q,cr_0) \setminus Q)|\lesssim cK^{d-1}  \text{\quad (see 9.102)}.\]
  \item \label{it:Lipschitz}$\phi$ is $C$-Lipschitz on $S' \setminus Q$
    
    In fact, (9.13) states that $\phi$ is $C$-Lipschitz except on
    $P_0\times \R^{n-d}$, and \itref{it:endsClear} implies that
    $(S' \setminus Q)\cap \supp \phi$ is disjoint from
    $P_0\times \R^{n-d}$.
  \item\label{it:collapsing}
    $\phi(S'\cap Q)$ has finite $(d-1)$-dimensional Hausdorff
    measure and thus $|\phi(S'\cap Q)|=0$ (9.12).
  \end{enumerate}
  
  Let $c\ge 2$.  Then $\phi$ is supported on $B(Q,r_0)$ and is $\sim
  1/K$-padded on $B(Q,c r_0)$.  Furthermore, by properties
  \ref{it:unitCubes}, \ref{it:Lipschitz}, and \ref{it:collapsing},
  \begin{equation}\label{eq:graphQuasi}
    |\phi(S'\cap B(Q,cr_0))|\le C^d |S'\cap (B(Q,cr_0)\setminus Q)|\lesssim
    c K^{d-1}.
  \end{equation}
  when $K$ is sufficiently large.  By \itref{it:containsBall} above, 
  $|S'\cap Q|\gtrsim K^d$, so
  \[|\phi(S'\cap B(Q,cr_0))|\lesssim c |S'\cap Q|/K.\]
  Taking $K\gg c k$, we see that if $S'$ has a big projection and is
  Ahlfors regular but not uniformly rectifiable, then $S'$ is not a
  $(k,\sim\!\! 1/K,r)$-weak quasiminimizer.  

  Now we use $\phi$ to construct a padded deformation of $S$.  Let
  $\phi':\R^N\to \R^N$ be the map $\phi'=p_2\circ \phi \circ \gamma$,
  let $U=p_2(S'\cap Q)$, and let $W=B(U,2r_0)$.  We claim that there
  are $\epsilon,K>0$ depending on $d$, $N$, and $k$ such that $\phi'$
  is $\epsilon$-padded on $W$ and $|\phi'(S\cap W)|< |S\cap W|/k.$

  First, we claim that $\supp \phi'\subset B(U,r_0)$.  Suppose that
  $x\in \R^N$ and $\phi'(x)\ne x$.  Then, by \itref{it:support}, we
  have
  \[\{\gamma(x), \phi(\gamma(x))\}\subset B(S'\cap Q,r_0).\] If we
  project to $\R^N$, we get
  \[\{x,\phi'(x)\}=p_2(\{\gamma(x),\phi(\gamma(x))\})\subset
  B(U,r_0).\] Therefore, $\supp \phi' \subset B(U,r_0)$, and
  $\phi'$ is $\sim 1/K$-padded on $W$.

  We thus consider $|\phi'(S\cap W)|$ and $|S\cap W|$.  If $x\in
  S'\cap Q$, then $p_2(x)\in S\cap W$ and $\gamma(p_2(x))=x$, so
  $S'\cap Q\subset \gamma(S\cap W)$.  Since $\gamma$ is a bilipschitz
  map,
  \[|S\cap W|\gtrsim |S'\cap Q| \gtrsim K^d.\]
  On the other hand,
  \[\gamma(W) \subset B(\gamma(U),2r_0\Lip(\gamma))\subset B(Q,2r_0\Lip(\gamma)),\]
  and by \eqref{eq:graphQuasi},
  \begin{align*}
    |\phi'(S\cap W)|&\lesssim |\phi(\gamma(S\cap W))| \\ 
    &\lesssim |\phi(S'\cap B(Q,2r_0\Lip(\gamma))| \\
    &\lesssim 2\Lip(\gamma)K^{d-1}\\ 
    &\lesssim |S\cap W|/K.
  \end{align*}
  If $K$ is sufficiently large and $\epsilon$ is sufficiently
  small, this implies that $S$ is not a $(k,\epsilon,r)$-weak quasiminimizer.
  Therefore, if $S$ is a weak quasiminimizer, then it is locally
  uniformly rectifiable, as desired.
\end{proof}

\subsection{Proof of Theorem~\ref{thm:URdecomp}}\label{sec:pfURdecomp}
In this section, we will prove that a cellular cycle $A\in \Ccell_d(\tau)$
can be decomposed into a sum of finitely many cellular cycles $M_i$
supported on uniformly rectifiable sets $E_i$.  We restrict the
proposition to cellular cycles to avoid infinite sums.  It is
possible that a similar proposition holds for Lipschitz chains or
Lipschitz currents, but a Lipschitz chain or current might need to be
decomposed into infinitely many pieces.

A key tool in our construction is the following coarse version of the Whitney decomposition.
\begin{lemma}\label{lem:Whitney}  
  Suppose that $W\subset \R^N$ is an open subset and let 
  \[w(x)=\max\{\sqrt{N},d(x,\R^N\setminus W)\}.\]
  There is a decomposition $\tau_W$ of $\R^N$ into a cell complex such
  that: 
  \begin{enumerate}
  \item Each $N$-cell $D$ of $\tau_W$ is a dyadic cube of side length
    $\ge 1$.
  \item If $D\in \tau_W^{(N)}$
    is a dyadic cube, then
    \[\frac{w(x)}{4} \le \diam D \le w(x)\]
    for all $x\in D$.  In particular, if $D$ and $D'$ are neighboring
    dyadic cubes in $\tau_W^{(N)}$, then
    \[(\diam D)/4\le \diam D'\le 4\diam D.\]
  \end{enumerate}
  We call $\tau_W$ a \emph{coarse Whitney cubulation} of $\R^N$.  In
  particular, $\tau$ is a coarse Whitney cubulation for the empty set.
\end{lemma}
\begin{proof}
  We construct $\tau_W$ from the Whitney decomposition $\tau_0$ of
  $W$.  This decomposition is a partition of $\R^N\setminus W$ into
  dyadic cubes (of all sizes) that intersect only on their boundaries
  and satisfy the property that
  \[\frac{d(D, \R^N\setminus W)}{4} \le \diam D \le d(D, \R^N\setminus W).\]
  Let $T$ be the set of cubes of $\tau_0$ of side length $\ge1$.  We
   partition the complement of $T$ into a set $T'$ of unit dyadic
  cubes. If $\tau_W$ is the cubulation whose set of top-dimensional
  cells is $T\cup T'$, then $\tau_W$ satisfies the conditions of the
  lemma.
\end{proof}

If a set $S$ is not uniformly rectifiable, the results of the previous
section imply that there is a deformation of $S$ that reduces its
measure.  We combine that deformation with an approximation to produce
a uniformly rectifiable set.

\begin{lemma} \label{lem:URapprox}
  Let $0<\delta<1$ and let $\rho>0$.  Suppose that $S$ is a
  $(\rho,\delta)$-locally UR set.  Let $W\subset \R^N$ be a bounded
  open set and let $h:\R^N\to \R^N$ be a $\delta$-padded deformation
  on $W$.  Let $\phi:\R^N\to \R^N$ be a map that deforms $h(S)$ into
  $\tau_W^{(d)}$, as in Lemmas~\ref{lem:FFsets} and \ref{lem:cellularize},
  so that $U=\phi(h(S))^*$ is a union of $d$-cells of $\tau_W$.  There is
  an $\epsilon=\epsilon(\delta,N,d)>0$ such that $E=S\cup U$
  is $(\rho,\epsilon)$-locally UR.
\end{lemma}
\begin{proof}
  First, we claim that there is a $c>0$ depending on $N$ and $\delta$
  such that for all $x\in E$,
  $$d(x,S)\le c w(x).$$
  It suffices to show this for all $x\in U$.  If
  $x\in U$, then $x$ is contained in a cell $D\subset \tau$
  that intersects $h(S)$.  Let $s\in S$ be such that $h(s)\in D$.  

  If $h(s)\in \Core_\delta(W)$, then $s\in \Core_\delta(W)$, so
  $d(x,s)\le \diam W$.  By Lemma~\ref{lem:Whitney},
  $w(x)\sim w(h(s))\ge \delta \diam W$, so
  $d(x,s)\lesssim \delta^{-1} w(x)$ as desired.

  If $h(s)\not \in \Core_\delta(W)$, then $h(s)=s$, so
  $d(x,s)\le \diam D\lesssim w(x)$.

  Now, suppose that $x\in E$ and $R<\rho$ and let $B=B(x,R)$.  To show
  the uniform rectifiability of $E$, we need to show two things: that
  $B\cap E$ contains a Lipschitz image with volume on the order of
  $R^d$ and that $|B\cap E|\sim_\delta R^d$.  (For the rest of the
  proof, implicit constants will depend on $d$, $N$, and $\delta$.)

  First, we show that $B\cap E$ contains a Lipschitz image.  This will
  also imply that $|B\cap E|\gtrsim R^d$.  If $x\in S$, this follows
  from the uniform rectifiability of $S$, so we consider the case that
  $x\in U$.  Let $D\subset U$ be a $d$-cell of
  $\tau_W$ that contains $x$.  If $R<2 c w(x)$, then
  $|B\cap D|$ is a Lipschitz image with volume 
  $$|B\cap D|\ge v_d\min \{(w(x)/2)^d,(R/2)^d\}\ge v_d
  \left(\frac{R}{4c}\right)^d.$$

  If $R\ge 2 c w(x)$, let $s\in S$ be such that $d(x,s)\le
  R/2$.  Then $B_{R/2}(s)\subset B$, and $B_{R/2}(s)\cap S$ contains a
  Lipschitz image with volume on the order of $R^d$.  

  It remains to show that $|B\cap E|\lesssim R^d$.  By the Ahlfors
  regularity of $S$, we have $|B\cap S|\lesssim R^d$, so we need only
  show that $|B\cap U|\lesssim R^d$.  

  For $X\subset \R^N$, let $\nbhd_W X$ be the closed set consisting of
  the union of every cell of $\tau_W$ that intersects $X$.  We write
  $U$ as a union $U=U_1\cup U_2$ where
  $U_1=U\cap \nbhd_W \Core_{\delta}(W)$ and
  $U_2=U\setminus \nbhd_W \Core_{\delta}(W)$.

  For the first set, we note that
  $U_1\subset \tau_W^{(d)}\cap \nbhd_W \Core_{\delta}(W)$.  If $D$ is
  a dyadic cube in $\tau_W$ that intersects $\Core_{\delta}(W)$, then
  the side length of $D$ is at least
  $\sigma=N^{-1/2}\delta\diam W/4\sim \diam W$.  If $\sigma'$ is the
  largest power of 2 such that $\sigma'<\sigma$, then
  $D^{(d)}\subset \tau_{\sigma'}^{(d)}$, where $\tau_{\sigma'}$ is the
  grid of side length $\sigma'$.  Therefore,
  $U_1\subset \tau_{\sigma'}^{(d)}$.  The $d$-skeleton of a cube is
  Ahlfors $d$-regular, and the number of dyadic cubes of
  $\tau_{\sigma'}$ that intersect $\nbhd_W \Core_{\delta}(W)$ is
  bounded, so $|B\cap U_1|\lesssim R^d$.

  For the second set, note that
  $\phi(\Core_{\delta}(W))\subset \nbhd_W \Core_{\delta}(W)$.  Therefore,
  $$U_2\subset \phi(h(S)\setminus \Core_{\delta}(W))^*\subset \phi(S)^*.$$
  Let $x\in U_2$.  We consider two cases.  

  If $w(x)\ge 2R$, then $w(y)\ge R$ for all $y\in B$, and each dyadic
  cube $D$ of $\tau_W$ that intersects $B$ has side length at least
  $N^{-1/2}R/4$.  Therefore, the number of cubes of $\tau_W$ that
  intersect $B$ is bounded, and $|B\cap U_2|\lesssim R^d$.  

  If $w(x)\le 2R$, then $w(y)\le 3R$ for all $y\in B$, so
  $\nbhd_W B\subset B(x,4R)$.  By Lemma~\ref{lem:FFsets}.\ref{it:localApproxSetsNbhd}, we have
  \[|U_1\cap B|\le |\phi(S)^*\cap B| \lesssim |S \cap \nbhd_W B|\lesssim R^d\].  
\end{proof}

We will use this lemma to prove Theorem~\ref{thm:URdecomp}.

\begin{proof}[Proof of Thm.\ \ref{thm:URdecomp}]
  Let $C$ be as in Lemma~\ref{lem:FFsets} and let $k=2C+2$.  Let
  $\epsilon,\epsilon'>0$ be as in Prop.~\ref{prop:QMtoUR}.  Note that
  $\epsilon$ and $\epsilon'$ depend only on $N$ and $d$.  Suppose that
  $\alpha\ne 0\in C_d(\tau;\Z_2)$.  We claim that there is a nonzero
  cycle $\alpha'\in C_d(\tau;\Z_2)$ and a uniformly rectifiable set
  $E$ such that $\supp(\alpha-\alpha')\subset E$ and
  \begin{equation}\label{eq:differenceBound}
    |\supp \alpha|-|\supp \alpha'|\gtrsim |E|.
  \end{equation}

  First, we construct $\alpha'$.  Let $S=\supp \alpha$.  Let $r$ be
  the maximal power of $2$ such that $S$ is $(r,\epsilon')$-locally UR
  (some such $r$ must exist because $0<|S|<\infty$).  Then $S$ is not
  $(2r,\epsilon')$-locally UR, so, by Prop.~\ref{prop:QMtoUR}, there
  is an $\epsilon$-padded deformation $h$ on some $W\subset \R^N$ such
  that $\diam W\le 2r$,
  \begin{equation}\label{eq:URdecompLower}
    |S\cap W|\ge \epsilon (\diam W)^d,
  \end{equation}
  and \begin{equation}\label{eq:URdecompdeform}
    |h(S\cap W)|\le |S\cap W|/k.
  \end{equation}

  By Lemma~\ref{lem:FFsets} and Lemma~\ref{lem:cellularize}, there is
  a Lipschitz deformation $\phi=q\circ p:\R^N\to \R^N$ that deforms $h(S)$ into
  $\tau_W^{(d)}$.  Then $(\phi\circ h)_\sharp(\alpha)$ is a chain
  supported on $\tau^{(d)}$ with boundary supported in $\tau^{(d-1)}$.
  Let $\alpha'\in \Ccell_d(\tau;\Z_2)$ be the cellular chain that is
  flat-equivalent to $(\phi\circ h)_\sharp(\alpha)$.  This exists by
  Lemma~\ref{lem:LipToCell}.

  The support $\supp(\alpha-\alpha')$ is contained in the set
  $E_0=S\cup \phi(h(S))^*$.  In fact, since $h$ is supported inside $W$
  and since $\tau_W$ and $\tau$ agree outside of $W$, we have
  $\supp(\alpha-\alpha')\subset E_0\cap \nbhd W$, where $\nbhd W$
  consists of the union of the cubes of $\tau$ that intersect $W$.  

  We claim that there is a uniformly rectifiable set $E$ such that
  $E_0\cap \nbhd W \subset E$ and $\diam E\sim \max \{1,r\}$.  If
  $r\le \sqrt{N}$, we let $E=\nbhd W\cap \tau^{(d)}$.  This is a union
  of boundedly many unit $d$-cubes, so it is uniformly rectifiable.
  Otherwise, if $r> \sqrt{N}$, then
  $$\diam(\nbhd W) \le \diam W+2\sqrt{N}\le 4r.$$
  Let $x\in E_0\cap \nbhd W$.  Since $E_0$ is $(r,\epsilon'')$-locally
  UR and $r\sim R$, Cor.~\ref{cor:rURto2rUR} implies that there is an
  $\epsilon'''\sim 1$ such that $E_0$ is $(4r,\epsilon''')$-locally UR.
  By Def.~\ref{def:locallyUR}, there is a uniformly rectifiable
  set $E$ such that
  \[E_0\cap \nbhd W\subset E_0\cap B(x,4r) \subset E \subset E_0\cap
    B(x,8r).\]
  In either case, the Ahlfors regularity of $E$ implies that
  $|E|\sim \max \{1,r^d\}$.

  It remains to prove \eqref{eq:differenceBound}.  By
  Lemma~\ref{lem:FFsets}.\ref{it:FFsetsBound} and Lemma~\ref{lem:cellularize}.\ref{it:cellularizeBound}, we have
  $$|\supp \alpha'| \le |\phi(h(S))| \le |h(S)|+C |h(S)\setminus \tau_W^{(d)}|.$$
  But $S$ and $h(S)$ coincide outside $W$, as do $\tau^{(d)}$ and
  $\tau_W^{(d)}$, so $h(S)\setminus \tau_W^{(d)}\subset W.$ We thus
  write
  \begin{align*}
    |\supp \alpha'|&\le |h(S)|+C |h(S)\cap W| \\ 
    &\le |S\setminus W|+(C+1)|h(S)\cap W|\\
    &\le |S\setminus W|+\frac{ (C+1) |S\cap W|}{k}\\
    &\le |S\setminus W|+\frac{|S\cap W|}{2}.
  \end{align*}
  Consequently, by \eqref{eq:URdecompLower}
  \[|\supp \alpha|-|\supp \alpha'|\ge \frac{|S\cap W|}{2}\gtrsim r^d\].

  If $r>1$, then 
  \[|E|\lesssim r^d\lesssim |\supp \alpha|-|\supp \alpha'|.\]
  If $r\le 1$, then $|\supp \alpha|-|\supp \alpha'|\ge 1$, because
  $|\supp \alpha|$ and $|\supp \alpha'|$ are both integers.  It
  follows that 
  \[|\supp \alpha|-|\supp \alpha'|\ge 1 \sim (\diam E)^d.\]
  In both cases, $\alpha'$ satisfies \eqref{eq:differenceBound}.

  Finally, to prove the theorem, we define a sequence of cellular
  cycles inductively.  Let $A_0=A$.  If we have defined $A_i$ and if
  $A_i\ne 0$, then, by applying the above argument with $\alpha=A_i$,
  we get a cycle $A_{i+1}\in C_d(\tau;\Z_2)$ and a uniformly
  rectifiable set $E_i$ such that $\supp(A_i-A_{i+1})\subset E_i$ and
  \[|\supp A_i|-|\supp A_{i+1}|\gtrsim |E_i|.\] We repeat this process
  until $A_n=0$.  This is guaranteed to happen eventually because
  $|\supp A_i|$ is a decreasing sequence of non-negative integers.  If
  $M_i=A_{i}-A_{i-1}$, then $A=\sum_i M_i$, $\supp M_i\subset E_i$,
  and
  \[\sum_i |E_i|\lesssim \sum_i |\supp A_i|-|\supp A_{i+1}| \lesssim |\supp A|=\mass A\]
  as desired.  
\end{proof} 

\section{The uniformly rectifiable case}\label{sec:URcase}

In this section, we complete the proof of
Theorem~\ref{thm:cellNObound} by proving Proposition~\ref{prop:URcase}.
All the implicit constants in this section will depend on $N$, $\nu$ and the
uniform rectifiability constant of $E$.

We will follow the outline sketched in Section~\ref{sec:URcaseSketch}.
Let $\Sigma$ be the QC complex subdividing $\R^N\times [1,\infty)$
into dyadic cubes that was constructed in Section~\ref{sec:defNot}.
Let $k$ be such that $2^{k}< \diam E\le 2^{k+1}$ and let
$\Delta=(\Delta_i)_{i=0}^k$ be a cellular cubical patchwork for $E$.
If $(\cB,\cG,\cF)$ is a coronization of $E$, then the patchwork and
the corona decomposition both correspond to partitions of
$\extend{E}=E\times [1,2^{k+1}]$.  Let
$\extend{Q}=Q\times [2^i, 2^{i+1}]$ for each $Q\in \Delta_i$, then the
$\extend{Q}$'s cover $\extend{E}$ and overlap only on their
boundaries.  Let $\extend{S}=\bigcup_{Q\in S} \extend{Q}$ for all
$S\in \cF$ and let $\partial \extend{Q}$ and $\partial \extend{S}$ be
the boundaries of $\extend{S}$ as subsets of $\extend{E}$.  Then
\[\extend{E}=\bigcup_{Q\in \cB} \extend{Q}\cup \bigcup_{S\in \cF}
\extend{S},\]
and, again, the sets in the union overlap only on their boundaries.

Let $\extend{A}=A\times [1,2^{k+1}]$.  We decompose $\extend{A}$
according to $(\cB,\cG,\cF)$.  For each pseudocube $Q\in \Delta$, let
$A_Q$ be the restriction of $A$ to $Q$ and let $\extend{A}_Q$ be the
restriction of $\extend{A}$ to $\extend{Q}$.  For each stopping-time
region $S\in \mathcal{F}$, let
\[\extend{A}_S=\sum_{Q\in S}\extend{A}_Q.\]
Then
\[\extend{A}=\sum_{Q\in \cB} \extend{A}_Q+\sum_{S\in \cF}\extend{A}_S\]
and 
\begin{align*}
  \partial\extend{A}&=\sum_{Q\in \cB} \extend{A}_Q+\sum_{S\in \cF}\extend{A}_S\\
  A\times[2^k]-A\times[1]&=\sum_{Q\in \cB} \extend{A}_Q+\sum_{S\in \cF}\extend{A}_S.
\end{align*}
We will approximate the terms in this equation by cellular chains in
$\Sigma$ to obtain the following lemma.  If $Q\in \Delta_i$, let
$s(Q)=2^i$; this is the ``side length'' of $Q$.  If
$W\subset \R^N\times[1,\infty)$, let $\nbhd W$ be the union of the
(closed) cells of $\Sigma$ that intersect $W$ and let
$\nbhd^k W=\nbhd \dots \nbhd W$ be the $k$-times iterated neighborhood
of $W$.

\begin{lemma}\label{lem:decompLemma}
  For any sufficiently small $\eta,\theta>0$, if $(\cB,\cG,\cF)$ is a
  coronization $(\cB,\cG,\cF)$ satisfying
  Definition~\ref{def:coronaDecomp}, there are:
  \begin{itemize}
  \item $C>0$ depending only on $N$, $\nu$, and the uniform rectifiability
    constant of $E$,
  \item a locally finite set of chains
    $\mathcal{T}\subset \CLip_*(\R^N\times[1,\infty))$ with a
    corresponding deformation operator $P$ as in
    Lemma~\ref{lem:simultFF}, and
  \item chains $D_0,D_Q,D_S\in \Ccell(\Sigma;\Z_\nu)$
  \end{itemize}
  such that
  \begin{enumerate}
  \item $A\times [1] = D_0 - \sum_{Q\in \cB} D_Q- \sum_{S\in \cF} D_S$.
  \item $\|D_0\|_1\le C$ and $\supp D_0\subset \R^N\times 2^{k+1}$.
  \item For each bad cube $Q\in \cB$, $\|D_Q\|_1\le C$ and
    $\supp D_0\subset \R^N\times [s(Q),2s(Q)]$.
  \item For any cell $\sigma\in \Sigma$, no more than $n$ elements of
    $\mathcal{T}$ intersect $\sigma$.  
  \item For each stopping-time region $S\in \cF$, let $\Gamma(S)$ be
    the corresponding Lipschitz graph and let
    $\extend{\Gamma}(S)=\Gamma(S)\times [1,\infty)$.  There is a
    $(d+1)$-chain $G_S\in \CLip_{d+1}(\extend{\Gamma}(S);\Z)$
    such that:
    \begin{itemize}
    \item $D_S\equiv P(\partial G_S)\pmod{\nu}$.
    \item $\supp \partial G_S\subset \nbhd^2 \partial \extend{S}$.
    \item The density of $G_S$ is bounded above.  That is,
      $\mass_{B_r(x)} G_S\lesssim r^{d+1}$ for all
      $x\in \R^N\times[1,\infty)$, $r>0$, where $\mass_{B_r(x)} G_S$ is
      as in \eqref{eq:defMassRestrict}.
    \end{itemize}
  \end{enumerate}
\end{lemma}  
We will prove the lemma in
Section~\ref{sec:decompLemmaProof}.

The nonorientability of the $D_0$ or $D_Q$ terms in the lemma is easy
to bound:
\begin{lemma}
  If $D_0$ is as in Lemma~\ref{lem:decompLemma}, then
  $\NO(D_0)\lesssim 2^{kd}$.  If $Q\in \cB$ and $D_Q$ is as in the
  lemma, then $\NO(D_Q)\lesssim s(Q)^d$.
\end{lemma}
\begin{proof}
  The cellulation $\Sigma$ divides $\R^N\times \{2^{k+1}\}$ into a grid of
  side length $2^{k}$, and since
  $D_0\in \Ccell_{d}(\R^N\times \{2^{k+1}\};\Z_\nu)$ is a cycle, it is the
  boundary of some $M\in \Ccell_{d+1}(\R^N\times \{2^{k+1}\};\Z_\nu)$.  By
  the isoperimetric inequality, we can choose $M$ such that $\|M\|_1$
  is bounded.  

  Let $M_\Z\in \Ccell_{d+1}(\R^N\times \{2^{k+1}\})$ be a chain with
  integer coefficients such that $M\equiv M_\Z\pmod{\nu}$ and
  $\|M_\Z\|_1=\|M\|_1$.  Since $\partial M_\Z\equiv \partial M=D_0$,
  the cycle $\partial M_\Z$ is a pseudo-orientation of $D_0$.
  The cells that make up $M_\Z$ are cubes with side length $2^{k}$, so
  $$\NO(D_0)\le \mass \partial M_\Z \lesssim 2^{kd}\|M_\Z\|_1\lesssim 2^{kd}.$$
  
  Similarly, if $D_Q$ is as in the lemma, then $D_Q=\partial M$ for
  some $M\in \Ccell_{d+1}(\R^N\times [s(Q),2s(Q)];\Z_\nu)$ such that
  $\|M\|_1\lesssim 1$.  If $M_\Z$ is a lift of $M$ to a chain with
  integer coefficients, then $M_\Z$ is a sum of cubes with side length
  between $s(Q)$ and $2s(Q)$, and 
  $$\NO(D_Q)\le \mass \partial M_\Z \lesssim
  s(Q)^d\|M_\Z\|_1\lesssim s(Q)^d.$$
\end{proof}

The nonorientability of $D_S$ is a little harder to bound.  Since
$D_S\equiv P(\partial G_S)\pmod{\nu}$, the cycle $P(\partial G_S)$ is
a pseudo-orientation of $D_S$, and we will prove the following lemma
in Section~\ref{sec:DSNO}:
\begin{lemma}\label{lem:DSNO}
  If $S\in \cF$ and $D_S$ and $G_S$ are as in
  Lemma~\ref{lem:decompLemma}, then 
  $$\NO(D_S)\le \mass P(\partial G_S)\lesssim s(Q(S))^d.$$
\end{lemma}

Given these lemmas, the proposition follows from the packing condition
on $(\cB,\cG,\cF)$.
\begin{proof}[Proof of Proposition~\ref{prop:URcase}]
  Let $\eta$ and $\theta$ be sufficiently small that the theorem
  holds.  By Theorem~\ref{thm:coronaUR}, there is a coronization
  $(\cB,\cG,\cF)$ satisfying Definition~\ref{def:coronaDecomp}, and
  the packing constants of the coronization depend only on the UR
  constant of $E$.  That is, $\sum_{Q\in \cB} |Q|\lesssim |E|$ and
  $\sum_{S\in \cF} |Q(S)|\lesssim |E|$.

  By Lemma~\ref{lem:decompLemma} and the lemmas above,
  \begin{align*}
    \NO(A)  &\le \NO(D_0) +\sum_{Q\in \cB} \NO(D_Q)+\sum_{S\in \cF}
              \NO(D_S)\\
            &\lesssim 2^{dk}+\sum_{Q\in \cB} s(Q)^d+\sum_{S\in \cF} s(Q(S))^d\\
            &\sim |E|+\sum_{Q\in \cB} |Q|+\sum_{S\in \cF} |Q(S)|\\
            &\lesssim |E|.
  \end{align*}
\end{proof}

\subsection{Preliminaries}

The proof of Proposition~\ref{prop:URcase} will use some lemmas about
coverings of pseudocubes and stopping-time regions.  We collect these
lemmas here.  We assume throughout that $\Delta$ is a cellular cubical
patchwork and that $(\cF,\cG,\cB)$ is a coronization with implicit
constants bounded by the UR constant of $E$ and the ambient dimension
$N$.

\begin{lemma}\label{lem:pseudFiniteness}
  For any $k>0$, the sets $\nbhd^k \extend{Q}$ and
  $\nbhd^k \extend{S}$ have multiplicity bounded by a function of the
  UR constant of $E$ as $Q$ ranges over $\Delta$ and $S$ ranges over
  $\cF$.  Any cell $\sigma\subset \Sigma$ intersects only boundedly
  many such sets.

  Conversely, for any $\delta>0$ and any $Q\in \Delta$, the set
  $\nbhd^k \extend{Q}$ intersects only boundedly many cells of
  $\Sigma$.
\end{lemma}
\begin{proof}
  First, we claim that the $\nbhd \extend{Q}$'s have bounded
  multiplicity.  Let $\sigma\subset \Sigma$ be a top-dimensional cell
  of $\Sigma$ with side length $2^i$.  If $\extend{Q}$ intersects
  $\sigma$, then $Q\in \Delta_{i-1}\cup \Delta_i\cup \Delta_{i+1}$ and
  $Q$ intersects the projection of $\sigma$ to $\R^N$.  This
  projection is a cube with side length $2^i$, and the number of
  pseudocubes in $\Delta_i$ (resp.\ $\Delta_{i-1}$, $\Delta_{i+1}$)
  that intersect such a cube is bounded in terms of the patchwork
  constants of $\Delta$.  

  Since $\Sigma$ has bounded degree, the sets $\nbhd^k\extend{Q}$ also
  have bounded multiplicity.  The sets $\extend{S}$ are unions of the
  $\extend{Q}$, so the sets $\nbhd^k\extend{S}$ also have bounded
  multiplicity. 

  Finally, if $Q\in \Delta$, then $\extend{Q}$ is a subset of
  $\R^N\times [2^i,2^{i+1}]$ with $\diam \extend{Q}\sim s(Q)$, so it
  intersects only boundedly many cells of $\Sigma$.  It follows that
  $\nbhd \extend{Q}$ intersects only boundedly many cells of $\Sigma$.
  Since $\Sigma$ has bounded degree, $\nbhd^k \extend{Q}$ also
  intersects only boundedly many cells of $\Sigma$.
\end{proof}

\begin{lemma}
  If $(x,t)\in \R^N\times [1,\infty)$, then $B_{t/4}(x)\subset \nbhd^2 (x,t)$.
\end{lemma}
\begin{proof}
  The set $\nbhd (x,t)$ contains a dyadic cube $\sigma$ of side length
  $2^i$ such that $x\in \sigma$ and $t\le 2^{i+1}$.  The set
  $\nbhd \sigma$ contains all the neighbors of $\sigma$, so it
  contains every $y$ such that $d(y,\sigma)\le 2^{i-1}$.  It follows
  that $B_{t/4}(x)\subset \nbhd \sigma\subset \nbhd^2 (x,t)$.
\end{proof}

For the last lemma, we define the \emph{$r$-covering number} of a
space $U$, denoted $\cov_r(U)$, to be the minimum number of closed
balls of radius $r$ necessary to cover $U$.  Note that any $2r$-ball
can be covered by $\sim 1$ balls of radius $r$, so
\[\cov_r(U)\sim \cov_{2r}(U).\]
Furthermore, coverings of $U_1$ and $U_2$ can be combined to get a covering of
$U_1\times U_2$, so
\begin{equation}\label{eq:covProduct}
  \cov_r(U_1\times U_2)\lesssim \cov_r(U_1)\cov_r(U_2).
\end{equation}
For any subset $U\subset \R^N\times [1,\infty)$ and any $0<\delta<1$,
let
\begin{equation}\label{eq:Ndelta}
  N_\delta(U)=\bigcup_{(x,t)\in U} B((x,t),\delta t).
\end{equation}
\begin{lemma}\label{lem:smallNeighborhoods}
  If $Q\in \Delta$ and $C'>1$ is the constant in
  \eqref{eq:PatchworkSmallBdry}, then for all $\delta\in (0,1)$, we have
  \begin{equation}\label{eq:smallNeighborhoods}
    \HC^{d+1}(N_\delta(\partial\extend{Q}))\lesssim \delta^{1/C'}s(Q)^{d+1}.
  \end{equation}
\end{lemma}
\begin{proof}
  Let $Q\in \Delta_i$.  We write $\partial \extend{Q}=U_1\cup U_2,$
  where
  \begin{align*}
    U_1&=\partial Q \times [2^i,2^{i+1}]\\
    U_2&=Q\times \{2^{i},2^{i+1}\}.
  \end{align*}
  Since $U_j\subset \R^N\times [2^i,2^{i+1}]$, we can construct a
  covering of $N_\delta(U_j)$ by covering $U_j$ by balls of radius
  $\delta 2^{i+1}$, then doubling the radius of each ball.  That is, 
  \[\cov_{\delta 2^i}(N_\delta(U_j))\sim \cov_{\delta 2^i}(U_j).\]
  
  By Lemma~\ref{lem:pseudocubeBoundary},
  $\cov_{\delta 2^i}(U_1) \lesssim \delta^{1/C'-d}$, so by
  \eqref{eq:covProduct}, 
  \[\cov_{\delta 2^i}(U_1) \lesssim \delta^{1/C'-d}\cdot \delta^{-1},\]
  and
  \begin{align*}
    \HC^{d+1}(N_\delta(U_1)) &\lesssim (\delta 2^i)^{d+1} \cdot \delta^{1/C'-d-1}\\
                             &\lesssim \delta^{1/C'} 2^{i(d+1)}.
  \end{align*}

  The bound on $U_2$ follows similarly.  By the Ahlfors
  regularity of $E$, we have
  $$\cov_{\delta 2^i}(U_2)\sim \cov_{\delta 2^i}(Q)\lesssim
  \delta^{-d},$$
  so
  \begin{align*}
    \HC^{d+1}(N_\delta(U_2)) &\lesssim (\delta 2^i)^{d+1} \delta^{-d}\\
    &\lesssim \delta 2^{i(d+1)}\le \delta^{1/C'}2^{i(d+1)}.
  \end{align*}
  This proves the desired bound.
\end{proof}

\subsection{Proof of Lemma~\ref{lem:decompLemma}}
\label{sec:decompLemmaProof}
Let $0<\eta,\theta<1$ be small constants to be chosen later and let
$(\cB,\cG,\cF)$ be a cellular corona decomposition of $E$, based on
$\Delta$, with constants $\eta$ and $\theta$.  

First, we construct $P$ and the chains $D_\cdot$'s.  The deformation
operator $P$ will approximate a locally finite set of chains
$\mathcal{T}\subset \CLip_*(\R^N\times[1,\infty))$ that we will
construct in the course of the proof.  Specifically, $\cT$ will
consist of $A\times [1],A\times [2^{k+1}]$, $\extend{A}_Q$ and
$\partial \extend{A}_Q$ for all $Q\in \Delta$, $\extend{A}_S$ and
$\partial \extend{A}_S$ for all $S\in \cF$, and eight auxiliary chains
for each $S$, consisting of chains $G_S$, $G^\nu_S$, $W_S$, $W'_S$ and
their boundaries.  To avoid circularity, none of these chains will depend
on the choice of $P$.  Their supports will all lie in
$\nbhd^2 \extend{S}=\nbhd \nbhd \extend{S}$, so by
Lemma~\ref{lem:pseudFiniteness}, the multiplicity of $\cT$ is bounded.

Let $P$ be a deformation operator approximating $\cT$.  Since the
multiplicity of $\cT$ is bounded by a constant depending on dimension
and the UR constant of $E$, we can choose $C$ sufficiently large that
Lemma~\ref{lem:simultFF} holds with constant $C$.

Let $D_0=P(A\times [2^{k+1}])$, $D_Q=\partial P(\extend{A}_Q)$ for all
$Q\in \Delta$ and $D_S=\partial P(\extend{A}_S)$ for all $S\in \cF$.
Then
\begin{align*}
  P(\partial \extend{A})
  &=P(\sum_{Q\in \cB} \partial\extend{A}_Q+\sum_{S\in \cF} \partial
    \extend{A}_S)\\
  D_0-A\times [1] &= \sum_{Q\in \cB} D_Q+ \sum_{S\in \cF} D_S\\
  A\times [1] &= D_0 - \sum_{Q\in \cB} D_Q-\sum_{S\in \cF} D_S
\end{align*}
as desired.

The desired properties of $D_0$ and the $D_Q$'s follow directly.
Since $D_0$ is a chain in $\R^N\times 2^{k+1}$, it is a sum of
$d$-cells of volume $2^{kd}$, and we have
\begin{align*}
  \|D_0\|_1&\lesssim \frac{\mass D_0}{2^{kd}}\\ 
           &\lesssim \frac{\mass(A\times [2^{k+1}])}{2^{kd}}\\
           &\lesssim 1.
\end{align*}
Likewise, $P(\extend{A}_Q)$ approximates a $(d+1)$-chain in
$\R^N\times [s(Q),2s(Q)]$, so it is supported in
$\R^N\times [s(Q),2s(Q)]$ and is a sum of cells of volume at least
$(s(Q)/2)^{d+1}$.  Thus
\begin{align*}
  \|D_Q\|_1&\lesssim \|P(\extend{A}_Q)\|_1\\
           &\lesssim \frac{\mass \extend{A}_Q}{s(Q)^{d+1}}\lesssim 1.
\end{align*}

Finally, we prove that $D_S$ satisfies the desired properties.  Let
$S\in \cF$ be a stopping-time region and let
$\Gamma=\Gamma(S)\subset \R^N$ be the corresponding Lipschitz graph.
Let $\extend{\Gamma}=\Gamma\times[1,\infty)$.  Let $V\subset \R^N$ and
$h:V\to V^\perp$ be the subspace and function such that
$\Gamma=\{v+h(v)\mid v\in V\}$, and let $f:\R^N\to \Gamma$ be the
projection $f(v+w)=v+h(v)$ for all $v\in V$ and $w\in V^\perp$.

We will first show that $D_S$ satisfies a mod-$\nu$ version of the desired
property, then replace the mod-$\nu$ chain with an integral one.  
\begin{lemma}  
  If $\theta$ is sufficiently small and
  \[G^\nu_S=\extend{f}_\sharp(\extend{A}_S)\in\CLip_{d+1}(\extend{\Gamma};\Z_\nu),\]
  then $D_S= P(\partial G^\nu_S)$ and
  $\supp \partial G^\nu_S\subset \nbhd^2 \partial \extend{S}$.
\end{lemma}
\begin{proof}
  Because $S$ is a stopping-time region and $\Gamma$ is the graph of
  an $\eta$-Lipschitz function, Definition~\ref{def:coronaDecomp}
  implies
  \begin{equation}\label{eq:extendFdisplace}
    d((x,t),\extend{f}(x,t))=d(x,f(x))\le 2 d(x,\Gamma)\le 2\theta
    t.
  \end{equation}
  for all $(x,t)\in \extend{S}$.  If $N_\delta(U)$ is as in
  \eqref{eq:covProduct} and if $\theta<\frac{1}{4}$, then
  $$\extend{f}(x,t)\in N_{2\theta}(\partial \extend{S})\subset \nbhd^2
  (x,t).$$

  It follows that $\supp G^\nu_S\subset \nbhd^2 \extend{S}$, so we may
  add $G^\nu_S$ and $\partial G^\nu_S$ to $\cT$ without affecting its
  bounded multiplicity.  In fact, because $A$ is a cycle, we have
  $\supp \partial \extend{A}_S\subset \partial \extend{S}$, and
  $$\supp \partial G^\nu_S\subset f(\partial \extend{S})\subset \nbhd^2 \partial \extend{S}.$$

  Let $W_S\in \CLip_{d+1}(\R^N\times[1,\infty);\Z_\nu)$ be the
  straight-line homotopy between $\partial \extend{A}_S$ and
  $\partial{G^\nu_S}=\extend{f}_\sharp(\partial \extend{A}_S)$.  As
  above, $\supp W_S \subset \nbhd^2 \extend{S}$, so adding $W_S$ and
  $\partial W_S$ to $\cT$ does not affect the bounded multiplicity of
  $\cT$.

  We claim that if $\theta$ is sufficiently small, then $\supp W_S$
  has small Hausdorff content and $P(W_S)=0$.  If
  $\sigma\in \Sigma^{(d+1)}$ is a $(d+1)$-cell of side length $2^i$,
  then
  \begin{align*}
    \supp W_S\cap \nbhd \sigma 
    &\subset N_{2\theta}(\partial \extend{S}) \cap \nbhd \sigma \\
    & \subset \bigcup_{Q\in \Delta} N_{2\theta}(\partial\extend{Q}) \cap \nbhd \sigma 
  \end{align*}
  Since $\theta<\frac{1}{4}$, there are only boundedly many
  $Q\in \Delta$ such that $N_{2\theta}(\partial\extend{Q})$ intersects
  $\nbhd \sigma$.  All of these have $s(Q)\sim 2^i$, so by
  Lemma~\ref{lem:smallNeighborhoods},
  $$\HC^{d+1}(\supp W_S\cap \nbhd \sigma) \lesssim \theta^{1/C'}2^{i(d+1)}.$$
  If $\theta$ is sufficiently small, then Lemma~\ref{lem:simultFF}.(6)
  implies that
  \[\HC^{d+1}(\supp P(W_S)\cap \sigma)< \HC^{d+1}(\sigma),\]
  so the support of $P(W_S)$ does not contain $\sigma$.  But this
  argument applies to any $(d+1)$-cell $\sigma$, so $P(W_S)=0$!  It
  follows that
  $$P(\partial W_S)=P(\partial \extend{A}_S-\partial G^\nu_S)=0$$
  and thus that $D_S=P(\partial \extend{A}_S)=P(\partial G^\nu_S)$.
\end{proof}

Since $G^\nu_S$ is a Lipschitz $(d+1)$-chain in a $(d+1)$-dimensional
Lipschitz graph, there is an integer $(d+1)$-chain $G_S$ with nearly
the same boundary.  In fact, $G_S$ will be a cellular approximation of
$G^\nu_S$.
\begin{lemma}
  For any $\epsilon>0$, there are chains
  $G_S,W'_S\in \CLip_{d+1}(\extend{\Gamma};\Z)$ such that:
  \begin{itemize}
  \item $\partial W'_S\equiv \partial G_S-\partial G^\nu_S\pmod{\nu}$
  \item $\mass W'_S<\epsilon$
  \item $\supp W'_S \subset \nbhd^2 \partial \extend{S}$
  \item $\mass_{B_r(x)} G_S\lesssim r^{d+1}$ for all
    $x\in \R^N\times[1,\infty)$, $r>0$
  \end{itemize}
\end{lemma}
\begin{proof}
  The graph $\extend{\Gamma}$ is bilipschitz equivalent to
  $\R^{d}\times [1,\infty)$, so we may give it the structure of a QC
  complex by letting $\kappa$ be the image of a grid in
  $\R^{d}\times [1,\infty)$ of side length $\epsilon'$.  By
  Theorem~\ref{thm:WhiteDeform}, there is a cellular
  chain $P^\nu\in \Ccell_{d+1}(\kappa;\Z_\nu)$ approximating $G^\nu_S$
  and a ``homotopy'' $H^\nu\in \CLip_{d+1}(\extend{\Gamma};\Z_\nu)$
  from $\partial G^\nu_S$ to $\partial P^\nu$ satisfying
  \begin{align*}
    \partial H^\nu&=\partial P^\nu-\partial G^\nu_S\\
    \mass H^\nu&\lesssim \epsilon' \mass \partial G^\nu_S\\
    \supp H^\nu&\subset \nbhd^2 \partial G^\nu_S\subset
                 \nbhd^2\partial \extend{S}.
  \end{align*}
  Choose $\epsilon'$ sufficiently small that $\mass H^\nu<\epsilon$.

  Fix an orientation on $\extend{\Gamma}$; we will use this
  orientation to lift $P^\nu$ to a chain with integer coefficients
  that has the same boundary.  (See also
  \cite{FedererRealVariational}.)  We orient the $(d+1)$-cells of
  $\kappa$ to match the orientation of $\extend{\Gamma}$.  This fixes
  signs for all of the coefficients of $(d+1)$-chains, and we define
  $G_S \in \Ccell_{d+1}(\kappa)$ to be the unique integer chain such
  that $G_S\equiv P^\nu\pmod{\nu}$ and the coefficients of $G_S$ are
  all between $0$ and $\nu-1$.  If $\sigma$ and $\sigma'$ are
  neighboring $(d+1)$-cells in $\sigma$, then they have the same
  coefficient in $G_S$ if and only if they have the same coefficient
  in $P^\nu$, so $\supp \partial G_S=\supp \partial P^\nu$.

  Let $W'_S\in \CLip_{d+1}(\extend{\Gamma};\Z)$ be a chain such that
  $W'_S\equiv H^\nu\pmod{\nu}$, $\supp W'_S=\supp H^\nu$, and
  $\mass W'_S=\mass H^\nu$.  Then $W'_S$ satisfies the first
  three conditions of the lemma.  To prove the last condition, note
  that the coefficients of $G_S$ are bounded, so
  $$\mass_{B_r(x)} G_S\le \nu |\extend{\Gamma}\cap B_r(x)|\lesssim
  r^{d+1}.$$
\end{proof}
Finally, if $\epsilon$ is sufficiently small, then $P(W'_S)=0$, so
$$P(\partial G_S)\equiv P(\partial G^\nu_S+\partial W'_S)=P(\partial G^\nu_S )=D_S,$$
as desired.

\subsection{Proof of Lemma~\ref{lem:DSNO}}
\label{sec:DSNO}


Finally, we bound $\NO(D_S)$.  Recall that
$$G_S\in \CLip_{d+1}(\extend{\Gamma}(S))$$
is a chain with integer coefficients and an upper bound on its density
and that $D_S\in \Ccell(\Sigma;\Z_\nu)$ is congruent modulo $\nu$ to
$P(\partial G_S)$.  The cycle $P(\partial G_S)$ is a
pseudo-orientation of $D_S$, so it suffices to show that
$$\mass P(\partial G_S)\lesssim s(Q(S))^d.$$

First, we note that the coefficients of $P(G_S)$ are bounded:
\begin{lemma}\label{lem:boundedCoeffs}
  If $G\in \CLip_{d+1}(\R^N\times[1,\infty))$ is a chain such that
  $\mass_{B_r(x)} G \lesssim r^{d+1}$ for all
  $x\in \R^N\times[1,\infty)$ and $r>0$, then the coefficients of
  $P(G)$ are bounded.
\end{lemma}
\begin{proof}
  Let $\sigma$ be a $(d+1)$-cell of $\Sigma$ and let $x_\sigma$ be the
  coefficient of $P(G)$ on $\sigma$.  By Lemma~\ref{lem:simultFF}.(\ref{it:localApproxNbhd})
  and the bound on the density of $G$, we have
  \begin{align*}
    |x_\sigma|&=\frac{\mass_\sigma P(G)}{\Haus^{d+1}(\sigma)} \\ 
              &\lesssim \frac{\mass_{\nbhd \sigma} G}{\Haus^{d+1}(\sigma)}\\
              &\lesssim \frac{(\diam \sigma)^{d+1}}{\Haus^{d+1}(\sigma)}\\
              &\lesssim 1.
  \end{align*}
\end{proof}
Since $\Sigma$ has bounded degree, the coefficients of $P(\partial
G_S)=\partial P(G_S)$ are also bounded.

Next, we bound the support of $P(\partial G_S)$.  By
Lemmas~\ref{lem:decompLemma} and \ref{lem:simultFF}, we have
$$\supp P(\partial G_S)\subset \nbhd^2 \partial \extend{S}.$$
If $L$ is a subcomplex of $\Sigma$, let
\[\size_d L=\sum_{\sigma \in (\nbhd L)^{(d)}} |\sigma|\]
be the total volume of the $d$-cells of $\nbhd L$.

Then: 
\begin{lemma}\label{lem:coverIsSize}
  Suppose that $U\subset X$ and $U_i=U\cap \R^N\times [2^i,2^{i+1}]$
  for $i=0,1,2,\dots$.  Then
  \[\size_d \nbhd^2 U \sim \size_d \nbhd U \lesssim \sum_{i=0}^k 2^{id}\cov_{2^i}U_i.\]
\end{lemma}
\begin{proof}
  The subcomplex $\nbhd U$ is a union of dyadic cubes.  For every
  dyadic cube $K$, we have $\size_d K=\size_d \nbhd K$, so
  $\size_d \nbhd U\sim \nbhd^2 U$.

  For any $i$, $\nbhd U_i$ is a union of dyadic cubes with side length
  $\sim 2^i$.  A covering of $U_i$ by balls of radius $2^i$ can be
  made into a cover of $\nbhd U_i$ by increasing the radius of balls
  by a factor of $\sqrt{N}$.  Each of the expanded balls intersects
  only boundedly many cells of $\nbhd U_i$, so
  \[\size_d \nbhd U_i \lesssim 2^{id}\cov_{2^i}U_i\]
  and
  \[\size_d U \lesssim \sum_{i=0}^k 2^{id}\cov_{2^i}U_i.\]
\end{proof}

We claim that
\begin{lemma}
  If $S\in \cF$ is a stopping-time region and
  $K=\nbhd^2 \partial \extend{S}$, then
  \[\size_d K\lesssim |Q(S)|.\]  
\end{lemma}
\begin{proof}
  For every pseudocube $Q\in \Delta$, let 
  $$M_Q=\bigcup_{Q'\subset Q}\extend{Q'}$$
  be the union of $\extend{Q}$ and all of its descendants and let 
  $$M'_Q=M_Q\setminus \extend{Q}.$$
  Let $S_\text{min}$ be the set of minimal pseudocubes in $S$.  Since
  $S$ is coherent, the elements of $S_\text{min}$ partition $Q(S)$.
  That is, they are all disjoint (since any two minimal pseudocubes
  are disjoint) and their union is $Q(S)$ (since if a pseudocube in
  $S$ is non-minimal, all its children are contained in $S$.)
  Note that we are using the fact that $\Delta$ is a cellular cubical
  patchwork and thus has a bottom level.

  If $Q\subset Q(S)$, but $Q\not \in S$, then $Q$ is a descendant of
  one of the $S_{\text{min}}$.  Therefore,
  \[\extend{S}=\bigcup_{Q\in S}\extend{Q}=M_{Q(S)}\setminus
  \bigcup_{Q\in S_{\text{min}}} M'_Q,\]
  and 
  \begin{align*}
    \partial \extend{S}&\subset \partial M_{Q(S)}\cup \bigcup_{Q\in
                         S_{\text{min}}} \partial M'_Q\\
    &\subset \partial M_{Q(S)}\cup \bigcup_{Q\in
                         S'_{\text{min}}} \partial M_Q,
  \end{align*}
  where $S'_{\text{min}}$ consists of the children of elements of
  $S_{\text{min}}$.  
  
  It follows that 
  \begin{equation}\label{eq:sizeDecomp}
    \size_d \nbhd^2 \partial \extend{S}\lesssim \sum \size_d
    \nbhd^2 \partial M_{Q(S)}+ \sum_{Q\in S'_{\text{min}}} \size_d
    \nbhd^2 \partial M_Q.
  \end{equation}
  We claim that for all $Q\in \Delta$, 
  $$\size_d \nbhd^2 \partial M_Q\lesssim |Q|.$$

  Let $s(Q)=2^i$ and write
  \[\partial M_Q=Q\times 2^i\cup \partial Q \times [1,2^i].\]
  By Lemma~\ref{lem:coverIsSize}, we have
  \[\size_d \nbhd^2 \partial M_Q \lesssim 2^{id}\cov_{2^i}(Q)+\sum_{j=0}^{i-1} 2^{jd} \cov_{2^j}(\partial Q\times [2^{j},2^{j+1}]).\]
  Since $\diam Q\sim 2^i$,
  \[\cov_{2^i}(Q\times \{2^i\}) \sim 1.\]
  If $j<i$ and $C'>0$ is as in Lemma~\ref{lem:pseudocubeBoundary}, then
  \[\cov_{2^j}(\partial Q\times[2^j,2^{j+1}])\lesssim 2^{(i-j)(d-1/C')},\]
  so
  \[\size_d \nbhd^2 \partial M_Q \lesssim 2^{id} + \sum_{j=0}^i 2^{id}
  2^{-(i-j)/C'}\lesssim 2^{id}\sim |Q|.\]

  Finally, by \eqref{eq:sizeDecomp}, we have
  $$\size_d K \lesssim |Q(S)| + \sum_{Q\in
    S'_{\text{min}}}|Q|.$$
  The children of the minimal pseudocubes of $S$ are all disjoint,
  so $\size_d K \lesssim |Q(S)|.$
\end{proof}
Then, by Lemma~\ref{lem:boundedCoeffs}, we have
$$\mass P(\partial G_S)\lesssim |\supp P(\partial G_S)|\lesssim
\size_d K\lesssim |Q(S)|.$$
Since $P(\partial G_S)$ is a pseudo-orientation of
$D_S$, this completes the proof of Lemma~\ref{lem:DSNO}.

\appendix

\section{Proof of Lemma~\ref{lem:simultFF}}\label{sec:simultFFProof}

In this section, we prove Lemma~\ref{lem:simultFF}.  None of the ideas
here are original; our proof follows similar lines to the argument of
Federer and Fleming \cite{FedFlemNormInt}, White's deformation lemma
\cite{WhiteDeform}, the argument used by David and Semmes to prove
Proposition~3.1 in \cite{DSQuasi}, and the proof of a cellular version
of the Deformation Theorem in Chapter~10 of \cite{ECHLPT}.

We recall some notation.  If $D\subset \R^d$ is a measurable set,
$d\le N$, and $\alpha:D\to B$ is Lipschitz, then by the arguments in
Section~\ref{sec:defNot}, the jacobian determinant $J_\alpha$ is
defined almost everywhere in $D$.  We define
\[\vol^d \alpha=\int_{x\in D} |J_\alpha(x)|\;dx.\]
Similarly, if $\Sigma$ is a QC complex, $B\subset \Sigma$ is a Borel
set, and $A$ is a Lipschitz chain in $\Sigma$, we let $\mass_B A$ be
the mass of the restriction of $A$ to $B$.  Let $\mu$ be Lebesgue measure on $\R^N$.

If $\Sigma$ is a QC complex, then each cell of $\Sigma$ is
bilipschitz equivalent to a ball, and if $B$ is a ball, we can
construct a map $p$ that takes all but one point of $B$ to its
boundary by choosing a random point $y\in B$, then projecting
$B\setminus \{y\}$ to its boundary along straight lines.  In the
following lemma, we use Fubini's theorem to bound the average amount
that this random projection increases the mass of a chain or the
Hausdorff content or Hausdorff measure of a set.  Note that we need to
smooth the projection on a ball of radius $\epsilon$ to make it a Lipschitz map
defined on all of $B$.
\begin{lemma}\label{lem:singInt}
  Let $B=B(0,r)$ be a ball in $\R^N$ and let $\gamma B=B(0,\gamma r)$
  for any $\gamma>0$.  For any $y\in B/2$ and any unit vector $v\in
  S^{N-1}$, let $t:B\setminus \{y\}\to \partial B$ be the projection
  of $B\setminus \{y\}$ to its boundary.  That is, $t(x)$ is the
  endpoint of the ray from $y$ to $x$.  Let
  $0<\epsilon<r/2$ and let $p_{y}:B\to B$ be the map
  \[p_{y}(x)=\begin{cases}
    y+\min\{1,\frac{\rho}{\epsilon}\} (t(x)-y) & x\ne y \\
    y & x=y
  \end{cases},\]
  where $\rho=d(x,y)$.  This is a Lipschitz map that sends $B(y,\epsilon)$
  surjectively to $\inter B$ and sends $B\setminus B(y,\epsilon)$ to
  $\partial B$.

  If $D\subset \R^d$ is a measurable set, $\alpha:D\to B$ is
  Lipschitz, and $U\subset B$ is a set with $\HC^d(U)<\infty$, then
  for all $\epsilon\in(0,r/2)$, we have
  \begin{align}
    \label{eq:singInt}\frac{1}{\mu (B/2)} \int_{y\in
      B/2}\vol^d(p_y\circ \alpha)\;dy& \lesssim_{N} \vol^d \alpha.\\
    \label{eq:HCsingInt}
    \frac{1}{\mu (B/2)} \int_{y\in B/2}\HC^{d}(p_y(U))\;dy& \lesssim_{N} \HC^{d}(U)
  \end{align}
\end{lemma}
David and Semmes prove an inequality similar to \eqref{eq:HCsingInt} in Chapter~3 of \cite{DSQuasi}.
\begin{proof} 
  If $\rho=d(x,y)$ and $\|Dp_y(x)\|$ is the operator norm of the
  derivative of $p_y$ at $x$, then
  \[\|Dp_y(x)\|\lesssim \frac{r}{\max\{\rho,\epsilon\}}.\]

  To prove \eqref{eq:singInt} when $d<N$, we write
  \begin{align*}
    \int_{y\in B/2}&\vol^d (p_y\circ \alpha) \;dy
    =\int_{y\in B/2} \int_{x\in D}|J_{p_y\circ \alpha}(x)| \;dx\;dy\\
    &\lesssim \int_{y\in B/2} \int_{x\in D} \|Dp_y(\alpha(x))\|^d |J_\alpha(x)| \;dx\;dy\\
    &\lesssim \int_{y\in B/2} \int_{x\in D}\left(\frac{r}{\max\{\rho,\epsilon\}}\right)^d  |J_\alpha(x)| \;dx\;dy\\
    &=  r^d \int_{x\in D} |J_\alpha(x)| \int_{y\in B/2} \min\{\rho^{-d},\epsilon^{-d}\} \;dy\;dx
  \end{align*}
  using Fubini's Theorem in the last step.  Since
  $\min\{\rho^{-d},\epsilon^{-d}\}=\epsilon^{-d}$ only when $\rho\le \epsilon$, we can bound the last integral by
  \begin{align}
\notag    \int_{y\in B/2} \min\{\rho^{-d},\epsilon^{-d}\} \;dy&\lesssim \int_{y\in B(x,2r)} \rho^{-d}\;dy+\int_{y\in B(x,\epsilon)}\epsilon^{-d}\;dy\\
    &\label{eq:avgLip}\lesssim \int_{0}^{2r}
    \rho^{N-1-d}\;d\rho+\epsilon^{N-d} \\
\notag    & \lesssim r^{N-d}
  \end{align}
  We thus have
  \[\frac{1}{\mu (B/2)} \int_{y\in B/2}\vol^d (p_y\circ \alpha)
  \;dy\lesssim \frac{r^N}{\mu (B/2)} \int_{x\in D} |J_\alpha(x)|\;dx
  \lesssim \vol^d \alpha,\]
  so \eqref{eq:singInt} holds when $d<N$.

  If $d=N$, the integral in \eqref{eq:avgLip} diverges, so we
  need a different argument.  In this case, $p_y$ sends $B\setminus
  B(y,\epsilon)$ to $\partial B$, which is $(d-1)$-dimensional, so the
  part of $\alpha(\Delta)$ outside of $B(y,\epsilon)$ doesn't
  contribute to $\vol^d (p_y\circ \alpha)$.  Therefore, using Fubini's
  Theorem as before, 
  \begin{align*}
    \int_{y\in B/2} \vol^d (p_y\circ \alpha) \;dy 
    &\lesssim r^N \int_{x\in D} |J_\alpha(x)| \int_{y\in B(x,\epsilon)} \epsilon^{-N} \;dy\;dx\\
    &\lesssim r^N \int_{x\in D} |J_\alpha(x)| \;dx \lesssim \mu(B) \vol^d \alpha.
  \end{align*}
  This proves \eqref{eq:singInt}.

  Next, consider \eqref{eq:HCsingInt}.  It suffices to show that for
  any ball $B_0=B(x,r_0)$ with $r_0\le r$, we have
  \begin{equation}\label{eq:HCsingleBall}
    \frac{1}{\mu(B/2)}\int_{y\in B/2}\HC^{d}(p_y(B_0))\;dy\lesssim r_0^d.
  \end{equation}
  If $d<N$, then let $\rho=d(x,y)$ as before.  If $\rho>2r_0$, then
  $d(y,B_0)\gtrsim \rho$, so 
  $\Lip(p_y|_{B_0})\lesssim \frac{r}{\rho}$, and
  \[\HC^{d}(p_y(B_0))\lesssim \left(\frac{r r_0}{\rho}\right)^d.\]
  On the other hand, if $\rho\le 2r_0$, then 
  \[\HC^{d}(p_y(B_0)) \le \HC^d(B)\sim r^d.\]
  Therefore, 
  \begin{align*}
    \int_{y\in B/2}\HC^{d}(p_y(B_0))\;dy
    & \lesssim \int_{y\in B/2} \left(\frac{r r_0}{\rho}\right)^d  \;dy
    + \int_{y\in B(x,2r_0)}r^d \;dy\\
    & \lesssim r^N r_0^d+r_0^N r^d
    \le r^N r_0^d,
  \end{align*}
  using the fact proved above that $\int_{y\in B/2}
  \rho^{-d}\;dy\lesssim r^{N-d}.$

  If $d=N$, then $p_y$ sends $B\setminus
  B(y,\epsilon)$ to $\partial B$, so $\HC^{N}(p_y(B_0))=0$ unless
  $y\in B(x, 2\max\{\epsilon,r_0\})$.  Since $\Lip(p_y)\le \frac{r}{\epsilon}$,
  we know that 
  \[\diam p_y(B_0)\lesssim \min\{r,\frac{rr_0}{\epsilon}\}=\frac{rr_0}{\max\{r_0,\epsilon\}},\]
  so 
  \begin{align*}
    \int_{y\in B/2}\HC^{N}(p_y(B_0))\;dy&=\int_{y\in
      B(x,2\max\{\epsilon,r_0\})}\HC^{N}(p_y(B_0))\;dy\\     
    &\lesssim (\max\{\epsilon,r_0\})^N \left(\frac{2rr_0}{\max\{\epsilon,r_0\}}\right)^N\sim r^Nr_0^N.
  \end{align*}
  This proves \eqref{eq:HCsingleBall}, which implies
  \eqref{eq:HCsingInt}.
\end{proof}

We can use the lemma
to construct a map $\Sigma^{(k)}\to \Sigma^{(k)}$ that sends most of
the $k$-skeleton of $\Sigma$ into its $(k-1)$-skeleton.

\begin{lemma}\label{lem:expectedFF}
  Suppose that $\Sigma$ is a QC complex of dimension $N$ and $c>0$ is
  such that each cell of $\Sigma$ is $c$-bilipschitz to a ball.
  Suppose that $k\le \dim \Sigma$ and $\cT_k \subset
  \CLip_*(\Sigma^{(k)})$ is a set of Lipschitz chains of dimension
  $\le k$ which is closed under taking boundaries.  Suppose that $n>0$
  is a number such that for any $k$-cell $K\subset \Sigma$, no more
  than $n$ chains in $\cT_k$ have support that intersects the interior
  of $K$.  Then there is a locally Lipschitz map $p_k:\Sigma\to
  \Sigma$ such that
  \begin{itemize}
  \item $p_k$ fixes $\Sigma^{(k-1)}$ pointwise,
  \item $p_k(\sigma)\subset \sigma$ for each cell $\sigma\subset \Sigma$,
  \item for every $T\in \cT_k$ such that $\dim T<k$, we have
    $p_k(\supp T)\subset \Sigma^{(k-1)}$, 
  \item for each $T\in \cT_k$, 
    \begin{align}\label{eq:massGrowth}
      \mass (p_k)_\sharp(T) & \lesssim_{c,n,N} \mass T\\
\label{eq:HCgrowth}
      \HC^{d}(p_k(\supp T))& \lesssim_{c,n,N} \HC^{d}(\supp T).
    \end{align}
  \end{itemize}
\end{lemma}
\begin{proof}
  We construct $p_k$ on each $k$-cell of $\Sigma$, then extend it to
  the higher-dimensional simplices.  Suppose that $K$ is a $k$-cell
  and suppose that $T_1,\dots, T_n\in \cT_k$ are the only chains in
  $\cT_k$ whose supports intersect $K$.  Since $\Sigma$ is a QC
  complex, we may identify $K$ with a closed euclidean ball $B$ of
  radius $r$.  By Lemma~\ref{lem:singInt}, there is a subset $K_0$ of
  $K$ (corresponding to $B/2$) such that for any sufficiently small
  $\epsilon>0$, there is a family of maps $p_y:K\to K$, $y\in K_0$,
  such that $p_y$ sends $B(y,\epsilon)$ surjectively onto $K$ and
  sends $K\setminus B(y,\epsilon)$ to $\partial K$.  Furthermore,
  \begin{align*}\label{eq:singInt2}
    \frac{1}{\mu(K_0)} \int_{y\in K_0}\mass_K (p_y)_\sharp(S)\;dy&\lesssim_{c,N} \mass_K S\\
    \frac{1}{\mu(K_0)} \int_{y\in K_0}\HC^{d}(p_y(\supp S\cap K))\;dy&\lesssim_{c,N} \HC^{d}(\supp S\cap K)
  \end{align*}
  for every chain $S$ of dimension $\le k$.

  Choose $\epsilon>0$ so that the $\epsilon$-neighborhood of the
  supports of the $T_i$'s is small.  That is,
  $\mu(E_\epsilon)<\mu(K_0)/2$, where
  \[E_\epsilon=\bigcup_{\dim T_i<k} \{y\in K\mid d(y,\supp T_i)<\epsilon\}.\]
  This is possible because $\supp T_i$ is a finite union of Lipschitz
  images of simplices.

  Let 
  \begin{align*}
    F_i(\gamma)=\{y\in K_0 \mid &\mass_K (p_{y})_\sharp(T_i) > \gamma \mass_K T_i \text{ or } \\
    & \HC^{d}(p_y(\supp T_i\cap K))\;dy>\gamma \HC^{d}(\supp T_i\cap K)\}
  \end{align*}
  By Chebyshev's inequality,
  \[\mu(F_i(\gamma)) \lesssim_{c,N} \gamma^{-1}\mu(K_0).\] 
  If $\gamma$ is large enough, depending on $c$, $N$, and $n$, there is some
  $y\in K_0$ such that $y\not\in E_\epsilon$ and $y\not \in F_i(\gamma)$ for all $i$.
  Then for all $i$, we have
  \[\mass_K (p_{y})_\sharp(T_i)\le \gamma\mass_K (T_i),\]
  \[\HC^{d}(p_y(\supp T_i\cap K))\;dy\le\gamma \HC^{d}(\supp T_i\cap K).\]
  Also, $p_y$ fixes $\partial K$ pointwise, and
  $p_y(\supp(T_i))\subset \partial K$ if $\dim T_i<k$.  Let $p_k$ be
  equal to $p_y$ on $K$.

  We define $p_k$ on the $k$-skeleton of $\Sigma$ by repeating this
  process for each $k$-cell.  Then, for each cell $L\subset \Sigma$ with
  $\dim L>k$, we have defined $p_k$ on $\partial L$ so that
  $p_k|_{\partial L}$ is a Lipschitz map, so we extend $p_k$ to
  $L$ by radial extension.  The result is Lipschitz and sends $L$ to
  itself, so the resulting $p_k$ satisfies the conditions of the
  lemma.
\end{proof}

This lets us prove Lemma~\ref{lem:simultFF}.
\begin{proof}[Proof of Lemma~\ref{lem:simultFF}]
  First, we construct $p$.  Recall that $\cT$ is a set of chains which
  is closed under taking boundaries and that $n>0$ is a number such
  that for any cell $D\in \Sigma$, no more than $n$ elements of $\cT$
  intersect $D$.  
  
  We can use Lemma~\ref{lem:expectedFF} repeatedly to construct a
  sequence of Lipschitz maps $p_1,\dots, p_N:\Sigma\to \Sigma$ such that for
  each $k=1,\dots, N$,
  \begin{itemize}
  \item $p_k$ fixes $\Sigma^{(k-1)}$ pointwise,
  \item $p_k(\sigma)\subset \sigma$ for each cell $\sigma\subset \Sigma$,
  \item for every $T\in\cT$ such that $\dim T<k$, we have
    \[(p_k\circ p_{k+1}\circ \dots \circ p_N)(\supp T)\subset \Sigma^{(k-1)},\] 
    and
  \item for each $T\in \cT$, 
    \begin{align}
      \label{eq:massGrowth2}
      \mass (p_k\circ p_{k+1}\circ \dots \circ p_N)_\sharp(T) & \lesssim_{c,n,N} \mass T\\
      \label{eq:HCgrowth2}
      \HC^{d}(\supp (p_k\circ p_{k+1}\circ \dots \circ p_N)_\sharp(T))& \lesssim_{c,n,N} \HC^{d}(\supp T).
    \end{align}
  \end{itemize}
  We first construct $p_N$ by applying Lemma~\ref{lem:expectedFF} to
  $\cT_N=\cT$, then construct $p_k$ inductively by applying
  Lemma~\ref{lem:expectedFF} to
  \[\cT_k=\{(p_{k+1}\circ \dots \circ p_N)_\sharp(T)\mid T\in \cT\}.\]
  By the local finiteness condition, no more than $n$ elements of $\cT_k$
  intersect the interior of any cell of $\Sigma$, so the implicit
  constants in \eqref{eq:massGrowth} and \eqref{eq:HCgrowth} are
  uniformly bounded.  It follows that the implicit constants in
  \eqref{eq:massGrowth2} and \eqref{eq:HCgrowth2} depend only on $c,
  n$, and $N$.  Let
  \[p=p_1\circ \dots \circ p_N.\]

  Then, for any $T\in \cT$, we have $p(\supp T)\subset \Sigma^{(d)},$
  $\mass p_\sharp(T)\lesssim_{c,n,N} \mass T,$ and $\HC^{d}(\supp
  p_\sharp(T)) \lesssim_{c,n,N} \HC^{d}(\supp T)$ as desired.
  Furthermore, if $Y\subset \Sigma$, then $p^{-1}(Y)\subset \nbhd Y$,
  so it is straightforward to check part \ref{it:localApproxNbhd} of
  the lemma and
  to bound $\HC^{d}(\supp P(T)\cap Y).$
\end{proof}

\bibliographystyle{alpha}
\bibliography{TwoT}
\end{document}